\documentclass[11pt,a4paper]{article}
\usepackage{amssymb,amsmath}
\usepackage{times,theorem,latexsym,color}
\usepackage[titletoc,toc,title]{appendix}

\allowdisplaybreaks[2]

\newcommand{\BOX}{\ensuremath\Box}

\newtheorem{theorem}{Theorem }[section]
\newtheorem{assumption}[theorem]{Assumption}
\newtheorem{corollary}[theorem]{Corollary}

{\theorembodyfont{\rmfamily}}
{\theorembodyfont{\rmfamily}}
{\theorembodyfont{\rmfamily}}
\newtheorem{lemma}[theorem]{Lemma}
\newtheorem{proposition}[theorem]{Proposition}
{\theorembodyfont{\rmfamily}\newtheorem{remark}[theorem]{Remark}}
{\theorembodyfont{\rmfamily}}

\newcommand{\N}{\mathbb{N}}

\newcommand{\C}{\mathbb{C}}
\newcommand{\R}{\mathbb{R}}
\newcommand{\Z}{\mathbb{Z}}
\newcommand{\dd}{\,{\rm d}}

\newenvironment{proof}{{\vskip\baselineskip\noindent\textbf{Proof:}}}%
{\hspace*{.1pt}\hspace*{\fill}\BOX\vskip\baselineskip}

\newenvironment{proofx}[1]%
{\vskip\baselineskip\noindent\textbf{Proof of {#1}:}}%
{\hspace*{.1pt}\hspace*{\fill}\BOX\vskip\baselineskip}
{\vskip\baselineskip\noindent\textbf{Proof of Theorem \protect\ref{#1}:}}%
{\hspace*{.1pt}\hspace*{\fill}\BOX\vskip\baselineskip}
{\vskip\baselineskip\noindent\textbf{Proof of Theorems \protect\ref{#1} --
\protect\ref{#2}:}}%
{\hspace*{.1pt}\hspace*{\fill}\BOX\vskip\baselineskip}

\begin{document}

\title{Note on the stability of planar stationary flows in an exterior domain without symmetry}

\author{Mitsuo Higaki$^\ast$ \\
Department of Mathematics, Kobe University\\
1-1 Rokkodai, Nada-ku, Kobe 657-8501, Japan}

\date{}

\maketitle

\noindent {\bf Abstract.}\, The asymptotic stability of two-dimensional stationary flows in a non-symmetric exterior domain is considered. Under the smallness condition on  initial perturbations, we show the stability of the small stationary flow whose leading profile at spatial infinity is given by the rotating flow decaying in the scale-critical order $O(|x|^{-1})$. Especially, we prove the  $L^p$-$L^q$ estimates to the semigroup associated with the linearized equations.


\footnote[0]{$^\ast$Most of this work was done when the author was a PhD student in Kyoto University.}

\section{Introduction}\label{intro}

In this paper we consider the perturbed Stokes equations for viscous incompressible flows in a two-dimensional exterior domain {\color{black} $\Omega$ with a smooth boundary.} 
\begin{equation}\tag{PS}\label{PS}
\left\{
\begin{aligned}
\partial_t v - \Delta v + V\cdot\nabla v + v\cdot\nabla V + \nabla q
&\,=\,0\,,~~~~t>0\,,~~x \in \Omega\,, \\
{\rm div}\,v &\,=\, 0\,,~~~~t\ge0\,,~~x \in \Omega\,, \\
v|_{\partial \Omega} &\,=\,0\,,~~~~t>0\,, \\
v|_{t=0} &\,=\, v_0\,,~~~~x \in \Omega\,.
\end{aligned}\right.
\end{equation}
Here the unknown functions $v=v(t,x)=(v_1(t,x), v_2(t,x))^\top$ and $q=q(t,x)$ are respectively the velocity field and the pressure field of the fluid, and $v_0=v_0(x)= (v_{0,1}(x), v_{0,2}(x))^\top$ is a given initial velocity field. The given vector field $V=V(x) = (V_1(x), V_2(x))^\top$ is assumed to be time-independent and decay in the scale-critical order $V(x)=O(|x|^{-1})$ at spatial infinity. We use the standard notations for differential operators with respect to the variables $t$ and $x=(x_1,x_2)^\top$: $\partial_t = \frac{\partial}{\partial t}$, $\partial_j = \frac{\partial}{\partial x_j}$, $\Delta=\partial_1^2 + \partial_2^2$, $V\cdot\nabla v + v\cdot\nabla V=\sum_{j=1}^{2} V_j \partial_j v + v_j \partial_j V$, ${\rm div}\,v=\partial_1 v_1 + \partial_2 v_2$. The exterior domain $\Omega$ is assumed to be contained by the domain exterior to the radius-$\frac12$ disk $\{x\in\R^2~|~|x|>\frac12\}$. 

{\color{black} 
The equations \eqref{PS} have been studied as the linearization of the Navier-Stokes equations, $\partial_t u-\Delta u + u\cdot\nabla u + \nabla p = f$, ${\rm div}\,u=0$ in $\Omega$, $u=b$ on $\partial \Omega$, and $u\to 0$ as $|x| \to \infty$ with some given data $f$ and $b$, around its stationary solution $V$. The analysis of \eqref{PS} is especially important when one considers the asymptotic stability of the stationary solutions.
The aim of this paper is to investigate the time-decay estimates to the equations \eqref{PS}, under a suitable condition on the vector field $V$.}
In the three-dimensional case, Borchers and Miyakawa \cite{BM} establishes the $L^p$-$L^q$ estimates to \eqref{PS} for the small stationary Navier-Stokes flow $V$ satisfying $V(x)=O(|x|^{-1})$ as $|x| \to \infty$. This result is extended to the case when $V$ belongs to the Lorenz space $L^{3,\infty}(\Omega)$ by Kozono and Yamazaki \cite{KY}. We also refer to the whole-space result by Hishida and Schonbek \cite{HS} considering the time-dependent $V=V(t,x)$ in the scale-critical space $L^\infty(0,\infty; L^{3,\infty}(\R^3))$, where the $L^p$-$L^q$ estimates are obtained for the evolution operator associated with the linearized equations around $V(t,x)$.

{\color{black} 
For the two-dimensional problem \eqref{PS},} 
the analysis becomes quite complicated and there is no general result especially for the time-decay estimate so far. The difficulty arises from the unavailability of the Hardy inequality in the form 
\begin{align}
\big\| \frac{f}{|x|} \big\|_{L^2(\Omega)} \le C \|\nabla f\|_{L^2(\Omega)}\,,~~~~~~
f \in \dot{W}^{1,2}_0 (\Omega)
\,=\, \overline{C^\infty_0 (\Omega)}^{\|\nabla f\|_{L^2(\Omega)}}
\,,\label{hardy}
\end{align}
where $C^\infty_0(\Omega)$ is the set of smooth and compactly supported functions in $\Omega$. The validity of this bound is well known for three-dimensional exterior domains, and the results mentioned in the above essentially rely on the inequality \eqref{hardy}. One can recover the Hardy inequality in the two-dimensional case if the factor $|x|^{-1}$ in the left-hand side of \eqref{hardy} is replaced with a logarithmic correction $|x|^{-1} \log(e+|x|)^{-1}$, but this inequality has only a narrow application in our scale-critical framework. Another way to recover the inequality \eqref{hardy} is to impose a symmetry on both $\Omega$ and $f$ about an axis which we may take to be the $x_1$-axis
\begin{equation}\tag{Sym$_1$}\label{sym1}
\left\{
\begin{aligned}
&{\rm If}~(x_1, x_2)^\top\in\Omega\,,~{\rm then}~(x_1, -x_2)^\top\in \Omega\,, \\
&f(x_1, x_2) = -f(x_1, -x_2)~\,{\rm holds}
\end{aligned}\right.
\end{equation}
(for the proof see Galdi \cite{G}), and {\color{black} such an inequality} is applied in the analysis of \eqref{PS} for the case when $V$ is symmetric. Yamazaki \cite{Y2} proves the $L^p$-$L^q$ estimates to $\eqref{PS}$ with the symmetric Navier-Stokes flow $V(x)=O(|x|^{-1})$, under the symmetry conditions on both the domain and given data. We note that these estimates imply the asymptotic stability of $V$ under symmetric initial $L^2$-perturbations; see also Galdi and Yamazaki \cite{GY}.

An important remark is given by Russo \cite{R2} concerning the Hardy-type inequality in two-dimensional exterior domains without symmetry. Let us introduce the next scale-critical radial flow $W=W(x)$, which is called the flux carrier.
\begin{align}
W(x) \,=\, \frac{x}{|x|^2}\,,~~~~~~
x \in \R^2\setminus\{0\}\,.
\label{fluxcarrier}
\end{align}
Then, from the existence of a potential to $W(x)=\nabla \log |x|$, one can show that the following Hardy-type inequality holds in the $L^2$-inner product $\langle \cdot, \cdot \rangle_{L^2(\Omega)}$:
\begin{align}
|\langle u\cdot\nabla u\,, W \rangle_{L^2(\Omega)}|
\le 
C \|\nabla u\|_{L^2(\Omega)}^2\,,~~~~~~
u \in \dot{W}^{1,2}_{0,\sigma} (\Omega)
\,=\, \overline{C^\infty_{0,\sigma} (\Omega)}^{\|\nabla u\|_{L^2(\Omega)}}\,,
\label{hardy2}
\end{align}
where $C^\infty_{0,\sigma}(\Omega)$ denotes the function space $\{ f \in C^\infty_0(\Omega)^2~|~{\rm div}\,f=0\}$. Based on the energy method with the application of \eqref{hardy2}, Guillod \cite{Gui} proves the global $L^2$-stability of the flux carrier $\delta W$ when the flux $\delta$ is small enough. On the other hand, the validity of the inequality \eqref{hardy2} essentially depends on the potential property of $W$. Indeed, as is pointed out in \cite{Gui}, the bound \eqref{hardy2} breaks down if $W$ is replaced by the next rotating flow $U=U(x)$:
\begin{align}
U(x) \,=\, \frac{x^\bot}{|x|^2}\,,~~~~~~
x^\bot\,=\,(-x_2,x_1)^\top\,,~~~~~~
x \in \R^2\setminus\{0\}\,.
\label{rotatingflow}
\end{align}
Hence, if we consider the problem \eqref{PS} with $V=\alpha U$, $\alpha\in\R\setminus\{0\}$, the linearized term $\alpha(U\cdot\nabla v +v\cdot\nabla U)$ can no more be regarded as a perturbation from the Laplacian, and we cannot avoid the difficulty coming from the lack of the Hardy inequality. Maekawa \cite{Ma1} studies the stability of the flow $\alpha U$ in the exterior unit disk. The symmetry of the domain allows us to express the solution to the problem \eqref{PS} explicitly through the Dunford integral of the resolvent operator. Based on this representation formula, \cite{Ma1} obtains the $L^p$-$L^q$ estimates to \eqref{PS} with $V=\alpha U$ for small $\alpha$, and shows the asymptotic $L^2$-stability of $\alpha U$ if $\alpha$ and initial perturbations are sufficiently small. This result is extended by the same author in \cite{Ma2}  for the more general class of $V$ in \eqref{PS} including the flow of the form $V=\alpha U + \delta W$ with small $\alpha$ and $\delta$; see \cite{Ma2} for details.

Our first motivation is to generalize the result in \cite{Ma1} to the case when the domain loses symmetry (and the second one is explained in 
{\color{black} 
Remark \ref{remark.assumption1}}
below). Let us prepare the assumptions on the domain $\Omega$ and the stationary vector field $V$ in \eqref{PS} considered in this paper. We denote by $B_{\rho}(0)$ the two-dimensional disk of radius $\rho>0$ centered at the origin.
%
\begin{assumption}\label{assumption}
{\rm (1)} There is a positive constant $d\in (0,\frac14)$ such that the complement of the domain $\Omega$ satisfies
\begin{align}\label{domain}
\overline{B_{1-2d}(0)} \subset \Omega^{\rm c} \subset \overline{B_{1-d}(0)}\,.
\end{align}

{\color{black}
\noindent {\rm (2)}
The vector field $V$ in \eqref{PS} satisfies ${\rm div}\,V=0$ in $\Omega$ and the asymptotic behavior 
\begin{align}
V(x) \,=\, \beta U(x) + R(x)\,,~~~~~~x\in\Omega\,,\label{asymptoticbehavior}
\end{align}
where $U$ is the rotating flow in \eqref{rotatingflow}. The constant $\beta$ and the remainder $R$ are assumed to satisfy the following conditions with some $\gamma\in(\frac12,1)${\rm :}
\begin{align}
&\beta\in(0,1)\,,\label{beta} \\
&\sup_{x\in\Omega} |x|^{1+\gamma} |R(x)| 
\le C d\,,
\label{remainder}
\end{align}
where the constant $C$ depends only on $\gamma$. Moreover, the derivative of $V$ satisfies $(1+|x|) \nabla^k V\in L^\infty(\Omega)$ for $k=1, 2${\rm ;} see Remark \ref{remark.assumption} {\rm (3)}} for this assumption.
\end{assumption}
%
\begin{remark}\label{remark.assumption}
(1) Let us consider the case where the constant $\beta$ in the assumption is given in the form of $\beta = \alpha + \tilde{\alpha}_d$ with $\alpha\in(0,1)$ and $|\tilde{\alpha}_d| \le C d$, and assume that $\alpha$ and $d$ are sufficiently small so that $\beta\in (0,1)$. Then formally taking $d=0$ in \eqref{domain}--\eqref{remainder} we obtain the flow $V=\alpha U$ in the exterior disk $\Omega=\R^2\setminus\overline{B_{1}(0)}$, which solves the following two-dimensional stationary Navier-Stokes equations (SNS): 
$-\Delta u + u\cdot\nabla u + \nabla p = 0$,
${\rm div}\,u=0$ in $\Omega$, $u=b$ on $\partial \Omega$, and $u\to 0$ as $|x| \to \infty$ with $b=\alpha x^{\bot}$. The vector field $V$ in \eqref{asymptoticbehavior}--\eqref{remainder} describes the flow around $\alpha U$ created from a small perturbation to the exterior disk, and hence, one can naturally expect the existence of {\color{black} such solutions} to the nonlinear problem (SNS) when $\alpha$ is small enough if $b-\alpha x^{\bot}$ is sufficiently small with respect to $0<d\ll 1$. Indeed, imposing the symmetry condition on both the domain perturbation in \eqref{domain} and $b$ as
\begin{equation}\tag{Sym$_2$}\label{sym2}
\left\{
\begin{aligned}
&{\rm If}~(x_1, x_2)^\top\in\Omega\,,
~{\rm then}~(x_1, -x_2)^\top\in \Omega
~{\rm and}~(-x_1, x_2)^\top\in \Omega\,, \\
&b_1(x_1, x_2) = -b_1(x_1, -x_2) = b_1(-x_1, x_2)~\,{\rm and}  \\
&b_2(x_1, x_2) = b_2(x_1, -x_2) = -b_2(-x_1, x_2)~\,{\rm hold},
\end{aligned}\right.
\end{equation}
under the smallness condition on $\alpha$ and $b-\alpha x^{\bot}$, we can construct the Navier-Stokes flow $V$ satisfying at least \eqref{asymptoticbehavior} and \eqref{beta} (but with no specific spatial decay rate for $R$ in general). The proof is based on the Galerkin method and the recovered Hardy-inequality \eqref{hardy}.
We refer to \cite{G}, Russo \cite{R1}, \cite{R2}, Yamazaki \cite{Y1}, and Pileckas and Russo \cite{PR} for the solvability of (SNS) under symmetry conditions. In particular, the solutions decaying in the scale-critical order $O(|x|^{-1})$ are obtained in \cite{Y1} under a stronger symmetry condition on the domain than \eqref{sym2}.
The reader is also referred to Hillairet and Wittwer \cite{HW} proving the existence of solutions to (SNS) in the exterior disk with $b=\alpha x^\bot + \tilde{b}$ when $\alpha$ is large enough and $\tilde{b}$ is sufficiently small.

\noindent (2) The novelty of our assumption is that we do not impose the symmetry either on the domain $\Omega$ and the flow $V$, and it is a crucial assumption for the stability analysis in \cite{GY, Y2} to resolve the difficulty related to the lack of the Hardy inequality. While one can realize the exterior disk case in \cite{Ma1} by putting $d=0$ to \eqref{domain}, \eqref{asymptoticbehavior}, and \eqref{remainder} formally. In this sense, the assumption above gives a generalization of the setting in \cite{Ma1} to non-symmetric domain cases.

\noindent (3) The assumption on the derivatives of $V$ is only used in proving that the operator ${\mathbb A}_V$ defined in \eqref{perturbed Stokes op.} is a relatively compact perturbation of the Stokes operator ${\mathbb A}$; see below for the definition of ${\mathbb A}$. Indeed, it is possible to show this property under a weaker assumption on the spatial decay of the derivatives of $V$, such as $(1+|x|)^r \nabla V\in L^\infty(\Omega)$ for some $r\in(0,1]$ and $\nabla^2 V\in L^\infty(\Omega)$, but we omit the details here for the sake of brevity.

\noindent (4) A vector field $V$ of the type \eqref{asymptoticbehavior} also naturally appears in the study of the Navier-Stokes flows around a rotating obstacle. Let us consider the situation where the obstacle $\Omega^{\rm c}$ rotates around the origin with a constant speed $\alpha\in \R\setminus\{0\}$. Then the time-periodic Navier-Stokes flow moving with the rotating obstacle gives a stationary solution to the problem (RNS): $\partial_t u -\Delta u - \alpha(x^{\bot}\cdot\nabla u - u^{\bot}) + u\cdot\nabla u + \nabla p = f$, ${\rm div}\,u=0$ in $\Omega$, $u=\alpha x^{\bot}$ on $\partial \Omega$, and $u\to 0$ as $|x| \to \infty$. Here we take the reference frame attached to the obstacle; see Hishida \cite{H} for details. The stationary problem of (RNS) is analyzed by Higaki, Maekawa, and Nakahara \cite{HMN}, where the existence and uniqueness of stationary solutions decaying in the order $O(|x|^{-1})$ is proved when $\alpha$ is sufficiently small and $f$ is of a divergence form $f={\rm div}\,F$ for some $F$ which is small in a scale-critical norm. Moreover, the leading profile at spatial infinity is shown to be $C \frac{x^\bot}{|x|^2} + O(|x|^{-1-\gamma})$ for some constant $C$ if $F$ satisfies a decay condition $F=O(|x|^{-2-\gamma})$ with $\gamma\in(0,1)$. By extending the proof in \cite{HMN}, one can construct the stationary solutions to (RNS) satisfying the estimates \eqref{asymptoticbehavior}--\eqref{remainder} under the condition on the domain \eqref{domain}. The analysis in this paper cannot be directly applied to the stability analysis of a stationary solution $V$ to (RNS) for its time-dependence in the original frame, and however, we can perform an analogous analysis to the linearized equations of (RNS) around $V$ as is explained in Remark \ref{remark.assumption1} below.
\end{remark}
%
%
\begin{remark}\label{remark.assumption1}
Another motivation comes from the stability analysis of stationary solutions to the equations (RNS) introduced in Remark \ref{remark.assumption} (4). Obviously, letting us denote the linearization to (RNS) around a stationary solution $V$ by (PRS), then the two equations \eqref{PS} and (PRS) are different from each other due to the additional term $-\alpha(x^{\bot}\cdot\nabla v - v^{\bot})$ in (PRS). However, if we consider the {\it resolvent problems} of each equation, there are some common features thanks to the property of the term $\alpha(x^{\bot}\cdot\nabla v - v^{\bot})= \sum_{n\in\Z} i\alpha n \mathcal{P}_n v$, which is derived from the Fourier expansion of $v|_{\{|x|>1\}}$; see \eqref{polar.e_theta} and \eqref{def.P_n} in Subsection \ref{op.s in polar coord.}. In particular, for stationary solutions to (RNS) satisfying \eqref{asymptoticbehavior}--\eqref{remainder} on the domain in \eqref{domain}, we can reproduce a similar calculation performed in this paper to the resolvent problem of (PRS), by observing that the appearance of $\sum_{n\in\Z} i\alpha n \mathcal{P}_n v$ in the resolvent equation (restricted on $|x|>1$) leads to the shifting of the resolvent parameter from $\lambda\in\C$ to $\lambda+in\alpha$ in the $n$-Fourier mode. Although the stability of the stationary solutions to (RNS) still remains open, our analysis in this paper will contribute to the resolvent estimate of the linearized problem (PRS). 
\end{remark}
%

Before stating the main result, let us introduce some notations and basic 
facts related to the problem \eqref{PS}. We denote by $L^2_\sigma(\Omega)$ the $L^2$-closure of $C^\infty_{0,\sigma}(\Omega)$. The orthogonal projection ${\mathbb P}: L^2(\Omega)^2 \to L^2_\sigma(\Omega)$ is called the Helmholtz projection. Then the Stokes operator ${\mathbb A}$ with the domain 
$D({\mathbb A}) = L^{2}_{\sigma}(\Omega) \cap W^{1,2}_0(\Omega)^2 \cap W^{2,2}(\Omega)^2$ is defined as ${\mathbb A} = -{\mathbb P} \Delta$, and it is well known that the Stokes operator is nonnegative and self-adjoint in $L^{2}_{\sigma}(\Omega)$. Finally we define the perturbed Stokes operator ${\mathbb A}_V$ as
\begin{equation}\label{perturbed Stokes op.}
\begin{split}
D({\mathbb A}_V) &\,=\, D({\mathbb A})\,, \\
{\mathbb A}_V v &\,=\, {\mathbb A} v 
+ {\mathbb P}(V\cdot\nabla v + v\cdot\nabla V)\,.
\end{split}
\end{equation}
A standard argument for sectorial operators implies that $-{\mathbb A}_V$ generates a $C_0$-analytic semigroup in $L^2_\sigma(\Omega)$. We denote this semigroup by $e^{-t {\mathbb A}_V}$. Then our main result is stated as follows. Let $\beta$ and $d$ be the constants in Assumption \ref{assumption}.
%
\begin{theorem}\label{maintheorem}
There are positive constants $\beta_\ast$ and $\mu_\ast$ such that if $\beta \in (0, \beta_\ast)$ and {\color{black} $d\in(0,\mu_\ast \beta^2)$} then the following statement holds. Let $1<q\le 2\le p<\infty$. Then we have 
\begin{align}
\|e^{-t {\mathbb A}_V} f\|_{L^p(\Omega)} 
& \le 
\frac{C}{\beta^2} 
t^{-\frac1q+\frac1p}
\|f\|_{L^q(\Omega)}\,,
~~~~~~ t>0\,,
\label{est1.maintheorem} \\
\|\nabla e^{-t {\mathbb A}_V} f\|_{L^2(\Omega)} 
& \le 
\frac{C}{\beta^2} 
t^{-\frac1q}
\|f\|_{L^q(\Omega)}\,,
~~~~~~ t>0\,,
\label{est2.maintheorem}
\end{align}
for $f\in L^2_\sigma(\Omega) \cap L^q(\Omega)^2$. Here the constant $C$ is independent of $\beta$ and depends on $p$ and $q$.
\end{theorem}
%
As an application of Theorem \ref{maintheorem}, we can prove the asymptotic stability of $V$. Suppose that $V$ gives a stationary Navier-Stokes flow. Then the time evolution of perturbations around $V$ is governed by the Navier-Stokes equations, whose integral form is written as
\begin{align}\tag{INS}\label{INS}
v(t) \,=\, e^{-t {\mathbb A}_V} v_0 
- \int_0^t e^{-(t-s) {\mathbb A}_V} \mathbb{P} (v\cdot\nabla v)(s) \dd s\,,
~~~~~~ t>0\,.
\end{align}
The proof of the following result is omitted in this paper, since {\color{black} it is just a reproduction of the argument in \cite[Section 4]{Ma1}} using the Banach fixed point theorem.
%
\begin{theorem}\label{maintheorem2}
Let $\beta_\ast$ and $\mu_\ast$ be the constants in Theorem \ref{maintheorem}. Then there is a positive constant $\nu_\ast$ such that if $\beta \in (0, \beta_\ast)$, 
{\color{black} $d\in(0,\mu_\ast \beta^2)$,} 
and $\|v_0\|_{L^2(\Omega)}\in(0, \nu_\ast \beta^4)$ then there exists a unique solution $v \in C([0,\infty); L^2_\sigma(\Omega))\,\cap\,C((0,\infty);W^{1,2}_0(\Omega)^2)$ to \eqref{INS} satisfying
\begin{align}
\lim_{t\to\infty} t^{\frac{k}{2}} \|\nabla^k v(t)\|_{L^2(\Omega)} \,=\, 0\,,~~~~~~~~
k\,=\,0,1\,.
\end{align}
\end{theorem}
%
The proof of Theorem \ref{maintheorem} relies on the resolvent estimate to the  perturbed Stokes operator $\mathbb{A}_V$. Since the difference 
{\color{black} $\mathbb{A}_V-\mathbb{A}$ is relatively compact to $\mathbb{A}$ in $L^2_\sigma(\Omega)$,}
one can show that the spectrum of $-\mathbb{A}_V$ has the structure $\sigma(-{\mathbb A}_V)={\color{black} (-\infty,0]} \cup \sigma_{{\rm disc}}(-{\mathbb A}_V)$ in $L^2_\sigma(\Omega)$, where $\sigma_{{\rm disc}}(-{\mathbb A}_V)$ denotes the set of discrete spectrum of $-{\mathbb A}_V$; see \cite[Lemma 2.11 and Proposition 2.12]{Ma1}. By using the identity $v\cdot \nabla v = \frac12 \nabla |v|^2 + v^\bot {\rm rot}\,v$ with ${\rm rot}\,v=\partial_1 v_2 - \partial_2 v_1$ and ${\rm rot}\, U =0$ in $x \in \Omega$, we can write the resolvent problem associated with \eqref{PS} as, for $V$ given by \eqref{asymptoticbehavior}, 
\begin{equation}\tag{RS}\label{RS}
\left\{
\begin{aligned}
\lambda v - \Delta v + \beta U^\bot {\rm rot}\,v
+ {\rm div}\,(R\otimes v +  v\otimes R) 
+ \nabla q
&\,=\,f\,,~~~~x \in \Omega\,, \\
{\rm div}\,v &\,=\, 0\,,~~~~x \in \Omega\,, \\
v|_{\partial \Omega} &\,=\,0\,. 
\end{aligned}\right.
\end{equation}
Here $\lambda\in\C$ is the resolvent parameter and we have used the conditions ${\rm div}\,v={\rm div}\,R=0$ to derive $R\cdot\nabla v +v\cdot\nabla R = {\rm div}\,(R\otimes v +  v\otimes R)$. Hence, the proof of Theorem \ref{maintheorem} is complete as soon as we show that there is a sector $\Sigma$ included in the resolvent set $\rho(-{\mathbb A}_V)$, and that the following estimates to \eqref{RS} hold for $q\in(1,2]$ and $f\in L^2_\sigma(\Omega) \cap L^q(\Omega)^2$:
\begin{equation}\label{resolvent estimates.0}
\begin{split}
\|(\lambda+{\mathbb A}_V)^{-1} f\|_{L^2(\Omega)} 
& \le 
\frac{C}{\beta^2} |\lambda|^{-\frac32+\frac1q}
\|f\|_{L^q(\Omega)}\,,
~~~~\lambda \in {\color{black} \Sigma}\,,\\
\|\nabla (\lambda+{\mathbb A}_V)^{-1} f\|_{L^2(\Omega)} 
& \le 
\frac{C}{\beta^2} |\lambda|^{-1+\frac1q}
\|f\|_{L^q(\Omega)}\,,
~~~~\lambda \in {\color{black} \Sigma}\,.
\end{split}
\end{equation}
Let us prepare the ingredients for the proof of the resolvent estimates \eqref{resolvent estimates.0}. Our approach is based on the energy method to \eqref{RS}, and thus one of the most important steps is to obtain the estimate for the term $|\langle \beta U^{\bot} {\rm rot}\,v\,,v \rangle_{L^2(\Omega)}|$ which enables us to close the energy computation. Again we note that the bound $|\langle \beta U^{\bot} {\rm rot}\,v\,,v \rangle_{L^2(\Omega)}| \le C \beta \|\nabla v\|_{L^2(\Omega)}^2$ is no longer available contrary to the three-dimensional cases.

Firstly let us examine the next inequality containing the parameter $T \gg 1$:
\begin{align}\label{resolvent estimates.1}
|\langle \beta U^{\bot} {\rm rot}\,v, v \rangle_{L^2(\Omega)}|
\le 
\frac{\beta}{T} 
\|\nabla v\|_{L^2(\Omega)} \|v\|_{L^2(\Omega)} 
+ C \beta \Theta(T) \|\nabla v\|_{L^2(\Omega)}^2\,, 
\end{align}
where the function $\Theta(T)$ satisfies $\Theta(T) \approx \log T$ if $T \gg 1$. This inequality leads to the closed energy computation for \eqref{RS}, as long as the coefficient $C \beta \Theta(T)$ is small enough so that the second term in the right-hand side of \eqref{resolvent estimates.1} can be controlled by the dissipation from the Laplacian in \eqref{RS}. However, this observation does not give the information about the spectrum of $-\mathbb{A}_V$ near the origin. More precisely, we cannot close the energy computation when the resolvent parameter $\lambda$ is exponentially small with respect to $\beta$, that is, when $0<|\lambda| \le O(e^{-\frac{1}{\beta}})$. We emphasize that this difficulty is essentially due to the unavailability of the Hardy inequality \eqref{hardy} in two-dimensional exterior domains.

To overcome the difficulty for the case $0<|\lambda| \le O(e^{-\frac{1}{\beta}})$, we rely on the representation formula to the resolvent problem in the exterior unit disk established in \cite{Ma1}. We denote by $D$ the exterior unit disk $\R^2\setminus\overline{B_1(0)}=\{x\in\R^2~|~|x|>1\}$. Since the restriction $(v|_{D}, q|_{D})$ gives a unique solution to the next problem for $(w,r)$:
\begin{equation}\tag{RS$^{\rm ed}$}\label{RSed}
\left\{
\begin{aligned}
\lambda w - \Delta w + \beta U^\bot {\rm rot}\,w + \nabla r
&\,=\,-{\rm div}\,(R\otimes v + v\otimes R) + f\,,~~~~x\in D\,, \\
{\rm div}\,w &\,=\, 0\,,~~~~x\in D\,, \\
w|_{\partial D} &\,=\,v|_{\partial D}\,,
\end{aligned}\right.
\end{equation}
we can study the a priori estimates of $w=v|_{D}$ based on the solution formula to \eqref{RSed}. Then a detailed calculation shows that 
$|\langle \beta U^{\bot} {\rm rot}\,w, w \rangle_{L^2(D)}|$
satisfies
\begin{equation}\label{resolvent estimates.2}
\begin{split}
&
|\langle \beta U^{\bot} {\rm rot}\,w, w \rangle_{L^2(D)}| \\
&\le 
\frac{C}{\beta^4} 
\big(\|R\otimes v + v\otimes R\|_{L^2(\Omega)}
+ 
{\color{black}
\beta \sum_{|n|=1} \|\mathcal{P}_n v\|_{L^{\infty}(\partial D)} 
} 
\big)^2 \\
& \quad
+ \frac{C}{\beta^4} |\lambda|^{-2+\frac2q} \|f\|_{L^q(\Omega)}^2 
+ C\beta\|\nabla v\|_{L^2(\Omega)}^2\,,
\end{split}
\end{equation}
{\color{black}
where $\mathcal{P}_n v$ denotes the Fourier $n$-mode of 
$v|_{\overline{D}}$; see \eqref{def.P_n} in Subsection \ref{op.s in polar coord.} for the definition.
If} 
we obtain \eqref{resolvent estimates.2} then the estimate of $|\langle \beta U^{\bot} {\rm rot}\,v, v \rangle_{L^2(\Omega)}|$ is derived by using the Poincar${\rm \acute{e}}$ inequality on the bounded domain $\Omega \setminus\overline{D}$. However, in closing the energy computation, we need to be careful about the $\beta$-singularity in the coefficients in \eqref{resolvent estimates.2}. In fact, the first term in the right-hand side of \eqref{resolvent estimates.2} {\color{black} has to be} controlled by the dissipation as
\begin{align*}
\frac{C}{\beta^4} 
\big(\|R\otimes v + v\otimes R\|_{L^2(\Omega)}
+ 
{\color{black}
\beta \sum_{|n|=1} \|\mathcal{P}_n v\|_{L^{\infty}(\partial D)} 
} 
\big)^2
\le 
\frac{C}{\beta^4}
(d + \beta d^\frac12)^2
\|\nabla v\|_{L^2(\Omega)}^2\,,
\end{align*}
and then the smallness of $C (d + \beta d^\frac12)^2 \beta^{-4} \ll 1$ is required  in order to close the energy computation. This condition is achieved by imposing the smallness on the distance $d$ between the domain $\Omega$ and the exterior unit disk, which is introduced in Assumption \ref{assumption}.

{\color{black} Finally, we pay close attention to the $\beta$-dependencies appearing in Theorem \ref{maintheorem}. If we consider the limit case $d=0$ and $V=\beta U$ in Assumption \ref{assumption}, then the term 
\begin{align*}
\beta U^\bot {\rm rot}\,v + {\rm div}\,(R\otimes v +  v\otimes R) 
\,=\, \beta U^\bot {\rm rot}\,v
\end{align*}
in \eqref{RS} has an oscillation effect on the solutions in the exterior disk $\Omega={\color{black} D}$ at least when $\lambda=0$. Indeed, for the solutions to \eqref{RS} with $\lambda=0$, this effect leads to the faster spatial decay for the nonradial part compared with the case ${\color{black} \beta}=0$ (i.e. the Stokes equations case), and this observation is indeed an important step in \cite{HW} to prove the existence of the Navier-Stokes flows around ${\color{black} \beta} U$ in the exterior disk when the rotation ${\color{black} \beta}$ is large, as explained in Remark \ref{remark.assumption} (1). However, contrary to the stationary problem, the situation becomes more complicated if we consider the nonstationary problem requiring the analysis of \eqref{RS} for nonzero $\lambda\in\C\setminus\{0\}$, since there is an interaction between the two oscillation effects due to the terms $\lambda v$ and ${\color{black} \beta} U^\bot {\rm rot}\,v$ in \eqref{RS}. In fact, even in the exterior disk, a detailed analysis to the representation of the resolvent operator suggests the existence of a time-frequency domain, which we call the {\it nearly-resonance regime}, where the oscillation effect from ${\color{black} \beta} U^\bot {\rm rot}\,v$ is drastically weakened by the one from $\lambda v$ and hence the ${\color{black} \beta}$-singularity appears in the operator norm of the resolvent. The existence of the nearly-resonance regime yields that the stability of the ${\color{black} \beta} U$-type flows is sensitive under the perturbation of the domain. This is the reason why the distance $d$ between the fluid domain $\Omega$ and the exterior disk is assumed to be small depending on ${\color{black} \beta}$ in Theorem \ref{maintheorem}, and therefore, our argument is only applicable to the problem imposed on the exterior disk with a small perturbation. However, it is unclear whether the condition on $d$ being algebraically small in $\beta$ in Theorem \ref{maintheorem} can be relaxed or not. Additionally, Lemma \ref{lem.est.F} in Appendix \ref{app.proof.prop.est.F} implies that the nearly-resonance regime lies in the annulus $e^{-\frac{c}{\beta^2}} \le |\lambda| \le e^{-\frac{c'}{\beta}}$ in the complex plane. As far as the author knows, the existence of such time-frequency domain and the qualitative analysis seem to be new and have not been achieved before.
}

This paper is organized as follows. In Section \ref{sec.pre} we recall some basic facts from vector calculus in polar coordinates, and derive the resolvent estimate to \eqref{RS} when $|\lambda|\ge O(\beta^2e^{-\frac{1}{6\beta}})$ by a standard energy method. In Section \ref{sec.RSed} the resolvent problem is discussed for the case $0<|\lambda| < e^{-\frac{1}{6\beta}}$. In Subsections \ref{subsec.RSed.f}, \ref{subsec.RSed.divF}, and \ref{subsec.RSed.b} we derive the estimates to the problem \eqref{RSed} by using the representation formula. The results in Subsections \ref{subsec.RSed.f}--\ref{subsec.RSed.b} are applied in Subsection \ref{apriori2}, where the resolvent estimate to \eqref{RS} is established in the {\color{black} exceptional} region $0<|\lambda| < e^{-\frac{1}{6\beta}}$. Section \ref{sec.maintheorem} is devoted to the proof of Theorem \ref{maintheorem}. 
\section{Preliminaries}\label{sec.pre}

This section is devoted to the preliminary analysis on the resolvent problem \eqref{RS} 
and \eqref{RSed} in the introduction. In Subsections \ref{op.s in polar coord.} and \ref{biotsavart} we recall some basic facts from vector calculus in polar coordinates. In Subsection \ref{apriori1} we show that the resolvent estimates in \eqref{resolvent estimates.0} are valid if the resolvent parameter $\lambda$ satisfies $|\lambda|\ge O(\beta^2e^{-\frac{1}{6\beta}})$. Let us recall that $D$ denotes the exterior unit disk $\R^2\setminus\overline{B_1(0)}=\{x\in\R^2~|~|x|>1\}$ as in the introduction.
%
\subsection{Vector calculus in polar coordinates and Fourier series}\label{op.s in polar coord.}
%
We introduce the usual polar coordinates on $D$. Set 
{\allowdisplaybreaks 
\begin{align*}
& 
x_1 \,=\, r\cos \theta\,,~~~~~~
x_2 \,=\, r\sin \theta\,,~~~~~~~~
r\,=\, |x| \ge 1\,,~~~~\theta\in [0,2\pi)\,,\\
& 
{\bf e}_r \,=\, \frac{x}{|x|}\,,~~~~~~
{\bf e}_\theta \,=\, \frac{x^\bot }{|x|} 
\,=\, \partial_\theta {\bf e}_r\,.
\end{align*}
}
Let $v=(v_1,v_2)^\top$ be a vector field defined on $D$. Then we set
\begin{align*}
& v\,=\, v_r {\bf e}_r  +  v_\theta {\bf e}_\theta\,,~~~~~~~~
v_r \,=\, v\cdot {\bf e}_r\,,~~~~~~
v_\theta \,=\, v\cdot {\bf e}_\theta\,.
\end{align*}
The following formulas will be used:
\begin{align}
{\rm div}\,v 
& \,=\, \partial_1 v_1 + \partial_2 v_2 
\,=\,  \frac1r \partial_r (r v_r) + \frac1r \partial_\theta v_\theta\,, 
\label{polar.div} \\
{\rm rot}\,v 
& \, = \, \partial_1 v_2 - \partial_2 v_1 
\,=\,  \frac1r \partial_r (r v_\theta) - \frac1r \partial_\theta v_r \,,
\label{polar.rot} \\
|\nabla v |^2 
& \,=\, 
|\partial_r v_r|^2 + |\partial_r v_\theta|^2 
+ \frac{1}{r^2} 
\big( |\partial_\theta v_r - v_\theta |^2 
+ |v_r + \partial_\theta v_\theta |^2 \big )\,,
\label{polar.grad}
\end{align} 
and 
\begin{align}
\begin{split}\label{polar.laplace}
-\Delta v & 
\,=\, \Big( -\partial_r \big( \frac1r \partial_r (r v_r ) \big)  
- \frac{1}{r^2} \partial_\theta^2 v_r 
+ \frac{2}{r^2} \partial_\theta v_\theta \Big) {\bf e}_r \\
&\quad
+  \Big( - \partial_r \big ( \frac1r \partial_r (r v_\theta ) \big ) - \frac{1}{r^2} \partial_\theta^2 v_\theta -  \frac{2}{r^2} \partial_\theta v_r \Big) {\bf e}_\theta\,.
\end{split}
\end{align}
The formulas
\begin{align*}
\begin{split}
{\bf e}_r \cdot \nabla v 
\,=\, (\partial_r v_r ) {\bf e}_r  + (\partial_r v_\theta) {\bf e}_\theta\,,~~~~~~
{\bf e}_\theta \cdot \nabla v 
\,=\, \frac{\partial_\theta v_r - v_\theta}{r} {\bf e}_r  
+ \frac{\partial_\theta v_\theta + v_r}{r} {\bf e}_\theta
\end{split}
\end{align*}
imply the following equality:
\begin{align}\label{polar.e_theta}
x^\bot \cdot \nabla v - v^\bot & = |x| \big ( {\bf e}_\theta \cdot \nabla v \big ) - \big ( v_r {\bf e}_r^\bot + v_\theta {\bf e}_\theta^\bot \big ) \nonumber \\
& = (\partial_\theta v_r - v_\theta )\,  {\bf e}_r  \, + \, (\partial_\theta v_\theta + v_r) \, {\bf e}_\theta - \big ( v_r {\bf e}_r^\bot + v_\theta {\bf e}_\theta^\bot \big ) \nonumber \\
& = \partial_\theta v_r \,  {\bf e}_r + \partial_\theta v_\theta \, {\bf e}_\theta\,.
\end{align} 
This relation has been used in Remark \ref{remark.assumption1} in the introduction. For each $n\in \Z$, we denote by $\mathcal{P}_n$ the projection on the Fourier mode $n$ with respect to the angular variable $\theta$:
\begin{align}
\mathcal{P}_n v 
\,=\, v_{r,n}(r) e^{i n \theta} {\bf e}_r + v_{\theta,n}(r) e^{i n \theta} {\bf e}_\theta\,, 
\label{def.P_n}
\end{align}
where 
\begin{align*}
v_{r,n} (r) 
& \,=\, \frac{1}{2\pi} \int_0^{2\pi} v_r(r \cos \theta, r\sin\theta) e^{-i n \theta} \dd \theta\,,\\
v_{\theta,n} (r) 
& \,=\, \frac{1}{2\pi} \int_0^{2\pi} v_\theta(r \cos \theta, r\sin\theta) e^{-i n \theta} \dd \theta\,.
\end{align*}
We also set  for $m\in \N\cup \{0\}$,
\begin{align}
\mathcal{Q}_m v = \sum_{|n|=m+1}^\infty \mathcal{P}_n v\,.\label{def.Q_m}
\end{align}
For notational simplicity we often write $v_n$ instead of $\mathcal{P}_n v$. Each $\mathcal{P}_n$ defines an orthogonal projection in $L^2 (D)^2$. From \eqref{polar.grad} and \eqref{def.P_n}, for $n\in \N \cup\{0\}$ and $v$ in~$W^{1,2}(D)^2$ we see that
\begin{align*}
\| \nabla v \|_{L^2 (D)}^2 &\,=\, \sum_{n\in \Z} \| \nabla \mathcal{P}_n v \|_{L^2 (D)}^2\,, \\
|\nabla \mathcal{P}_n v |^2 
&\,=\, 
|\partial_r v_{r,n}|^2 
+ \frac{1+n^2}{r^2} |v_{r,n}|^2 
+ |\partial _r v_{\theta,n}|^2 
+ \frac{1+n^2}{r^2}|v_{\theta,n}|^2  
- \frac{4 n}{r^2} {\rm Im}( v_{\theta,n} \overline{v_{r,n}} )\,. 
\end{align*}
In particular, we have 
\begin{align}
| \nabla \mathcal{P}_n v |^2 \geq  |\partial_r v_{r,n}|^2 + \frac{(|n|-1)^2}{r^2} |v_{r,n}|^2 + |\partial _r v_{\theta,n}|^2 + \frac{(|n|-1)^2}{r^2}|v_{\theta,n}|^2 \,,\label{polar.grad.n'} 
\end{align}
and thus, from the definition of $\mathcal{Q}_m$ in \eqref{def.Q_m}, we have for $m\in \N\cup \{0\}$, 
\begin{align}
\| \nabla \mathcal{Q}_m v\|_{L^2(D)}^2 
\ge \| \partial_r (\mathcal{Q}_m v)_r \|_{L^2(D)}^2 
+ \| \partial_r (\mathcal{Q}_m v)_\theta \|_{L^2(D)}^2 
+ m^2 \big\| \frac{\mathcal{Q}_m v}{|x|} \big\|_{L^2(D)}^2\,.
\label{polar.grad.n''} 
\end{align}
%
\subsection{The Biot-Savart law in polar coordinates}\label{biotsavart}
%
For a given scalar field $\omega$ in $D$, the streamfunction $\psi$ is formally defined as the solution to the Poisson equation:  $-\Delta \psi = \omega$ in $D$ and $\psi =0$ on $\partial D$. For $n\in \Z$ and $\omega\in L^2(D)$ we set $\mathcal{P}_n \omega = \mathcal{P}_n \omega(r,\theta)$ and $\omega_n = \omega_n(r)$ as
\begin{align}
\begin{split}
\mathcal{P}_n \omega 
\,=\, 
\bigg( \frac{1}{2\pi} \int_0^{2\pi} \omega (r \cos s, r\sin s) e^{- i n s} \dd s \bigg) e^{i n \theta} \,,~~~~~~
\omega_n \,=\, \big ( \mathcal{P}_n \omega \big ) e^{-i n \theta}\,.
\label{def.w_n}
\end{split}
\end{align}
From the Poisson equation in polar coordinates, we see that each $n$-Fourier mode of $\psi$ satisfies the following ODE:
\begin{align}
-\frac{\dd \psi_n}{\dd r^2}  
- \frac{1}{r} \frac{\dd \psi_n}{\dd r} 
+ \frac{n^2}{r^2} \psi_n 
\,=\, \omega_n\,,~~~~r>1\,, ~~~~~~~~\psi_n (1) =0\,.
\label{eq.stream}
\end{align}
Let $|n|\ge 1$. Then the solution $\psi_n= \psi_n[\omega_n]$ to \eqref{eq.stream} decaying at spatial infinity is given by
\begin{align*}
\begin{split}
\psi_n[\omega_n] (r) 
& \,=\,  \frac{1}{2 |n|} 
\bigg(-\frac{d_n [\omega_n] }{r^{|n|} } 
+ \frac{1}{r^{|n|}} \int_1^r s^{1+|n|} \omega_n (s) \dd s   
+ r^{|n|}\int_r^\infty s^{1-|n|} \omega_n (s) \dd s \bigg ) \,,\\
d_n [\omega_n] & \,=\, \int_1^\infty s^{1-|n|} \omega_n (s) \dd s\,.
\end{split}
\end{align*}
The formula $V_n[\omega_n]$ in the next is called the Biot-Savart law for $\mathcal{P}_n \omega$:
\begin{align}\label{def.V_n}
\begin{split}
&V_n [\omega_n] \,=\, 
V_{r,n} [\omega_n](r) e^{i n \theta}{\bf e}_r  
+  V_{\theta,n} [\omega_n](r) e^{i n\theta} {\bf e}_\theta\,,\\
& V_{r,n} [\omega_n] \, = \, \frac{i n}{r} \psi_n [\omega_n] \,,
~~~~~~ V_{\theta,n} [\omega_n] \, = \, - \frac{\dd}{\dd r} \psi_n [\omega_n]\,.
\end{split}
\end{align}
The velocity $V_n[\omega_n]$ is well defined at least when $r^{1-|n|} \omega_n\in L^1 ((1,\infty))$, and it is straightforward to see that
\begin{align}\label{V_n}
\begin{split}
&{\rm div}\,V_n [\omega_n] \,=\, 0\,,~~~~
{\rm rot}\,V_n [\omega_n] \,=\, \mathcal{P}_n \omega~~~~~~{\rm in}~~D\,,\\
& {\bf e}_r \cdot V_n [\omega_n] \,=\,  0 ~~~~~~{\rm on}~~\partial D\,.
\end{split}
\end{align}
The condition $r^{1-|n|} \omega_n\in L^1 ((1,\infty))$ is automatically satisfied when $\omega\in L^2 (D)$ and~$|n|\geq 2$. When $|n|=1$, however, the integral in the definition of $\psi_n[\omega_n]$ does not converge absolutely for general $\omega\in L^2 ({\color{black} D})$. We can justify this integral for $|n|=1$ if $\omega$ is given in a rotation form $\omega={\rm rot}\,u$ with some $u\in W^{1,2}(D)^2$, since the integration by parts leads to the convergence of $\displaystyle \lim_{N\rightarrow \infty} \int_r^N \omega_n \dd r$. Hence, for any $v \in L^2_\sigma ({\color{black} D}) \cap W^{1,2} ({\color{black} D})^2$, the $n$-mode $v_n = \mathcal{P}_n v$ can be expressed in terms of its vorticity $\omega_n$ by the formula \eqref{def.V_n} when $|n|\geq 1$. Before closing this subsection, we mention that the definition of $f_n$ differs according to whether $f$ is a vector $f=(f_1(x), f_2(x))^\top$ or scalar $f=f(x)$ function on $D$. The vector case is defined in \eqref{def.P_n}, while the scalar case is defined in \eqref{def.w_n}.
%
\subsection{A priori resolvent estimate by energy method}
\label{apriori1}
%
In this subsection we study the energy estimate to the resolvent problem \eqref{RS}: 
\begin{equation}\tag{RS}
\left\{
\begin{aligned}
\lambda v - \Delta v + \beta U^\bot {\rm rot}\,v
+ {\rm div}\,(R\otimes v +  v\otimes R) 
+ \nabla q
&\,=\,f\,,~~~~x \in \Omega\,, \\
{\rm div}\,v &\,=\, 0\,,~~~~x \in \Omega\,, \\
v|_{\partial \Omega} &\,=\,0\,. 
\end{aligned}\right.
\end{equation}
Here $\lambda\in \C$ is the resolvent parameter, the vector field $U$ is the rotating flow of \eqref{rotatingflow} in the introduction, and $\beta$ and $R$ are defined in Assumption \ref{assumption}. The first result of this subsection is the a priori estimates to \eqref{RS} obtained by the energy method. We recall that $D$ denotes the exterior disk $\{x\in\R^2~|~|x|>1\}$, and that {\color{black} $\gamma$ is the constant in Assumption \ref{assumption}.}
%
\begin{proposition}\label{prop.general.energy.est.resol.}
Let $q\in(1,2]$, $f\in L^q(\Omega)^2$, and $\lambda\in\C$. Suppose that $v \in D(\mathbb{A}_V)$ is a solution to \eqref{RS}. 
{\color{black} Then there are constants $\beta_1\in(0,1)$ and $d_1\in(0,\frac14)$} depending only on 
{\color{black} $\Omega$ and $\gamma$} 
such that the following estimates hold.
\begin{align}
&{\rm Re}(\lambda) \| v\|_{L^2(\Omega)}^2
+ \frac34
\|\nabla v\|_{L^2(\Omega)}^2 \nonumber \\
& \quad
\le 
\beta 
\big| \sum_{|n|=1} \big\langle ({\rm rot}\,v)_{n}, \frac{v_{r,n}}{|x|} \big\rangle_{L^2(D)} \big| 
+ C \|f\|_{L^q(\Omega)}^{\frac{2q}{3q-2}} 
\|v\|_{L^2(\Omega)}^{\frac{4(q-1)}{3q-2}}\,,
\label{est1.prop.general.energy.est.resol.} \\
&|{\rm Im}(\lambda)| \| v\|_{L^2(\Omega)}^{2} 
\le 
\frac14
\|\nabla v\|_{L^2(\Omega)}^{2} 
+ \beta \big| \sum_{|n|=1} \big\langle ({\rm rot}\,v)_{n}, \frac{v_{r,n}}{|x|} \big\rangle_{L^2(D)} \big| \nonumber \\
&~~~~~~~~~~~~~~~~~~~~~~~~~~~~~~~~
+ C
\|f\|_{L^q(\Omega)}^{\frac{2q}{3q-2}} 
\|v\|_{L^2(\Omega)}^{\frac{4(q-1)}{3q-2}}\,,
\label{est2.prop.general.energy.est.resol.}
\end{align}
as long as $\beta\in(0,\beta_1)$ {\color{black} and $d\in(0, d_1)$}. The constant $C$ is independent of $\beta$
{\color{black} and $d$}.
\end{proposition}
%
\begin{proof}
Taking the inner product with $v$ to the first equation of \eqref{RS}, we find
\begin{align}
& {\rm Re}(\lambda) \| v\|_{L^2(\Omega)}^{2}
+ \| \nabla v\|_{L^2(\Omega)}^{2} \nonumber \\
&\,=\,
- \beta {\rm Re}\langle U^{\bot} {\rm rot}\,v, v \rangle_{L^2(\Omega)}
+ {\rm Re} \langle R \otimes v + v \otimes R, \nabla v \rangle_{L^2(\Omega)}
+ {\rm Re} \langle f, v \rangle_{L^2(\Omega)}\,,
\label{eq1.proof.prop.general.energy.est.resol.} \\
& {\rm Im}(\lambda) \| v\|_{L^2(\Omega)}^{2} \nonumber \\
&\,=\,
- \beta {\rm Im}\langle U^{\bot} {\rm rot}\,v, v \rangle_{L^2(\Omega)}
+ {\rm Im} \langle R \otimes v + v \otimes R, \nabla v \rangle_{L^2(\Omega)}
+ {\rm Im} \langle f, v \rangle_{L^2(\Omega)}\,.
\label{eq2.proof.prop.general.energy.est.resol.}
\end{align}
After decomposing the domain $\Omega = (\Omega \setminus D)\,\cup\,D$, from $U^{\bot}=-\frac{{\bf e_{r}}}{r}$ on $D$ we have
\begin{align}
\beta |\langle U^{\bot} {\rm rot}\,v, v \rangle_{L^2(\Omega)}| 
\le 
\beta |\langle U^{\bot} {\rm rot}\,v, v \rangle_{L^2(\Omega \setminus D)}| 
+ \beta \big|\big\langle {\rm rot}\,v, \frac{v_{r}}{|x|} \big\rangle_{L^2(D)}\big|\,.
\label{est1.proof.prop.general.energy.est.resol.} 
\end{align}
Then the Poincare inequality on $\Omega \setminus D$ implies that 
\begin{align}
\beta |\langle U^{\bot} {\rm rot}\,v, v \rangle_{L^2(\Omega \setminus D)}| 
\le
C \beta \| \nabla v \|_{L^2(\Omega)}^2\,,
\label{est2.proof.prop.general.energy.est.resol.} 
\end{align}
and by applying the Fourier series expansion on $D$, we see from \eqref{def.P_n} and \eqref{def.w_n} that
\begin{align}
&\big|\big\langle {\rm rot}\,v, \frac{v_{r}}{|x|} \big\rangle_{L^2(D)}\big|
\,=\, 
\big| \big( \sum_{|n|=1} + \sum_{n \in \Z \setminus \{ \pm1\}} \big)
\big\langle ({\rm rot}\,v)_{n}, \frac{v_{r,n}}{|x|} \big\rangle_{L^2(D)} \big| 
\nonumber \\
& \le \big| \sum_{|n|=1} \big\langle ({\rm rot}\,v)_{n}, \frac{v_{r,n}}{|x|} \big\rangle_{L^2(D)} \big| 
+ \sum_{n \in \Z \setminus \{ \pm1\}} 
\|{\rm rot}\,v_{n}\|_{L^2(D)} 
\big\|\frac{v_{r,n}}{|x|}\big\|_{L^2(D)}\,.
\label{est3.proof.prop.general.energy.est.resol.} 
\end{align}
Then the 
{\color{black} inequalities \eqref{polar.grad.n'} for $n=0$ and \eqref{polar.grad.n''} for $m=1$ ensure}
that
\begin{align}
\sum_{n \in \Z \setminus \{ \pm1\}} 
\|{\rm rot}\,v_{n}\|_{L^2(D)}
\big\|\frac{v_{r,n}}{|x|}\big\|_{L^2(D)}
& \le
C \sum_{n \in \Z \setminus \{ \pm1\}} \|\nabla v_{n}\|_{L^2(D)}^2 \nonumber \\
& \le
C \|\nabla v\|_{L^2(\Omega)}^2\,.
\label{est4.proof.prop.general.energy.est.resol.} 
\end{align}
Inserting 
\eqref{est2.proof.prop.general.energy.est.resol.}--\eqref{est4.proof.prop.general.energy.est.resol.}  into \eqref{est1.proof.prop.general.energy.est.resol.} we obtain
\begin{align}
\beta |\langle U^{\bot} {\rm rot}\,v, v \rangle_{L^2(\Omega)}| 
\le 
C_1 \beta \|\nabla v\|_{L^2(\Omega)}^2
+ \beta 
\big| \sum_{|n|=1} \big\langle ({\rm rot}\,v)_{n}, \frac{v_{r,n}}{|x|} \big\rangle_{L^2(D)} \big|\,.
\label{est5.proof.prop.general.energy.est.resol.} 
\end{align}
Next by \eqref{remainder} in Assumption \ref{assumption} we have
\begin{align}
|\langle R \otimes v + v \otimes R, \nabla v \rangle_{L^2(\Omega)}| 
& \le
{\color{black} C_2 d}
\,\| \nabla v\|_{L^2(\Omega)}^2\,,
\label{est6.proof.prop.general.energy.est.resol.} 
\end{align}
where the inequality $\||x|^{-(1+\gamma)} v\|_{L^2} \le C \| \nabla v \|_{L^2}$ is applied. The constant $C_2$ depends on $\gamma\in(0,1)$. By the Gagliardo-Nirenberg inequality we see that for $q\in(1,2]$ and $q'= \frac{q}{q-1}$,
\begin{align}
|\langle f, v \rangle_{L^2(\Omega)}|
& \le C \|f\|_{L^q(\Omega)} \|u\|_{L^{q'}(\Omega)} \nonumber \\
& \le C \|f\|_{L^q(\Omega)} \|u\|_{L^2(\Omega)}^{2(1-\frac1q)}
\|\nabla u\|_{L^2(\Omega)}^{\frac2q-1} \nonumber \\
& \le C \|f\|_{L^q(\Omega)}^{\frac{2q}{3q-2}} 
\|u\|_{L^2(\Omega)}^{\frac{4(q-1)}{3q-2}}
+ \frac18 \|\nabla u\|_{L^2(\Omega)}^{2}\,,
\label{est7.proof.prop.general.energy.est.resol.} 
\end{align}
where the Young inequality is applied in the last line. 
{\color{black}
Now we take $\beta_1\in(0,1)$ and $d_1\in(0,\frac14)$ small enough so that
\begin{align}
C_1 \beta_1 + C_2 d_1
\le \frac18
\label{est8.proof.prop.general.energy.est.resol.} 
\end{align}
holds}. Then the assertions 
\eqref{est1.prop.general.energy.est.resol.} and 
\eqref{est2.prop.general.energy.est.resol.} are proved by inserting 
\eqref{est5.proof.prop.general.energy.est.resol.}--\eqref{est7.proof.prop.general.energy.est.resol.}
into \eqref{eq1.proof.prop.general.energy.est.resol.} and \eqref{eq2.proof.prop.general.energy.est.resol.}, and using the condition 
\eqref{est8.proof.prop.general.energy.est.resol.}. This completes the proof.
\end{proof}
%

\noindent As can be seen from Proposition \ref{prop.general.energy.est.resol.}, the key object in closing the energy computation is to derive the estimate for the next term appearing in the right-hand sides of \eqref{est1.prop.general.energy.est.resol.} and \eqref{est2.prop.general.energy.est.resol.}:
\begin{align*}
\big| \sum_{|n|=1} \big\langle ({\rm rot}\,v)_{n}, \frac{v_{r,n}}{|x|} \big\rangle_{L^2(D)} \big|\,.
\end{align*}
Note that the Hardy inequality in polar coordinates \eqref{polar.grad.n'} cannot be applied to this term. The next proposition shows that this term can be handled if $\lambda$ in \eqref{RS} satisfies $|\lambda|\ge O(\beta^2 e^{-\frac{1}{6\beta}})$.
%
\begin{proposition}\label{prop.laege.lambda.energy.est.resol.} 
{\color{black}
Let $\beta_1$ and $d_1$ be the constants in Proposition \ref{prop.general.energy.est.resol.}, 
and let $d \in (0, d_1)$. 
}
Then the following statements hold.

\noindent {\rm (1)} Fix a positive number $\beta_2\in(0,\min\{\frac1{12},\beta_1\})$. Then the set
\begin{align}\label{set.prop.laege.lambda.energy.est.resol.} 
\mathcal{S}_\beta \,=\,
\big\{{\color{black} \lambda} \in \C~|~
|{\rm Im}(\lambda)| > -{\rm Re}(\lambda) + 12 e^{\frac{1}{e}} \beta^2 e^{-\frac{1}{6\beta}}
\big\}
\end{align}
is included in the resolvent $\rho(-\mathbb{A}_V)$ for any $\beta\in (0,\beta_2)$. 

\noindent {\rm (2)} Let $q\in(1,2]$ and $f\in L^2_\sigma(\Omega) \cap L^q(\Omega)^2$. Then we have
\begin{equation}\label{est.laege.lambda.energy.est.resol.} 
\begin{split}
\|(\lambda+{\mathbb A}_V)^{-1} f\|_{L^2(\Omega)} 
& \le 
C |\lambda|^{-\frac32+\frac1q}
\|f\|_{L^q(\Omega)}\,,~~~~~~
\lambda\in
{\color{black} \mathcal{S}_\beta \cap \mathcal{B}_{\frac12 e^{-\frac{1}{6\beta}}}(0)^{{\rm c}}}\,, \\
\|\nabla (\lambda+{\mathbb A}_V)^{-1} f\|_{L^2(\Omega)} 
& \le 
C |\lambda|^{-1+\frac1q}
\|f\|_{L^q(\Omega)}\,,~~~~~~
\lambda\in
{\color{black} \mathcal{S}_\beta \cap \mathcal{B}_{\frac12 e^{-\frac{1}{6\beta}}}(0)^{{\rm c}}}\,,
\end{split}
\end{equation}
as long as $\beta\in (0,\beta_2)$. {\color{black} Here} the constant $C$ is independent of $\beta$ 
{\color{black} and $d$}, 
{\color{black} and $\mathcal{B}_\rho(0)\subset\C$ denotes the disk centered at the origin with radius $\rho>0$.}
\end{proposition}
%
\begin{proof}
(1) Let us denote the function space $L^q(\Omega)$ by $L^q$ in this proof to simplify notation. Let $|n|=1$, $\beta\in (0,\beta_2)$, and $v \in D(\mathbb{A}_V)$ solve \eqref{RS}. Define a function $\Theta=\Theta(T)$ by
\begin{align}
\Theta(T) \,=\, \int_0^T \frac{1}{\tau} e^{-\frac{1}{\tau}} \dd \tau\,,~~~~T>e\,,
\label{ThetaT}
\end{align}
which satisfies the following lower and upper bounds:
\begin{align}
e^{-\frac{1}{e}} \log T \le \Theta(T) \le \log T\,,~~~~T>e\,,
\label{est.ThetaT}
\end{align}
which can be easily checked. Then, as is shown in \cite[Lemma 3.26]{Ma1}, we have
\begin{align}
\beta \big|\big\langle ({\rm rot}\,v)_{n}, \frac{v_{r,n}}{|x|} \big\rangle_{L^2(D)}\big|
\le
\frac{\beta}{T} \| v\|_{L^2}\ \| \nabla v\|_{L^2}
+ \beta \Theta(T)
\| \nabla v\|_{L^2}^2\,,~~~~T>e\,.
\label{est.logT}
\end{align}
The proof is done by extending $v \in D(\mathbb{A}_V)$ by zero to the whole space $\R^2$ and using the nondegenerate condition $\{x\in\R^2~|~|x|\le \frac12\} \subset \Omega^{\rm c}$. Then the Young inequality yields
\begin{align}\label{est.T}
\frac{\beta}{T} \| v\|_{L^2} \| \nabla v\|_{L^2} 
& \le
\frac{\beta \Theta(T)}{2} \| \nabla v\|_{L^2}^2
+ \frac{\beta}{2T^2 \Theta(T)} \| v\|_{L^2}^2\,.
\end{align}
Inserting \eqref{est.logT} and \eqref{est.T} into 
\eqref{est1.prop.general.energy.est.resol.} and \eqref{est2.prop.general.energy.est.resol.} in Proposition \ref{prop.general.energy.est.resol.}, we see that
\begin{align}
&\big( {\rm Re}(\lambda) 
- \frac{\beta}{2T^2 \Theta(T)} \big) 
\|v\|_{L^2}^2
+ (\frac34 - \frac{3\beta \Theta(T)}{2}) \| \nabla v\|_{L^2}^2 
\le C \|f\|_{L^q}^{\frac{2q}{3q-2}}
\|v\|_{L^2}^{\frac{4(q-1)}{3q-2}}\,,
\label{est1.prop1.laege.lambda.energy.est.resol.} \\
& \big( |{\rm Im}(\lambda)| 
- \frac{\beta}{2T^2 \Theta(T)} \big) 
\| v\|_{L^2}^2 
\le 
(\frac14 + \frac{3\beta \Theta(T)}{2})
\| \nabla v\|_{L^2}^{2} 
+ C \|f\|_{L^q}^{\frac{2q}{3q-2}} 
\|v\|_{L^2}^{\frac{4(q-1)}{3q-2}}\,.
\label{est2.prop1.laege.lambda.energy.est.resol.}
\end{align}
Then \eqref{est1.prop1.laege.lambda.energy.est.resol.} and \eqref{est2.prop1.laege.lambda.energy.est.resol.} lead to
\begin{align}
&\big(|{\rm Im}(\lambda)| + {\rm Re}(\lambda) - \frac{\beta}{T^2 \Theta(T)} \big)
\|v\|_{L^2}^2
+ (\frac12 - 3\beta \Theta(T)) \| \nabla v\|_{L^2}^{2} \nonumber \\
& \le
C \|f\|_{L^q}^{\frac{2q}{3q-2}}
\|v\|_{L^2}^{\frac{4(q-1)}{3q-2}}\,.
\label{est3.prop1.laege.lambda.energy.est.resol.}
\end{align}
Now let us take $T=e^{\frac{1}{12\beta}}$. Since $T>e$ by the condition $\beta\in(0,\frac{1}{12})$, from \eqref{est.ThetaT} we have 
\begin{align}
3\beta \Theta(T) \le 3  \beta \log T = \frac14~~~~~~
{\rm and}~~~~~~
\frac{\beta}{T^2 \Theta(T)} 
\le 
\frac{e^{\frac{1}{e}} \beta}{T^2 \log T} 
\,=\,
12 e^{\frac{1}{e}} \beta^2 e^{-\frac{1}{6\beta}}\,.
\label{est4.prop1.laege.lambda.energy.est.resol.}
\end{align}
By inserting \eqref{est4.prop1.laege.lambda.energy.est.resol.} into \eqref{est3.prop1.laege.lambda.energy.est.resol.} we obtain the assertion $\mathcal{S}_\beta \subset \rho(-\mathbb{A}_V)$. \\
\noindent (2) Let {\color{black} $\lambda\in \mathcal{S}_\beta \cap \mathcal{B}_{\frac12 e^{-\frac{1}{6\beta}}}(0)^{{\rm c}}$. If additionally $\lambda\in\{z\in\C~|~ {\rm Re}(z) < 0\}$ then we have}
\begin{align*}
|{\rm Im}(\lambda)| \ge \frac{\beta}{T^2 \Theta(T)}~~~~~~
{\rm and}~~~~~~
|{\rm Im}(\lambda)| \le |\lambda| \le \sqrt{2} |{\rm Im}(\lambda)|\,.
\end{align*}
Then we see from \eqref{est2.prop1.laege.lambda.energy.est.resol.} and \eqref{est3.prop1.laege.lambda.energy.est.resol.} that, 
\begin{align}
|\lambda| \| v\|_{L^2}^2
& \le 
\frac{6\sqrt{2}}{8} \| \nabla v\|_{L^2}^{2} 
+ 2\sqrt{2} C \|f\|_{L^q}^{\frac{2q}{3q-2}} 
\|v\|_{L^2}^{\frac{4(q-1)}{3q-2}}\,,
\label{est5.prop1.laege.lambda.energy.est.resol.} \\
\| \nabla v\|_{L^2}^{2} 
&\le 4C \|f\|_{L^q}^{\frac{2q}{3q-2}} \|v\|_{L^2}^{\frac{4(q-1)}{3q-2}}\,,
\label{est6.prop1.laege.lambda.energy.est.resol.}
\end{align}
where the constant $C$ is independent of $\beta$
{\color{black} and $d$}. 
{\color{black} On the other hand, if additionally $\lambda\in\{z\in\C~|~ {\rm Re}(z) \ge 0\}$ then we have from \eqref{est4.prop1.laege.lambda.energy.est.resol.}, 
\begin{align*}
|{\rm Im}(\lambda)| + {\rm Re}(\lambda) - \frac{\beta}{T^2 \Theta(T)}
\ge 
|\lambda| - 12 e^{\frac{1}{e}} \beta^2 e^{-\frac{1}{6\beta}} 
\ge \frac{|\lambda|}{2}\,,
\end{align*}
since $12 e^{\frac{1}{e}} \beta^2 e^{-\frac{1}{6\beta}} \le 24 e^{\frac{1}{e}} \beta^2 |\lambda| \le \frac{|\lambda|}{2}$ holds by $\beta\in(0,\frac{1}{12})$. Then from \eqref{est3.prop1.laege.lambda.energy.est.resol.} we see that,
\begin{align}
|\lambda| \| v\|_{L^2}^2 + \frac12 \| \nabla v\|_{L^2}^{2}
& \le 
2C \|f\|_{L^q}^{\frac{2q}{3q-2}} \|v\|_{L^2}^{\frac{4(q-1)}{3q-2}}
\label{est7.prop1.laege.lambda.energy.est.resol.}\,,
\end{align}
where the constant $C$ is independent of $\beta$} 
{\color{black} and $d$}.
The estimates in \eqref{est.laege.lambda.energy.est.resol.} follow from
{\color{black} \eqref{est5.prop1.laege.lambda.energy.est.resol.} and  \eqref{est6.prop1.laege.lambda.energy.est.resol.}, and \eqref{est7.prop1.laege.lambda.energy.est.resol.}.} This completes the proof of Proposition \ref{prop.laege.lambda.energy.est.resol.}.
\end{proof}
%
\section{Resolvent analysis in region exponentially close to the origin}
\label{sec.RSed}
%
The resolvent analysis in Proposition \ref{prop.laege.lambda.energy.est.resol.} is applicable to the problem \eqref{RS} only when the resolvent parameter $\lambda\in\C$ satisfies $|\lambda|\ge e^{-\frac{1}{a\beta}}$ for some $a\in(1,\infty)$, and we have taken $a=6$ in the proof for simplicity. This restriction is essentially due to the unavailability of the Hardy inequality in two-dimensional exterior domains. In fact, in the proof of Proposition \ref{prop.laege.lambda.energy.est.resol.}, we rely on the following inequality singular in $T\gg 1$:
\begin{align*}
\big| \sum_{|n|=1} \big\langle ({\rm rot}\,v)_{n}, \frac{v_{r,n}}{|x|} \big\rangle_{L^2(D)} \big|
\le
\frac{1}{T} \| v\|_{L^2(\Omega)}\ \| \nabla v\|_{L^2(\Omega)}
+ \log T \| \nabla v\|_{L^2(\Omega)}^2\,,
\end{align*}
as a substitute for the Hardy inequality, and this leads to the lack of information about the spectrum of $-\mathbb{A}_V$ in the region $0<|\lambda| \le O(e^{-\frac{1}{\beta}})$. Here we set $D=\{x\in\R^2~|~|x|>1\}$.

To perform the resolvent analysis in the region exponentially close to the origin, we firstly observe that a solution $(v, q)$ to \eqref{RS} satisfies the next problem in the exterior disk $D$:
\begin{equation}\tag{RS$^{\rm ed}$}\label{RSed}
\left\{
\begin{aligned}
\lambda w - \Delta w + \beta U^\bot {\rm rot}\,w + \nabla r
&\,=\,(-{\rm div}\,(R\otimes v + v\otimes R) + f )|_{D}\,,~~~~x\in D\,, \\
{\rm div}\,w &\,=\, 0\,,~~~~x\in D\,, \\
w|_{\partial D} &\,=\,v|_{\partial D}\,.
\end{aligned}\right.
\end{equation}
Then thanks to the symmetry, we can use a solution formula to \eqref{RSed} by using polar coordinates, and study the a priori estimate for $w=v|_{D}$. To make calculation simple, we decompose the linear problem \eqref{RSed} into three parts \eqref{RSed.f}, \eqref{RSed.divF}, and \eqref{RSed.b}, which are respectively introduced in Subsections \ref{subsec.RSed.f}, \ref{subsec.RSed.divF}, and \ref{subsec.RSed.b}. Then we derive the estimates to each problem in the corresponding subsections, and finally we collect them in Subsection \ref{apriori2} in order to establish the resolvent estimate to \eqref{RS} when $0<|\lambda|<e^{-\frac{1}{6\beta}}$. 
\subsection{Problem I: External force $f$ and Dirichlet condition}
\label{subsec.RSed.f}
%
In this subsection we study the following resolvent problem for $(w,r)=(w^{\rm ed}_{f}, r^{\rm ed}_{f})$:
\begin{equation}\tag{RS$^{\rm ed}_{f}$}\label{RSed.f}
\left\{
\begin{aligned}
\lambda w - \Delta w + \beta U^{\bot} {\rm rot}\,w + \nabla r
& \,=\, 
f\,,~~~~x \in D\,, \\
{\rm div}\,w &\,=\, 0\,,~~~~x \in D\,, \\
w|_{\partial D} & \,=\, 0\,.
\end{aligned}\right.
\end{equation}
Especially, we are interested in the estimates for the $\pm 1$-Fourier mode of $w^{\rm ed}_{f}$. Although the $L^p\mathchar`-L^q$ estimates to \eqref{RSed.f} are already proved in \cite{Ma1}, we revisit this problem here in order to study the $\beta$-dependence in these estimates, and it is one of the most important steps for the energy computation when $0<|\lambda|<e^{-\frac{1}{6\beta}}$.

Let us recall the representation formula established in \cite{Ma1} for the solution to \eqref{RSed.f} in each Fourier mode. Fix $n \in \Z\setminus\{0\}$ and $\lambda\in\C\setminus \overline{\R_{-}}$, $\,\overline{\R_{-}}=(-\infty,0]$. Then, by applying the Fourier mode projection $\mathcal{P}_n$ to \eqref{RSed.f} and using the invariant property $\mathcal{P}_n(U^\bot {\rm rot}\,w)=U^\bot {\rm rot}\,\mathcal{P}_n w$ in \cite[Lemma 2.9]{Ma1}, we observe that the $n$-mode $w_n=\mathcal{P}_n w$ solves
\begin{equation}\tag{RS$^{\rm ed}_{f,n}$}\label{nFourier.RSed.f}
\left\{
\begin{aligned}
\lambda w_n - \Delta w_n + \beta U^{\bot} {\rm rot}\,w_n + \mathcal{P}_n \nabla r
& \,=\, 
{\color{black} \mathcal{P}_n f}
\,,~~~~x \in D\,, \\
{\rm div}\,w_n &\,=\, 0\,,~~~~x \in D\,, \\
w_n|_{\partial D} & \,=\, 0\,.
\end{aligned}\right.
\end{equation}
Since the formula in \cite{Ma1} is written in terms of some special functions, we introduce these definitions here. The modified Bessel function of first kind $I_\mu(z)$ of order $\mu$ is defined as
\begin{align}\label{def.I}
I_\mu(z) 
\,=\, 
\big(\frac{z}{2}\big)^\mu 
\sum_{m=0}^{\infty} \frac{1}{m!\,\Gamma(\mu+m+1)} \big(\frac{z}{2}\big)^{2m}\,,
~~~~~~ z \in \C\setminus \overline{\R_{-}}\,,
\end{align}
where $z^\mu = e^{\mu {\rm Log}\,z}$ and ${\rm Log}\,z$ denotes the principal branch to the logarithm of $z\in\C\setminus \overline{\R_{-}}$, and the function $\Gamma(z)$ in \eqref{def.I} denotes the Gamma function. Next we define the modified Bessel function of second kind $K_\mu(z)$ of order $\mu\notin\Z$ in the following manner:
\begin{align}\label{def.K}
K_\mu(z) 
\,=\, 
\frac{\pi}{2}\,\frac{I_{-\mu}(z) - I_\mu(z)}{\sin{\mu \pi}}\,,~~~~~~ 
z \in \C\setminus \overline{\R_{-}}\,.
\end{align}
It is classical that $K_\mu(z)$ and $I_\mu(z)$ are linearly independent solutions to the ODE
\begin{align}\label{ode.bessel.func}
-\frac{\dd^2 \omega}{\dd z^2} 
- \frac{1}{z} \frac{\dd \omega}{\dd z} 
+ \big(1+\frac{\mu^2}{z^2} \big) \omega \,=\, 0\,,
\end{align}
and that 
{\color{black} their Wronskian} 
is $z^{-1}$. Applying the rotation operator ${\rm rot}$ to the first equation of \eqref{nFourier.RSed.f}, we find that $\omega=({\rm rot}\,w)_n=({\rm rot}\,w_n)e^{-in\theta}$ satisfies the ODE
\begin{align}\label{ode.nfourier.vorticity}
-\frac{\dd^2 \omega}{\dd r^2} - \frac{1}{r} \frac{\dd \omega}{\dd r} 
+ \big(\lambda + \frac{n^2+in\beta}{r^2} \big) \omega 
\,=\, ({\rm rot}\,f)_n\,,~~~~~~
r>1\,.
\end{align}
Hence, if we set 
\begin{align}\label{def.mu}
\mu_n \,=\, \mu_n(\beta) \,=\, (n^2+in\beta)^\frac12\,,~~~~~~
{\rm Re}(\mu_n)>0\,,
\end{align}
then $K_{\mu_n}(\sqrt{\lambda} r)$ and $I_{\mu_n}(\sqrt{\lambda} r)$ give linearly independent solutions to the homogeneous equation of \eqref{ode.nfourier.vorticity} and their Wronskian is $r^{-1}$. Here and in the following we always take the square root $\sqrt{z}$ so that ${\rm Re}(\sqrt{z})>0$ for $z\in\C\setminus \overline{\R_{-}}$. Furthermore, we set 
\begin{align}\label{def.F}
F_n(\sqrt{\lambda};\beta) 
\,=\,  \int_1^\infty s^{1-|n|} K_{\mu_n}(\sqrt{\lambda} s) \dd s\,,~~~~~~~~
\lambda\in\C\setminus \overline{\R_{-}}\,,
\end{align}
and denote by $\mathcal{Z}(F_n)$ the set of the zeros of $F_n(\sqrt{\lambda};\beta)$ lying in $\C\setminus\overline{\R_{-}}$;
\begin{align}\label{def.zeros.F}
\mathcal{Z}(F_n) \,=\, \{z\in\C\setminus\overline{\R_{-}}~|~ 
F_n(
{\color{black}
\sqrt{z}
}
;\beta)\,=\,0 \}\,.
\end{align}
Let $\lambda\in\C\setminus (\overline{\R_{-}} \cup \mathcal{Z}(F_n))$. Then, from the argument in \cite[Section 3]{Ma1}, we have the following representation formula for  $w^{\rm ed}_{f,n}$ solving \eqref{nFourier.RSed.f}:
\begin{align}\label{rep.velocity.nFourier.RSed.f}
& w^{\rm ed}_{f,n}
\,=\,
-\frac{c_{n,\lambda}[f_n]}{F_n(\sqrt{\lambda};\beta)} 
V_n [K_{\mu_n}(\sqrt{\lambda}\,\cdot\,)]
+ V_n[\Phi_{n,\lambda}[f_n]]\,.
\end{align}
Here $V_n[\,\cdot\,]$ is the Biot-Savart law in \eqref{def.V_n} and the function $\Phi_{n,\lambda}[f_n]$ is defined as
\begin{equation}\label{Phi.nFourier.RSed.f}
\begin{aligned}
\Phi_{n,\lambda}[f_n](r) 
&\,=\, 
-K_{\mu_n}(\sqrt{\lambda} r)
\bigg( \int_{1}^{r} I_{\mu_n}(\sqrt{\lambda} s)\,
\big( \mu_n f_{\theta,n}(s) + in f_{r,n}(s) \big) \dd s \\
& ~~~~~~~~~~~~~~~~~~~~~~~~~~~~~~~~~~~~
+ \sqrt{\lambda} \int_{1}^{r} sI_{\mu_n+1}(\sqrt{\lambda} s)\,f_{\theta,n}(s) \dd s
\bigg) \\
& \quad
+ I_{\mu_n}(\sqrt{\lambda} r)
\bigg( \int_{r}^{\infty} K_{\mu_n}(\sqrt{\lambda} s)\,
\big( \mu_n f_{\theta,n}(s) - in f_{r,n}(s) \big) \dd s \\
& ~~~~~~~~~~~~~~~~~~~~~~~~~~~~~~~~~~~~
+ \sqrt{\lambda} \int_{r}^{\infty} sK_{\mu_n-1}(\sqrt{\lambda} s)\,f_{\theta,n}(s) \dd s
\bigg) \,,
\end{aligned}
\end{equation}
while the constant $c_{n,\lambda}[f_n]$ is defined as
\begin{align}\label{const.nFourier.RSed.f}
c_{n,\lambda}[f_n] 
\,=\, 
\int_1^\infty s^{1-|n|} \Phi_{n,\lambda}[f_n](s) \dd s\,.
\end{align}
Moreover, the vorticity ${\rm rot}\,w^{\rm ed}_{f,n}$ is represented as
\begin{align}\label{rep.vorticity.nFourier.RSed.f}
{\rm rot}\,w^{\rm ed}_{f,n}
\,=\, 
-\frac{c_{n,\lambda}[f_n]}{F_n(\sqrt{\lambda};\beta)} 
K_{\mu_n}(\sqrt{\lambda} r) e^{in\theta}
+ \Phi_{n,\lambda}[f_n](r) e^{in\theta}\,.
\end{align}
We shall estimate $w^{\rm ed}_{f,n}$ and ${\rm rot}\,w^{\rm ed}_{f,n}$, represented respectively as in \eqref{rep.velocity.nFourier.RSed.f} and \eqref{rep.vorticity.nFourier.RSed.f}, when $|n|=1$ in the following two subsections. Our main tools for the proof are the asymptotic analysis of $\mu_n=\mu_n(\beta)$ for small $\beta$ in Appendix \ref{app.est.mu}, and the detailed estimates to the modified Bessel functions in Appendix \ref{app.est.bessel}. Before going into details, let us state the estimate of $F_n(\sqrt{\lambda};\beta)$ in a region exponentially close to the origin with respect to $\beta$. We denote by $\Sigma_{\phi}$ the sector $\{z\in\C\setminus\{0\}~|~|{\rm arg}\,z|<\phi\}$, $\phi\in(0,\pi)$, in the complex plane $\C$, and by $\mathcal{B}_\rho(0)\subset\C$ the disk centered at the origin with radius $\rho>0$. 
%
\begin{proposition}\label{prop.est.F}
Let $|n|=1$. Then for any $\epsilon \in (0,\frac{\pi}{2})$ there is a positive constant $\beta_0$ depending only on $\epsilon$ such that as long as $\beta\in (0,\beta_0)$ and $\lambda \in \Sigma_{\pi-\epsilon}\cap\mathcal{B}_{e^{-\frac{1}{6\beta}}}(0)$ we have
\begin{align}\label{est1.prop.est.F} 
\frac{1}{|F_n(\sqrt{\lambda};\beta)|} 
\le C |\lambda|^{\frac{{\rm Re}(\mu_n)}{2}}\,,
\end{align}
where the constant $C$ depends only on $\epsilon$. In particular, we have
$\mathcal{Z}(F_n)\cap\mathcal{B}_{e^{-\frac{1}{6\beta}}}(0)=\emptyset$.
\end{proposition}
\begin{proof}
The assertion \label{est1.prop.est.F} follows from Lemma \ref{lem.est.F} in Appendix \ref{app.proof.prop.est.F}, since we have $e^{-\frac{1}{6\beta}}<\beta^4$ for any $\beta\in(0,1)$. See Appendix \ref{app.proof.prop.est.F} for the proof of Lemma \ref{lem.est.F}.
\end{proof}
%
\subsubsection{Estimates of the velocity solving \eqref{nFourier.RSed.f} with $|n|=1$}
\label{subsec.RSed.f.velocity}
%
In this subsection we derive the estimates for the solution $w^{\rm ed}_{f,n}$ to \eqref{nFourier.RSed.f} which is now represented as \eqref{rep.velocity.nFourier.RSed.f}. The novelty of the following result is the investigation on the $\beta$-singularity appearing in each estimate. Let $\beta_0$ be the constant in Proposition \ref{prop.est.F}.
%
\begin{theorem}\label{thm.est.velocity.RSed.f}
Let $|n|=1$ and $1\le q<p\le\infty$ or $1<q\le p<\infty$. Fix $\epsilon\in(0,\frac{\pi}{2})$. Then there is a positive constant $C=C(q,p,\epsilon)$ independent of $\beta$ such that the following statement holds. Let $f\in C^\infty_0(D)^2$ and $\beta\in(0,\beta_0)$. Then for $\lambda\in\Sigma_{\pi-\epsilon} \cap \mathcal{B}_{e^{-\frac{1}{6\beta}}}(0)$ we have 
\begin{align}
\|w^{\rm ed}_{f,n}\|_{L^p(D)}
& \le \frac{C}{\beta^2}
|\lambda|^{-1+\frac1q-\frac1p}
\|f\|_{L^q(D)}\,, 
\label{est1.thm.est.velocity.RSed.f} \\
\big\|\frac{w^{\rm ed}_{f,n}}{|x|}\big\|_{L^2(D)}
& \le \frac{C}{\beta}
\Big(\frac{1}{\beta^2}
+ |\log {\rm Re}(\sqrt{\lambda})|^{\frac12} \Big)
|\lambda|^{-1+\frac1q}
\|f\|_{L^q(D)}\,.
\label{est2.thm.est.velocity.RSed.f} 
\end{align}
Moreover, \eqref{est1.thm.est.velocity.RSed.f} and \eqref{est2.thm.est.velocity.RSed.f} hold all for $f\in L^q(D)^2$. 
\end{theorem}
%
\begin{remark}\label{rem.thm.est.velocity.RSed.f}
The logarithmic factor $|\log {\rm Re}(\sqrt{\lambda})|$ in \eqref{est2.thm.est.velocity.RSed.f} cannot be removed in our analysis. This singularity might prevent us from closing the energy computation in view of  the scaling, however, we observe that it is resolved by considering the following products: 
\begin{align*}
\big|\big\langle \omega^{{\rm ed}\,(1)}_{f,n}, \frac{(w^{\rm ed}_{f,r})_n}{|x|} \big\rangle_{L^2(D)} \big|\,,~~~~
\big|\big\langle \omega^{\rm ed\,(1)}_{{\rm div}F,n}, \frac{(w^{\rm ed}_{f,r})_n}{|x|} \big\rangle_{L^2(D)} \big|\,,~~~~
\big|\big\langle \omega^{\rm ed}_{b,n}, \frac{(w^{\rm ed}_{f,r})_n}{|x|} \big\rangle_{L^2(D)} \big|\,.
\end{align*}
Here the vorticities $\omega^{\rm ed\,(1)}_{f,n}$, $\omega^{\rm ed\,(1)}_{{\rm div}F,n}$, $\omega^{\rm ed}_{b,n}$ will be introduced respectively 
in Subsections \ref{subsec.RSed.f.vorticity}, \ref{subsec.RSed.divF.vorticity}, and \ref{subsec.RSed.b}. This is indeed a key observation in proving Proposition \ref{prop1.small.lambda.energy.est.resol.} in Subsection \ref{apriori2}, where the estimate for $\big\langle ({\rm rot}\,v)_{n}, 
\frac{v_{r,n}}{|x|} \big\rangle_{L^2(D)}$ is established when $0<|\lambda|<e^{-\frac{1}{6\beta}}$.
\end{remark}
%

We postpone the proof of Theorem \ref{thm.est.velocity.RSed.f} at the end of this subsection, and focus on the term $V_n[\Phi_{n,\lambda}[f_n]]$ in \eqref{rep.velocity.nFourier.RSed.f} for the time being. In order to estimate $V_n[\Phi_{n,\lambda}[f_n]]$, taking into account the definition of $V_n[\,\cdot\,]$ in \eqref{def.V_n}, firstly we study the following two integrals
\begin{align*}
\frac{1}{r^{|n|}} \int_1^r s^{1+|n|} \Phi_{n,\lambda}[f_n]{\color{black} (s)} \dd s\,,~~~~~~~~
r^{|n|} \int_r^\infty s^{1-|n|} \Phi_{n,\lambda}[f_n](s) \dd s\,.
\end{align*}
Let us recall the decompositions for them used in \cite{Ma1} which are useful in calculations. To state the result we define the functions $g^{(1)}_n(r)$ and $g^{(2)}_n(r)$ by
\begin{align*}
g^{(1)}_n(r) \,=\, \mu_n f_{\theta,n}(r) + in f_{r,n}(r)\,, ~~~~~~
g^{(2)}_n(r) \,=\, \mu_n f_{\theta,n}(r) - in f_{r,n}(r)\,, 
\end{align*}
and fix a resolvent parameter $\lambda\in\C\setminus\overline{\R_{-}}$.
%
\begin{lemma}[ {\rm\cite[Lemmas 3.6 and 3.9]{Ma1}}]\label{lem.velocitydecom.RSed.f} 
Let $n \in \Z \setminus \{0\}$ and $f\in C^\infty_0(D)^2$. Then we have
\begin{align*}
\frac{1}{r^{|n|}} \int_1^r s^{1+|n|} \Phi_{n,\lambda}[f_n]{\color{black} (s)} \dd s
\,=\, \sum_{l=1}^{9} J^{(1)}_l[f_n](r)\,,
\end{align*}
where 
{\allowdisplaybreaks
\begin{align*}
J^{(1)}_1[f_n](r) &\,=\, 
-\frac{1}{r^{|n|}} \int_1^r I_{\mu_n}(\sqrt{\lambda} \tau)\,g^{(1)}_n(\tau) 
\int_\tau^r s^{1+|n|} K_{\mu_n}(\sqrt{\lambda} s) \dd s \dd \tau\,, \\
J^{(1)}_2[f_n](r) &\,=\, -\frac{\mu_n+|n|}{r^{|n|}} \int_1^r \tau I_{\mu_n+1}(\sqrt{\lambda} \tau)\,f_{\theta,n}(\tau) \int_\tau^r s^{|n|} K_{\mu_n-1}(\sqrt{\lambda} s) \dd s \dd \tau\,, \\
J^{(1)}_3[f_n](r) &\,=\, \frac{1}{r^{|n|}} \int_1^r K_{\mu_n}(\sqrt{\lambda} \tau)\,g^{(2)}_n(\tau)  \int_1^\tau s^{1+|n|} I_{\mu_n}(\sqrt{\lambda} s) \dd s\,, \\
J^{(1)}_4[f_n](r) &\,=\, \frac{\mu_n-|n|}{r^{|n|}} \int_1^r \tau K_{\mu_n-1}(\sqrt{\lambda}\tau)\,f_{\theta,n}(\tau) \int_1^\tau s^{|n|} I_{\mu_n+1}(\sqrt{\lambda} s) \dd s\,, \\
J^{(1)}_5[f_n](r) &\,=\, \frac{1}{r^{|n|}} \bigg(\int_r^\infty K_{\mu_n}(\sqrt{\lambda}s)\,g^{(2)}_n(s) \dd s \bigg) 
\bigg(\int_1^r s^{1+|n|} I_{\mu_n}(\sqrt{\lambda} s) \dd s \bigg)\,, \\
J^{(1)}_6[f_n](r) &\,=\, \frac{\mu_n-|n|}{r^{|n|}} \bigg(\int_r^\infty sK_{\mu_n-1}(\sqrt{\lambda}s)\,f_{\theta,n}(s) \dd s \bigg) 
\bigg(\int_1^r s^{|n|} I_{\mu_n+1}(\sqrt{\lambda} s) \dd s \bigg)\,, \\
J^{(1)}_7[f_n](r) &\,=\, rK_{\mu_n-1}(\sqrt{\lambda}r) \int_1^r 
sI_{\mu_n+1}(\sqrt{\lambda} s)\,f_{\theta,n}(s) \dd s\,, \\
J^{(1)}_8[f_n](r) &\,=\, rI_{\mu_n+1}(\sqrt{\lambda}r) \int_r^\infty sK_{\mu_n-1}(\sqrt{\lambda} s)\,f_{\theta,n}(s) \dd s\,, \\
J^{(1)}_9[f_n](r) &\,=\, -\frac{I_{\mu_n+1}(\sqrt{\lambda})}{r^{|n|}} \int_1^\infty sK_{\mu_n-1}(\sqrt{\lambda} s)\,f_{\theta,n}(s) \dd s\,,
\end{align*}
}
and
\begin{align*}
r^{|n|} \int_r^\infty s^{1-|n|} \Phi_{n,\lambda}[f_n](s) \dd s
\,=\, \sum_{l=10}^{17} J^{(1)}_l[f_n](r)\,,
\end{align*}
where 
{\allowdisplaybreaks
\begin{align*}
J^{(1)}_{10}[f_n](r) &\,=\, -r^{|n|} \bigg(\int_1^r I_{\mu_n}(\sqrt{\lambda} s)\,g^{(1)}_n(s) \dd s
\bigg) \bigg(\int_r^\infty s^{1-|n|} K_{\mu_n}(\sqrt{\lambda} s) \dd s\bigg)\,, \\
J^{(1)}_{11}[f_n](r) &\,=\, -r^{|n|} \int_r^\infty I_{\mu_n}(\sqrt{\lambda} \tau)\,g^{(1)}_n(\tau) \int_\tau^\infty s^{1-|n|} K_{\mu_n}(\sqrt{\lambda} s) \dd s \dd \tau\,, \\
J^{(1)}_{12}[f_n](r) &\,=\, -(\mu_n-|n|)r^{|n|} \bigg(\int_1^r  sI_{\mu_n+1}(\sqrt{\lambda} s)\,f_{\theta,n}(s) \dd s \bigg) \bigg(\int_r^\infty s^{-|n|} K_{\mu_n-1}(\sqrt{\lambda} s) \dd s\bigg)\,, \\
J^{(1)}_{13}[f_n](r) &\,=\, -(\mu_n-|n|)r^{|n|} \int_r^\infty \tau I_{\mu_n+1}(\sqrt{\lambda} \tau)\,f_{\theta,n}(\tau) \int_\tau^\infty s^{-|n|} K_{\mu_n-1}(\sqrt{\lambda} s) \dd s \dd \tau\,, \\
J^{(1)}_{14}[f_n](r) &\,=\, r^{|n|} \int_r^\infty K_{\mu_n}(\sqrt{\lambda} \tau)\,g^{(2)}_n(\tau) \int_r^\tau s^{1-|n|} I_{\mu_n}(\sqrt{\lambda} s) \dd s \dd \tau\,, \\
J^{(1)}_{15}[f_n](r) &\,=\, (\mu_n+|n|) r^{|n|} \int_r^\infty \tau K_{\mu_n-1}(\sqrt{\lambda} \tau)\,f_{\theta,n}(\tau) \int_r^\tau s^{-|n|} I_{\mu_n+1}(\sqrt{\lambda} s) \dd s \dd \tau\,, \\
J^{(1)}_{16}[f_n](r) &\,=\, -r K_{\mu_n-1}(\sqrt{\lambda} r) \int_1^r s I_{\mu_n+1}(\sqrt{\lambda} s)\,f_{\theta,n}(s)\dd s\,, \\
J^{(1)}_{17}[f_n](r) &\,=\, -r I_{\mu_n+1}(\sqrt{\lambda} r) \int_r^\infty s K_{\mu_n-1}(\sqrt{\lambda} s)\,f_{\theta,n}(s)\dd s\,.
\end{align*}
}
\end{lemma}
%
\begin{remark}\label{rem.lem.velocitydecom.RSed.f} 
(1) The estimate to the term $J^{(1)}_9[f_n]$ is not needed in the following analysis thanks to the cancellation $J^{(1)}_9[f_n](r)-r^{-|n|} J^{(1)}_{17}[f_n](1)=0$ in the Biot-Savart law $V_n[\Phi_{n,\lambda}[f_n]]$. This fact will be used in the proof of Proposition \ref{prop2.est.velocity.RSed.f}. \\
\noindent (2) Note that $J^{(1)}_{7}[f_n]=-J^{(1)}_{16}[f_n]$ and $J^{(1)}_{8}[f_n]=-J^{(1)}_{17}[f_n]$ hold. Therefore we will skip the derivation of the estimates for $J^{(1)}_{16}[f_n]$ and $J^{(1)}_{17}[f_n]$ in Lemma \ref{lem2.est.velocity.RSed.f}. \\
\noindent (3) We can express the constant $c_{n,\lambda}[f_n]$ in \eqref{const.nFourier.RSed.f} in terms of $J^{(1)}_{l}[f_n](r)$ as $c_{n,\lambda}[f_n]=\sum_{l=11,13,14,15,17} J^{(1)}_l[f_n](1)$.
\end{remark}
%
The estimates to $J^{(1)}_{l}[f_n]$, $l\in\{1,\ldots,8\}$, in Lemma \ref{lem.velocitydecom.RSed.f} are given as follows.
%
\begin{lemma}\label{lem1.est.velocity.RSed.f}
Let $|n|=1$ and $q\in[1,\infty)$, and let $\lambda\in\Sigma_{\pi-\epsilon} \cap \mathcal{B}_1(0)$ for some $\epsilon\in(0,\frac{\pi}{2})$. Then there is a positive constant $C=C(q,\epsilon)$ independent of $\beta$ such that the following statements hold. \\
{\rm (1)} Let $f\in C^\infty_0(D)^2$. Then for $l\in\{1,\ldots,8\}$ we have 
\begin{align}
|J^{(1)}_{l}[f_n](r)| 
\le \frac{C}{\beta} r^{3-\frac{2}{q}} \|f \|_{L^q(D)}\,,~~~~~~~~
1\le r<{\rm Re}(\sqrt{\lambda})^{-1}\,.
\label{est1.lem1.est.velocity.RSed.f}
\end{align}
On the other hand, for $l\in\{1,\ldots,6\}$ we have 
\begin{align}
|J^{(1)}_{l}[f_n](r)| 
\le \frac{C}{\beta} |\lambda|^{-1} r^{1-\frac{2}{q}} \|f \|_{L^q(D)}\,,~~~~~~~~
r \ge {\rm Re}(\sqrt{\lambda})^{-1}\,,
\label{est2.lem1.est.velocity.RSed.f}
\end{align}
while for $l\in\{7,8\}$ we have
\begin{align}
|J^{(1)}_{l}[f_n](r)| 
\le 
C |\lambda|^{-1+\frac{1}{2q}} r^{1-\frac1q} \|f \|_{L^q(D)}\,,~~~~~~~~
r \ge {\rm Re}(\sqrt{\lambda})^{-1}\,.
\label{est3.lem1.est.velocity.RSed.f}
\end{align}
{\rm (2)} Let $f\in C^\infty_0(D)^2$. Then for $l\in\{7,8\}$ we have
\begin{align}
\|r^{-1} J^{(1)}_l[f_n]\|_{L^\infty(D)}
&\le \frac{C}{\beta} |\lambda|^{-1} \|f \|_{L^\infty(D)} 
\label{est4.lem1.est.velocity.RSed.f}\,, \\
\|r^{-1} J^{(1)}_l[f_n]\|_{L^1(D)}
&\le \frac{C}{\beta} |\lambda|^{-1} \|f \|_{L^1(D)}\,.
\label{est5.lem1.est.velocity.RSed.f}
\end{align}
\end{lemma}
%
\begin{proof}
(1) (i) Estimate of $J^{(1)}_1[f_n]$: For $1 \le r < {\rm Re}(\sqrt{\lambda})^{-1}$, by \eqref{est5.lem.est1.bessel} for $k=0$ in Lemma \ref{lem.est1.bessel} and \eqref{est1.lem.est2.bessel} for $k=0$ in Lemma \ref{lem.est2.bessel} in Appendix \ref{app.est.bessel}, we find
\begin{align*}
|J^{(1)}_1[f_n](r)| 
& \le 
r^{-1} \int_1^r 
|I_{\mu_n}(\sqrt{\lambda} \tau)\,g^{(1)}_n(\tau)|\,
\bigg|\int_\tau^r s^2 K_{\mu_n}(\sqrt{\lambda} s) \dd s\bigg| \dd \tau \\
& \le 
C r
\int_1^r |f_n(\tau)| \tau \dd \tau\,,
\end{align*}
which leads to the estimate \eqref{est1.lem1.est.velocity.RSed.f}. For $r \ge {\rm Re}(\sqrt{\lambda})^{-1}$, by \eqref{est5.lem.est1.bessel} and \eqref{est7.lem.est1.bessel} for $k=0$ in Lemma \ref{lem.est1.bessel} and \eqref{est2.lem.est2.bessel} and \eqref{est3.lem.est2.bessel} for $k=0$ in Lemma \ref{lem.est2.bessel}, we have
\begin{align*}
&|J^{(1)}_1[f_n](r)| 
\le 
r^{-1} 
\bigg(
\int_1^{\frac{1}{{\rm Re}\,(\sqrt{\lambda})}}
+ \int_{\frac{1}{{\rm Re}\,(\sqrt{\lambda})}}^r
\bigg)
|I_{\mu_n}(\sqrt{\lambda} \tau)\,g^{(1)}_n(\tau)|\,
\bigg|\int_\tau^r s^2 K_{\mu_n}(\sqrt{\lambda} s) \dd s\bigg| \dd \tau \\
&\le 
C |\lambda|^{-1} r^{-1} 
\int_1^{\frac{1}{{\rm Re}\,(\sqrt{\lambda})}}
|f_n(\tau)| \tau \dd \tau 
+ C\,|\lambda|^{-1} r^{-1} \int_{\frac{1}{{\rm Re}\,(\sqrt{\lambda})}}^r 
|f_n(\tau)| \tau \dd \tau\,,
\end{align*}
which implies the estimate \eqref{est2.lem1.est.velocity.RSed.f}. \\
\noindent (ii) Estimate of $J^{(1)}_{2}[f_n]$: The proof is parallel to that for $J^{(1)}_1[f_n]$ using the results in Lemmas \ref{lem.est1.bessel} and \ref{lem.est2.bessel} for $k=1$. We omit the details here. \\
\noindent (iii) Estimate of $J^{(1)}_{3}[f_n]$: For $1\le r<{\rm Re}(\sqrt{\lambda})^{-1}$, by \eqref{est1.lem.est1.bessel} and \eqref{est3.lem.est1.bessel} in Lemma \ref{lem.est1.bessel} and \eqref{est1.lem.est3.bessel} for $k=0$ in Lemma \ref{lem.est3.bessel}, we see that
\begin{align*}
|J^{(1)}_3[f_n](r)| 
& \le 
r^{-1} \int_1^r |K_{\mu_n}(\sqrt{\lambda} \tau)\,g^{(2)}_n(\tau)|\,
\int_1^\tau |s^2 I_{\mu_n}(\sqrt{\lambda} s)| \dd s \dd \tau \\
& \le 
C\,r \int_1^r |f_n(\tau)| \tau \dd \tau \,.
\end{align*}
Thus we have \eqref{est1.lem1.est.velocity.RSed.f}. For $r \ge  {\rm Re}(\sqrt{\lambda})^{-1}$, by \eqref{est1.lem.est1.bessel}, \eqref{est3.lem.est1.bessel}, and \eqref{est6.lem.est1.bessel} for $k=0$ in Lemma \ref{lem.est1.bessel} and \eqref{est1.lem.est3.bessel} and \eqref{est2.lem.est3.bessel} for $k=0$ in Lemma \ref{lem.est3.bessel} we have
\begin{align*}
&|J^{(1)}_3[f_n](r)| 
\le 
r^{-1} 
\bigg(
\int_1^{\frac{1}{{\rm Re}\,(\sqrt{\lambda})}}
+ 
{\color{black}
\int_{\frac{1}{{\rm Re}\,(\sqrt{\lambda})}}^r
}
\bigg)
|K_{\mu_n}(\sqrt{\lambda} \tau)\,g^{(2)}_n(\tau)|\,
\int_1^\tau |s^2 I_{\mu_n}(\sqrt{\lambda} s)| \dd s \dd \tau \\
& \le
C\,|\lambda|^{-1} r^{-1} 
\int_1^{\frac{1}{{\rm Re}\,(\sqrt{\lambda})}} |f_n(\tau)| \tau \dd \tau 
+ C\,|\lambda|^{-1} r^{-1} 
\int_{\frac{1}{{\rm Re}\,(\sqrt{\lambda})}}^r 
|f_n(\tau)| \tau \dd \tau\,,
\end{align*}
which leads to \eqref{est2.lem1.est.velocity.RSed.f}. \\
\noindent (iv) Estimate of $J^{(1)}_{4}[f_n]$: The proof is parallel to that for $J^{(1)}_3[f_n]$ using the results in Lemmas \ref{lem.est1.bessel} and \ref{lem.est3.bessel} for $k=1$, and we omit here. \\
\noindent (v) Estimates of $J^{(1)}_{5}[f_n]$ and $J^{(1)}_{6}[f_n]$: We give a proof only for $J^{(1)}_{5}[f_n]$ since the proof for $J^{(1)}_{6}[f_n]$ is similar. For $1\le r < {\rm Re}(\sqrt{\lambda})^{-1}$, 
{\color{black}
by \eqref{est1.lem.est1.bessel}, \eqref{est3.lem.est1.bessel}, and \eqref{est6.lem.est1.bessel} for $k=0$ in Lemma \ref{lem.est1.bessel} and \eqref{est1.lem.est3.bessel} for $k=0$ in Lemma \ref{lem.est3.bessel}
}
we observe that
\begin{align*}
& |J^{(1)}_5[f_n](r)|
\le
r^{-1} \int_1^r |s^2 I_{\mu_n}(\sqrt{\lambda} s)| \dd s 
\bigg(
\int_r^{\frac{1}{{\rm Re}\,(\sqrt{\lambda})}} 
+ \int_{\frac{1}{{\rm Re}\,(\sqrt{\lambda})}}^\infty
\bigg)
|K_{\mu_n}(\sqrt{\lambda}s)\,g^{(2)}_n(s)| \dd s \\
& \le
C\,r^{{\rm Re}(\mu_n)+2}
\int_r^{\frac{1}{{\rm Re}\,(\sqrt{\lambda})}}
s^{-({\rm Re}(\mu_n)+1)} |f_n(s)| s \dd s \\
& \quad
+ C\,|\lambda|^{\frac{{\rm Re}(\mu_n)}{2}-\frac14} r^{{\rm Re}(\mu_n)+2}
\int_{\frac{1}{{\rm Re}\,(\sqrt{\lambda})}}^\infty
s^{-\frac32} e^{-{\rm Re}(\sqrt{\lambda}) s} |f_n(s)| s \dd s \,.
\end{align*}
Then a direct calculation shows \eqref{est1.lem1.est.velocity.RSed.f}. For $r \ge {\rm Re}(\sqrt{\lambda})^{-1}$,
{\color{black}
by \eqref{est6.lem.est1.bessel} for $k=0$ in Lemma \ref{lem.est1.bessel} and \eqref{est2.lem.est3.bessel} for $k=0$ in Lemma \ref{lem.est3.bessel}
}
we have
\begin{align*}
|J_5^{(1)}[f_n](r)|
& \le
r^{-1} \int_1^r |s^2 I_{\mu_n}(\sqrt{\lambda} s)| \dd s
\int_r^\infty |K_{\mu_n}(\sqrt{\lambda}s)\,g^{(2)}_n(s)| \dd s \\
& \le
C\,|\lambda|^{-1} r^{-\frac12} e^{{\rm Re}\,(\sqrt{\lambda}) r}
\int_r^\infty
s^{-\frac12} e^{-{\rm Re}(\sqrt{\lambda}) s} |f_n(s)|s \dd s\,, 
\end{align*}
which implies \eqref{est2.lem1.est.velocity.RSed.f}. \\
\noindent (vi) Estimate of $J^{(1)}_{7}[f_n]$: For $1\le r<{\rm Re}(\sqrt{\lambda})^{-1}$,  by \eqref{est2.lem.est1.bessel}, \eqref{est4.lem.est1.bessel}, and \eqref{est5.lem.est1.bessel} for $k=1$ in Lemma \ref{lem.est1.bessel} we find
\begin{align}
|J^{(1)}_7[f_n](r)| 
& \le 
|rK_{\mu_n-1}(\sqrt{\lambda}r)|\,
\int_1^r |I_{\mu_n+1}(\sqrt{\lambda} s)\,f_{\theta,n}(s) s| \dd s \nonumber \\
& \le
C\beta^{-1} r
\int_1^r |f_{n}(s)| s \dd s\,.
\label{est1.proof.lem1.est.velocity.RSed.f}
\end{align}
Thus we have \eqref{est1.lem1.est.velocity.RSed.f}. For $r \ge {\rm Re}(\sqrt{\lambda})^{-1}$, by \eqref{est5.lem.est1.bessel}--\eqref{est7.lem.est1.bessel} for $k=1$ in Lemma \ref{lem.est1.bessel} we have
\begin{align}
&|J^{(1)}_7[f_n](r)|
\le
|rK_{\mu_n-1}(\sqrt{\lambda}r)|
\bigg(
\int_1^{\frac{1}{{\rm Re}\,(\sqrt{\lambda})}}
+ \int_{\frac{1}{{\rm Re}\,(\sqrt{\lambda})}}^r
\bigg)
|I_{\mu_n+1}(\sqrt{\lambda} s)\,f_{\theta,n}(s) s| \dd s \nonumber \\
& \le
C\,|\lambda|^{-\frac14} r^{\frac12} e^{-{\rm Re}\,(\sqrt{\lambda}) r}
\int_1^{\frac{1}{{\rm Re}(\sqrt{\lambda})}}
|f_n(s)|s \dd s \nonumber \\
& \quad
+ C\,|\lambda|^{-\frac12} r^{\frac12} e^{-{\rm Re}\,(\sqrt{\lambda}) r}
\int_{\frac{1}{{\rm Re}\,(\sqrt{\lambda})}}^r
s^{-\frac12} e^{{\rm Re}\,(\sqrt{\lambda}) s} 
|f_n(s)|s \dd s\,, 
\label{est2.proof.lem1.est.velocity.RSed.f}
\end{align}
which leads to \eqref{est3.lem1.est.velocity.RSed.f}. \\
\noindent (vii) Estimate of $J^{(1)}_{8}[f_n]$: For $1\le r<{\rm Re}(\sqrt{\lambda})^{-1}$, 
{\color{black}
by the results in Lemmas \ref{lem.est1.bessel} for $k=1$
}
we find
\begin{align}
& |J^{(1)}_8[f_n](r)| 
\le 
|rI_{\mu_n+1}(\sqrt{\lambda}r)|\,
\bigg( 
\int_r^{{\frac{1}{{\rm Re}(\sqrt{\lambda})}}} 
+ \int_{{\frac{1}{{\rm Re}(\sqrt{\lambda})}}}^\infty
\bigg)
|K_{\mu_n-1}(\sqrt{\lambda} s)\,f_{\theta,n}(s) s| \dd s \nonumber \\
& \le
C\beta^{-1} |\lambda| r^3
\int_r^{\frac{1}{{\rm Re}(\sqrt{\lambda})}} 
|f_n(s)| s \dd s 
+ C |\lambda|^{\frac34} r^{3}
\int_{\frac{1}{{\rm Re}\,(\sqrt{\lambda})}}^\infty
s^{-\frac12} e^{-{\rm Re}(\sqrt{\lambda}) s} 
|f_n(s)| s \dd s\,, 
\label{est3.proof.lem1.est.velocity.RSed.f}
\end{align}
which implies \eqref{est1.lem1.est.velocity.RSed.f}. For $r \ge {\rm Re}(\sqrt{\lambda})^{-1}$,
{\color{black}
by Lemma \ref{lem.est1.bessel} for $k=1$ again
}
we have
\begin{align}
&|J^{(1)}_8[f_n](r)|
\le
|rI_{\mu_n+1}(\sqrt{\lambda}r)|\,
\int_r^\infty
|K_{\mu_n-1}(\sqrt{\lambda} s)\,f_{\theta,n}(s) s| \dd s \nonumber \\
&\le
C |\lambda|^{-\frac12} r^{\frac12} e^{{\rm Re}\,(\sqrt{\lambda}) r}
\int_r^\infty
s^{-\frac{1}{2}} e^{-{\rm Re}\,(\sqrt{\lambda}) s} 
|f_n(s)| s \dd s\,, 
\label{est4.proof.lem1.est.velocity.RSed.f}
\end{align}
which leads to \eqref{est3.lem1.est.velocity.RSed.f}. Hence we obtain the assertion (1) of
Lemma \ref{lem1.est.velocity.RSed.f}. \\
\noindent (2) The estimate \eqref{est4.lem1.est.velocity.RSed.f} follows from  \eqref{est1.proof.lem1.est.velocity.RSed.f}--\eqref{est4.proof.lem1.est.velocity.RSed.f} in the above. For the proof of \eqref{est5.lem1.est.velocity.RSed.f}, one can reproduce the calculation performed in \cite[Lemma 3.7]{Ma1} using the results in Lemma \ref{lem.est1.bessel}, and hence we omit the details here. This completes the proof of Lemma \ref{lem1.est.velocity.RSed.f}.
\end{proof}
%
%
The next lemma summarizes the estimates to $J^{(1)}_{l}[f_n](r)$, $l\in\{10,\ldots,17\}$, in Lemma \ref{lem.velocitydecom.RSed.f}. We skip the proofs for $J^{(1)}_{16}[f_n]$ and $J^{(1)}_{17}[f_n]$ as is already mentioned in Remark \ref{rem.lem.velocitydecom.RSed.f} (2).
%
\begin{lemma}\label{lem2.est.velocity.RSed.f}
Let $|n|=1$ and $q\in[1,\infty)$, and let $\lambda\in\Sigma_{\pi-\epsilon} \cap \mathcal{B}_1(0)$ for some $\epsilon\in(0,\frac{\pi}{2})$. Then there is a positive constant $C=C(q,\epsilon)$ independent of $\beta$ such that the following statements hold. \\
\noindent {\rm (1)} Let $f\in C^\infty_0(D)^2$. Then for $l\in\{10,\ldots,17\}$ we have 
\begin{align}
|J^{(1)}_{l}[f_n](r)| 
\le \frac{C}{\beta} |\lambda|^{-1+\frac1q} r \|f \|_{L^q(D)}\,,~~~~~~
1\le r<{\rm Re}(\sqrt{\lambda})^{-1}\,.
\label{est1.lem2.est.velocity.RSed.f}
\end{align}
On the other hand, for $l\in\{10,\ldots,15\}$ we have 
\begin{align}
|J^{(1)}_{l}[f_n](r)| 
\le {\color{black} C} |\lambda|^{-1} r^{1-\frac{2}{q}} \|f \|_{L^q(D)}\,,~~~~~~
r \ge {\rm Re}(\sqrt{\lambda})^{-1}\,,
\label{est2.lem2.est.velocity.RSed.f}
\end{align}
while for $l\in\{16,17\}$ we have
\begin{align}
|J^{(1)}_{l}[f_n](r)| 
\le 
C |\lambda|^{-1+\frac{1}{2q}} r^{1-\frac1q} \|f \|_{L^q(D)}\,,~~~~~~
r \ge {\rm Re}(\sqrt{\lambda})^{-1}\,.
\label{est3.lem2.est.velocity.RSed.f}
\end{align}
{\rm (2)} Let $f\in C^\infty_0(D)^2$. Then for $l\in\{16,17\}$ we have
\begin{align}
\|r^{-1} J^{(1)}_l[f_n]\|_{L^\infty(D)}
&\le \frac{C}{\beta} |\lambda|^{-1} \|f \|_{L^\infty(D)} 
\label{est4.lem2.est.velocity.RSed.f}\,, \\
\|r^{-1} J^{(1)}_l[f_n]\|_{L^1(D)}
&\le \frac{C}{\beta} |\lambda|^{-1} \|f \|_{L^1(D)}\,.
\label{est5.lem2.est.velocity.RSed.f}
\end{align}
\end{lemma}
%
\begin{proof}
\noindent (1) (i) Estimate of $J^{(1)}_{10}[f_n]$: For $1\le r< {\rm Re}(\sqrt{\lambda})^{-1}$, by \eqref{est5.lem.est1.bessel} for $k=0$ in Lemma \ref{lem.est1.bessel} and \eqref{est4.lem.est2.bessel} for $k=0$ in Lemma \ref{lem.est2.bessel} in Appendix \ref{app.est.bessel}, we find
\begin{align*}
|J^{(1)}_{10}[f_n](r)| 
&\le 
r \bigg|\int_r^\infty K_{\mu_n}(\sqrt{\lambda} s) \dd s\bigg|\,
\int_1^r |I_{\mu_n}(\sqrt{\lambda} s)\,g^{(1)}_n(s)| \dd s \\
&\le 
C \beta^{-1} r
\int_1^r |f_n(s)| s \dd s\,,
\end{align*}
which implies \eqref{est1.lem2.est.velocity.RSed.f}. For $r \ge {\rm Re}(\sqrt{\lambda})^{-1}$, by \eqref{est5.lem.est1.bessel} and \eqref{est7.lem.est1.bessel} for $k=0$ in Lemma \ref{lem.est1.bessel} and \eqref{est5.lem.est2.bessel} for $k=0$ in Lemma \ref{lem.est2.bessel}, we have
{\allowdisplaybreaks
\begin{align*}
& |J^{(1)}_{10}[f_n](r)| 
\le 
r \int_r^\infty |K_{\mu_n}(\sqrt{\lambda} s)| \dd s
\bigg(
\int_1^{\frac{1}{{\rm Re}(\sqrt{\lambda})}} 
+ \int_{\frac{1}{{\rm Re}(\sqrt{\lambda})}}^r
\bigg)
|I_{\mu_n}(\sqrt{\lambda} s)\,g^{(1)}_n(s)| \dd s \\
&\le 
C\,|\lambda|^{-\frac14} r^\frac12 e^{-{\rm Re}(\sqrt{\lambda}) r} 
\int_1^{\frac{1}{{\rm Re}(\sqrt{\lambda})}} |f_n(s)| s \dd s \\
& \quad
+ C\,|\lambda|^{-1} r^\frac12 e^{-{\rm Re}(\sqrt{\lambda}) r}
\int_{\frac{1}{{\rm Re}\,(\sqrt{\lambda})}}^r
s^{-\frac32} e^{{\rm Re}(\sqrt{\lambda}) s}
|f_n(s)| s \dd s \,,
\end{align*}
}
which leads to \eqref{est2.lem2.est.velocity.RSed.f}. \\
\noindent (ii) Estimate of $J^{(1)}_{11}[f_n]$: For $1\le r < {\rm Re}(\sqrt{\lambda})^{-1}$,  by \eqref{est5.lem.est1.bessel} and \eqref{est7.lem.est1.bessel} for $k=0$ in Lemma \ref{lem.est1.bessel} and \eqref{est4.lem.est2.bessel} and \eqref{est5.lem.est2.bessel} for $k=0$ in Lemma \ref{lem.est2.bessel}, we see that
\begin{align*}
&|J^{(1)}_{11}[f_n](r)| 
\le 
r\,\bigg( \int_r^{\frac{1}{{\rm Re}(\sqrt{\lambda})}} 
+ \int_{\frac{1}{{\rm Re}(\sqrt{\lambda})}}^\infty \bigg)
|I_{\mu_n}(\sqrt{\lambda} \tau)\,g^{(1)}_n(\tau)|\, 
\bigg| \int_\tau^\infty K_{\mu_n}(\sqrt{\lambda} s) \dd s \bigg| \dd \tau \\
& \le 
C \beta^{-1} r
\int_r^{\frac{1}{{\rm Re}(\sqrt{\lambda})}} 
|f_n(\tau)| \tau \dd \tau 
+ C\,|\lambda|^{-1} r
\int_{\frac{1}{{\rm Re}(\sqrt{\lambda})}}^\infty
\tau^{-2} |f_n(\tau)| \tau \dd \tau \,,
\end{align*}
which implies \eqref{est1.lem2.est.velocity.RSed.f}. For $r \ge {\rm Re}(\sqrt{\lambda})^{-1}$, by \eqref{est7.lem.est1.bessel} for $k=0$ in Lemma \ref{lem.est1.bessel} and \eqref{est5.lem.est2.bessel} for $k=0$ in Lemma \ref{lem.est2.bessel}, we have
\begin{align*}
|J^{(1)}_{11}[f_n](r)| 
&\le 
r\,\int_r^\infty
|I_{\mu_n}(\sqrt{\lambda} \tau)\,g^{(1)}_n(\tau)|\, 
\int_\tau^\infty |K_{\mu_n}(\sqrt{\lambda} s)| \dd s \dd \tau \\
& \le 
C\,|\lambda|^{-1} r
\int_r^\infty
\tau^{-2} |f_n(\tau)| \tau \dd \tau\,,
\end{align*}
which leads to \eqref{est2.lem2.est.velocity.RSed.f}. \\
\noindent (iii) Estimates of $J^{(1)}_{12}[f_n]$ and $J^{(1)}_{13}[f_n]$: The proof for $J^{(1)}_{12}[f_n]$ is parallel to that for $J^{(1)}_{10}[f_n]$ using the bound $|\mu_n-1| \le C\beta$ and the results in Lemmas \ref{lem.est1.bessel} and \ref{lem.est2.bessel} for $k=1$. The proof for $J^{(1)}_{13}[f_n]$ is similar to that for $J^{(1)}_{11}[f_n]$. Thus we omit the details here. \\
\noindent (iv) Estimate of $J^{(1)}_{14}[f_n]$: For $1\le r < {\rm Re}(\sqrt{\lambda})^{-1}$, by \eqref{est1.lem.est1.bessel}, \eqref{est3.lem.est1.bessel}, and \eqref{est6.lem.est1.bessel} for $k=0$ in Lemma \ref{lem.est1.bessel} and \eqref{est3.lem.est3.bessel} and \eqref{est4.lem.est3.bessel} for $k=0$ in Lemma \ref{lem.est3.bessel}, we observe that
\begin{align*}
& |J^{(1)}_{14}[f_n](r)| 
\le 
r \bigg(
\int_r^{\frac{1}{{\rm Re}(\sqrt{\lambda})}}
+ \int_{\frac{1}{{\rm Re}\,(\sqrt{\lambda})}}^\infty
\bigg)
|K_{\mu_n}(\sqrt{\lambda} \tau)\,\,g^{(2)}_n(\tau)| 
\int_r^\tau |I_{\mu_n}(\sqrt{\lambda} s)| \dd s \dd \tau \\
& \le 
C r
\int_r^{\frac{1}{{\rm Re}\,(\sqrt{\lambda})}}
|f_n(\tau)| \tau \dd \tau 
+ C |\lambda|^{-1} r
\int_{\frac{1}{{\rm Re}\,(\sqrt{\lambda})}}^\infty
\tau^{-2} |f_n(\tau)| \tau \dd \tau\,,
\end{align*}
which implies \eqref{est1.lem2.est.velocity.RSed.f}. For $r \ge {\rm Re}(\sqrt{\lambda})^{-1}$, by \eqref{est6.lem.est1.bessel} in Lemma \ref{lem.est1.bessel} and \eqref{est5.lem.est3.bessel} in Lemma \ref{lem.est3.bessel} for $k=0$ we have
\begin{align*}
|J^{(1)}_{14}[f_n](r)| 
& \le 
r \int_r^\infty |K_{\mu_n}(\sqrt{\lambda} \tau)\,\,g^{(2)}_n(\tau)|\, 
\int_r^\tau |I_{\mu_n}(\sqrt{\lambda} s)| \dd s \dd \tau \\
& \le 
C |\lambda|^{-1} r\,
\int_r^\infty \tau^{-2} |f_{n}(\tau)| \tau \dd \tau\,,
\end{align*}
which leads to \eqref{est2.lem2.est.velocity.RSed.f}. \\
\noindent (v) Estimate of $J^{(1)}_{15}[f_n]$: The proof is parallel to that for $J^{(1)}_{14}[f_n]$ using Lemmas \ref{lem.est1.bessel} and \ref{lem.est3.bessel} for $k=1$, and thus we omit here. This completes the proof of Lemma \ref{lem2.est.velocity.RSed.f}.
\end{proof}
%
%
Lemmas \ref{lem1.est.velocity.RSed.f} and \ref{lem2.est.velocity.RSed.f} lead to the next important estimates that we shall need in the proof of Proposition \ref{prop2.est.velocity.RSed.f} below. Let $c_{n,\lambda}[f_n]$ be the constant in \eqref{const.nFourier.RSed.f}.
%
\begin{corollary}\label{cor1.est.velocity.RSed.f}
Let $|n|=1$ and $1\le q<p\le\infty$ or $1<q\le p<\infty$, and let $\lambda\in\Sigma_{\pi-\epsilon} \cap \mathcal{B}_1(0)$ for some $\epsilon\in(0,\frac{\pi}{2})$. Then there is a positive constant $C=C(q,p,\epsilon)$ independent of $\beta$ such that the following statement holds. Let $f\in C^\infty_0(D)^2$. Then for $l\in\{1,\ldots,17\}\setminus\{9\}$ we have
\begin{align}
|c_{n,\lambda}[f_n]|
&\le
\frac{C}{\beta} |\lambda|^{-1+\frac1q} \|f \|_{L^q(D)}\,,
\label{est1.cor1.est.velocity.RSed.f} \\
\|r^{-1} J^{(1)}_l[f_n]\|_{L^p(D)}
&\le
\frac{C}{\beta} |\lambda|^{-1+\frac1q-\frac1p} \|f \|_{L^q(D)}\,, 
\label{est2.cor1.est.velocity.RSed.f} \\
\|r^{-2} J^{(1)}_l[f_n]\|_{L^2(D)}
& \le 
\frac{C}{\beta}
|\lambda|^{-1+\frac1q}
|\log {\rm Re}(\sqrt{\lambda})|^{\frac12}
\|f\|_{L^q(D)}\,.
\label{est3.cor1.est.velocity.RSed.f}
\end{align}
\end{corollary}
%
%
\begin{proof}
(i) Estimate of $c_{n,\lambda}[f_n]$: Remark \ref{rem.lem.velocitydecom.RSed.f} (3) ensures that 
\begin{align*}
|c_{n,\lambda}[f_n]| \le  \sum_{l=11,13,14,15,17} |J^{(1)}_l[f_n](1)|\,.
\end{align*}
Then the estimate \eqref{est1.cor1.est.velocity.RSed.f} follows from putting $r=1$ to
\eqref{est1.lem2.est.velocity.RSed.f} in Lemma \ref{lem2.est.velocity.RSed.f}. \\
\noindent (ii) Estimate of $r^{-1} J^{(1)}_l[f_n]$: If $l\in\{1,\ldots,17\}\setminus\{7,8,9,16,17\}$, then it is easy to see from the pointwise estimates in Lemmas \ref{lem1.est.velocity.RSed.f} and \ref{lem2.est.velocity.RSed.f} that 
\begin{align*}
\sup_{r\ge1} r^{\frac2q} |r^{-1} J^{(1)}_l[f_n](r)|
\le C \beta^{-1} |\lambda|^{-1} \|f \|_{L^q(D)}\,,~~~~~~
1\le q<\infty\,.
\end{align*}
Thus by the Marcinkiewicz interpolation theorem we have \eqref{est2.cor1.est.velocity.RSed.f} for the case $1<p=q<\infty$. Moreover, again from Lemmas \ref{lem1.est.velocity.RSed.f} and \ref{lem2.est.velocity.RSed.f} one can see that
\begin{align}
\sup_{r\ge1} |r^{-1} J^{(1)}_l[f_n](r)|
\le C \beta^{-1} \|f \|_{L^1(D)} 
\label{est1.proof.cor1.est.velocity.RSed.f}\,,
\end{align}
which leads to \eqref{est2.cor1.est.velocity.RSed.f} for the case $1<p\le\infty$ and $q=1$. Hence finally we have \eqref{est2.cor1.est.velocity.RSed.f} for $1\le q<p\le\infty$ and $1<q\le p<\infty$ by the Marcinkiewicz interpolation theorem again. \\
If $l\in\{7,8,16,17\}$, from \eqref{est4.lem1.est.velocity.RSed.f}, \eqref{est5.lem1.est.velocity.RSed.f}, \eqref{est4.lem2.est.velocity.RSed.f}, and \eqref{est5.lem2.est.velocity.RSed.f} we have 
\eqref{est2.cor1.est.velocity.RSed.f} for the case $1\le p=q\le\infty$ by the interpolation argument. Moreover, 
\eqref{est1.lem1.est.velocity.RSed.f}, \eqref{est3.lem1.est.velocity.RSed.f}, 
\eqref{est1.lem2.est.velocity.RSed.f}, and \eqref{est3.lem2.est.velocity.RSed.f} lead to the estimate in the form \eqref{est1.proof.cor1.est.velocity.RSed.f} for $l\in\{7,8,16,17\}$. Thus we obtain \eqref{est2.cor1.est.velocity.RSed.f} for the case $1\le p\le\infty$ and $q=1$, and hence \eqref{est2.cor1.est.velocity.RSed.f} for $1\le q\le p\le \infty$ by the Marcinkiewicz interpolation theorem. \\
\noindent (iii) Estimate of $r^{-2} J^{(1)}_l[f_n]$: The assertion \eqref{est3.cor1.est.velocity.RSed.f} can be checked easily by a direct calculation using Lemmas \ref{lem1.est.velocity.RSed.f} and \ref{lem2.est.velocity.RSed.f}. We note that the logarithmic factor in \eqref{est3.cor1.est.velocity.RSed.f} is due to the estimate \eqref{est1.lem2.est.velocity.RSed.f}. The proof of Corollary \ref{cor1.est.velocity.RSed.f} is complete.
\end{proof}
%
%

Now we are in position to prove the main theorem of this subsection. Let us start with the simple proposition about the estimate for the term $V_n [K_{\mu_n}(\sqrt{\lambda}\,\cdot\,)]$ in \eqref{rep.velocity.nFourier.RSed.f}.
%
\begin{proposition}\label{prop1.est.velocity.RSed.f}
Let $|n|=1$, $p\in(1,\infty]$, and let $\lambda\in\Sigma_{\pi-\epsilon} \cap \mathcal{B}_1(0)$ for some $\epsilon\in(0,\frac{\pi}{2})$. Then there is a positive constant $C=C(p,\epsilon)$ independent of $\beta$ such that we have 
\begin{align}
\|V_n [K_{\mu_n}(\sqrt{\lambda}\,\cdot\,)]\|_{L^p(D)}
&\le
\frac{C}{\beta} |\lambda|^{-\frac{{\rm Re}(\mu_n)}{2}-\frac1p}\,, 
\label{est1.prop1.est.velocity.RSed.f} \\
\big\|\frac{V_n [K_{\mu_n}(\sqrt{\lambda}\,\cdot\,)]}{|x|}\big\|_{L^2(D)}
& \le 
\frac{C}{\beta^2} |\lambda|^{-\frac{{\rm Re}(\mu_n)}{2}}\,.
\label{est2.prop1.est.velocity.RSed.f}
\end{align}
\end{proposition}
%
\begin{proof}
It is easy to see from the definition of $V_n[\,\cdot\,]$ in \eqref{def.V_n} that
\begin{align*}
|V_n [K_{\mu_n}(\sqrt{\lambda}\,\cdot\,)]| 
\le 
C r^{-2}
\bigg( |F_n(\sqrt{\lambda};\beta)|
+ \bigg|\int_1^r s^2 K_{\mu_n}(\sqrt{\lambda} s) \dd s \bigg|
\bigg)
+ C \bigg|\int_r^\infty K_{\mu_n}(\sqrt{\lambda} s) \dd s\bigg|\,.
\end{align*}
By the results in Lemma \ref{lem.est2.bessel} for $k=0$ in Appendix \ref{app.est.bessel} we have 
\begin{align}
|V_n [K_{\mu_n}(\sqrt{\lambda}\,\cdot\,)](r)|
& \le 
C \beta^{-1} |\lambda|^{-\frac{{\rm Re}(\mu_n)}{2}} r^{-{\rm Re}(\mu_n)+1}\,,~~~~~~
1\le r < {\rm Re}(\sqrt{\lambda})^{-1}\,,
\label{est1.proof.prop1.est.velocity.RSed.f} \\
|V_n [K_{\mu_n}(\sqrt{\lambda}\,\cdot\,)](r)|
& \le 
C \beta^{-1} |\lambda|^{-\frac32} r^{-2}\,,~~~~~~
r \ge {\rm Re}(\sqrt{\lambda})^{-1}\,.
\label{est2.proof.prop1.est.velocity.RSed.f}
\end{align}
Then for $p\in[1,\infty]$ we find
\begin{align*}
\sup_{r\ge1} r^{\frac2p} |V_n [K_{\mu_n}(\sqrt{\lambda}\,\cdot\,)](r)|
& \le
C \beta^{-1} |\lambda|^{-\frac{{\rm Re}(\mu_n)}{2}-\frac1p}\,.
\end{align*}
Hence by an interpolation argument \eqref{est1.prop1.est.velocity.RSed.f} follows. Moreover, a direct calculation combined with  \eqref{est1.proof.prop1.est.velocity.RSed.f}, \eqref{est2.proof.prop1.est.velocity.RSed.f},
and $({\rm Re}(\mu_n(\beta))-1)^\frac12\approx O(\beta)$ yield \eqref{est2.prop1.est.velocity.RSed.f}. This completes the proof.
\end{proof}
%
The next proposition gives the estimate for the term $V_n[\Phi_{n,\lambda}[f_n]]$ in \eqref{rep.velocity.nFourier.RSed.f}.
\begin{proposition}\label{prop2.est.velocity.RSed.f}
Let $|n|=1$ and $1\le q<p\le\infty$ or $1<q\le p<\infty$, and let $\lambda\in\Sigma_{\pi-\epsilon} \cap \mathcal{B}_1(0)$ for some $\epsilon\in(0,\frac{\pi}{2})$. Then there is a positive constant $C=C(q,p,\epsilon)$ independent of $\beta$ such that for $f\in C^\infty_0(D)^2$ we have 
\begin{align}
\|V_n[\Phi_{n,\lambda}[f_n]]\|_{L^p(D)}
&\le
\frac{C}{\beta} |\lambda|^{-1+\frac1q-\frac1p} \|f \|_{L^q(D)}\,,
\label{est1.prop2.est.velocity.RSed.f} \\
\big\|\frac{V_n[\Phi_{n,\lambda}[f_n]]}{|x|}\big\|_{L^2(D)}
& \le 
\frac{C}{\beta}
|\lambda|^{-1+\frac1q}
|\log {\rm Re}(\sqrt{\lambda})|^{\frac12}
\|f\|_{L^q(D)}\,.
\label{est2.prop2.est.velocity.RSed.f}
\end{align}
\end{proposition}
%
\begin{proof}
The definition of the Biot-Savart law $V_n[\,\cdot\,]$ in \eqref{def.V_n} leads to the next representations for the radial part $V_{r,n}[\Phi_{n,\lambda}[f_n]]$ and the angular part $V_{\theta,n}[\Phi_{n,\lambda}[f_n]]$ of $V_n[\Phi_{n,\lambda}[f_n]]$:
\begin{align*}
V_{r,n}[\Phi_{n,\lambda}[f_n]] 
& \,=\,  -\frac{in}{2r} 
\bigg( \frac{c_{n,\lambda}[f_n]}{r} 
- \frac{1}{r} \int_1^r s^2 \Phi_{n,\lambda}[f_n](s) \dd s  
- r \int_r^\infty \Phi_{n,\lambda}[f_n](s) \dd s \bigg)\,, \\
V_{\theta,n}[\Phi_{n,\lambda}[f_n]] 
& \,=\, \frac{1}{2r} 
\bigg( \frac{c_{n,\lambda}[f_n]}{r} 
- \frac{1}{r} \int_1^r s^2 \Phi_{n,\lambda}[f_n](s) \dd s   
+ r \int_r^\infty \Phi_{n,\lambda}[f_n](s) \dd s \bigg)\,,
\end{align*}
where $c_{n,\lambda}[f_n]$ is defined in \eqref{const.nFourier.RSed.f}. From Lemma \ref{lem.velocitydecom.RSed.f} and Remark \ref{rem.lem.velocitydecom.RSed.f} (1) and (3) we see that
\begin{align}\label{decom.proof.prop2.est.velocity.RSed.f}
&\frac{c_{n,\lambda}[f_n]}{r} - \frac{1}{r} \int_1^r s^2 \Phi_{n,\lambda}[f_n](s) \dd s
\nonumber \\
&\,=\, r^{-1} \sum_{l=11,13,14,15} J^{(1)}_l[f_n](1) - \sum_{l=1}^{8} r^{-1} J^{(1)}_l[f_n](r)\,.
\end{align}
Then, by \eqref{decom.proof.prop2.est.velocity.RSed.f} and the decomposition of the integral $r \int_r^\infty \Phi_{n,\lambda}[f_n](s) \dd s$ in Lemma \ref{lem.velocitydecom.RSed.f}, we find the following pointwise estimate of $V_n[\Phi_{n,\lambda}[f_n]](r)$:
\begin{equation}\label{est1.proof.prop2.est.velocity.RSed.f}
\begin{aligned}
&|V_n[\Phi_{n,\lambda}[f_n]](r)| \\
&\le C \Big( r^{-2} \sum_{l=11,13,14,15} |J^{(1)}_l[f_n](1)|
+ \sum_{l\in\{1,\ldots,17\}\setminus\{9\}} |r^{-1} J^{(1)}_l[f_n](r)| \Big)\,.
\end{aligned}
\end{equation}
Thus the assertions \eqref{est1.prop2.est.velocity.RSed.f} and 
\eqref{est2.prop2.est.velocity.RSed.f} follow from Corollary \ref{cor1.est.velocity.RSed.f}. This completes the proof.
\end{proof}
%
Finally we give a proof of Theorem \ref{thm.est.velocity.RSed.f}, which is a direct consequence of Corollary \ref{cor1.est.velocity.RSed.f} and Propositions \ref{prop1.est.velocity.RSed.f} and \ref{prop2.est.velocity.RSed.f}.
\begin{proofx}{Theorem \ref{thm.est.velocity.RSed.f}} 
In view of Proposition \ref{prop2.est.velocity.RSed.f}, it suffices to show that the first term in the right-hand side of \eqref{rep.velocity.nFourier.RSed.f} satisfies the estimates \eqref{est1.thm.est.velocity.RSed.f} and \eqref{est2.thm.est.velocity.RSed.f}. By using Proposition \ref{prop.est.F} and \eqref{est1.cor1.est.velocity.RSed.f} in Corollary \ref{cor1.est.velocity.RSed.f}, one can see that \eqref{est1.thm.est.velocity.RSed.f} and \eqref{est2.thm.est.velocity.RSed.f} respectively follow from \eqref{est1.prop1.est.velocity.RSed.f} and \eqref{est2.prop1.est.velocity.RSed.f} in Proposition \ref{prop1.est.velocity.RSed.f}. This completes the proof of Theorem \ref{thm.est.velocity.RSed.f}.
\end{proofx}
%
\subsubsection{Estimates of the vorticity for \eqref{nFourier.RSed.f} with $|n|=1$}
\label{subsec.RSed.f.vorticity}
This subsection is devoted to the estimate of the vorticity $\omega^{{\rm ed}}_{f,n}(r)=({\rm rot}\,w^{\rm ed}_{f,n})e^{-in\theta}$ with $|n|=1$, where $w^{\rm ed}_{f,n}$ solves \eqref{nFourier.RSed.f} in Subsection \ref{subsec.RSed.f}. We recall that $\omega^{{\rm ed}}_{f,n}$ is represented as
\begin{align*}
\omega^{{\rm ed}}_{f,n}(r)
&\,=\, 
-\frac{c_{n,\lambda}[f_n]}{F_n(\sqrt{\lambda};\beta)} 
K_{\mu_n}(\sqrt{\lambda} r) 
+ \Phi_{n,\lambda}[f_n](r)
\end{align*}
by \eqref{rep.vorticity.nFourier.RSed.f}. The main result is stated as follows. Let $\beta_0$ be the constant in Proposition \ref{prop.est.F}.
%
\begin{theorem}\label{thm.est.vorticity.RSed.f}
Let $|n|=1$, $q\in(1,\infty)$, and $\tilde{q}\in(\max\{1,\frac{q}{2}\}, q]$. Fix $\epsilon\in(0,\frac{\pi}{2})$. Then there is a positive constant $C=C(q,\tilde{q},\epsilon)$ independent of $\beta$ such that the following statement holds. Let $f\in C^\infty_0(D)^2$ and $\beta\in(0,\beta_0)$. Set
\begin{align}
\omega^{{\rm ed}\,(1)}_{f,n}(r) 
\,=\, -\frac{c_{n,\lambda}[f_n]}{F_n(\sqrt{\lambda};\beta)} 
K_{\mu_n}(\sqrt{\lambda} r)\,, ~~~~~~~~
\omega^{{\rm ed}\,(2)}_{f,n}(r) 
\,=\, \Phi_{n,\lambda}[f_n](r)\,.
\label{def.thm.est.vorticity.RSed.f} 
\end{align}
Then for $\lambda\in\Sigma_{\pi-\epsilon} \cap \mathcal{B}_{e^{-\frac{1}{6\beta}}}(0)$ we have
\begin{align}
\|\omega^{{\rm ed}\,(1)}_{f,n} \|_{L^2(D)}
&\le 
\frac{C}{\beta^2} 
|\lambda|^{-1 + \frac1q} 
\|f\|_{L^q(D)}\,,
\label{est1.thm.est.vorticity.RSed.f} \\
\big\|\frac{\omega^{{\rm ed}\,(2)}_{f,n}}{|x|} \big\|_{L^{\tilde{q}}(D)}
&\le 
\frac{C}{\beta} |\lambda|^{-\frac{1}{\tilde{q}}+\frac1q} 
\|f\|_{L^q(D)}\,,
\label{est2.thm.est.vorticity.RSed.f} \\
\big|\big\langle \omega^{{\rm ed}\,(1)}_{f,n}, \frac{(w^{\rm ed}_{f,r})_n}{|x|} \big\rangle_{L^2(D)} \big|
&\le 
\frac{C}{\beta^5} |\lambda|^{-2+\frac2q} \|f \|_{L^q(D)}^2\,.
\label{est3.thm.est.vorticity.RSed.f}
\end{align}
Moreover, \eqref{est1.thm.est.vorticity.RSed.f}, \eqref{est2.thm.est.vorticity.RSed.f}, and \eqref{est3.thm.est.vorticity.RSed.f} hold all for $f\in L^q(D)^2$. 
\end{theorem}
%
\begin{proof}
(i) Estimate of $\omega^{{\rm ed}\,(1)}_{f,n}$: The estimate \eqref{est1.thm.est.vorticity.RSed.f} is a direct consequence of Proposition \ref{prop.est.F}, \eqref{est1.cor1.est.velocity.RSed.f} in Corollary \ref{cor1.est.velocity.RSed.f}, and \eqref{est1.lem.est4.bessel} with $p=2$ in Lemma \ref{lem.est4.bessel} in Appendix \ref{app.est.bessel}. \\
\noindent (ii) Estimate of $|x|^{-1} \omega^{{\rm ed}\,(2)}_{f,n}$: We decompose $\omega^{{\rm ed}\,(2)}_{f,n}$ into $\omega^{{\rm ed}\,(2)}_{f,n}=\sum_{l=1}^{4} \Phi_{n,\lambda}^{(l)}[f_n]$ by setting 
{\allowdisplaybreaks
\begin{align*}
\Phi_{n,\lambda}^{(1)}[f_n]
&\,=\, -K_{\mu_n}(\sqrt{\lambda} r) \int_{1}^{r} I_{\mu_n}(\sqrt{\lambda} s)\,g^{(1)}_n(s) \dd s\,, \\
\Phi_{n,\lambda}^{(2)}[f_n]
&\,=\, -\sqrt{\lambda} K_{\mu_n}(\sqrt{\lambda} r) \int_{1}^{r} sI_{\mu_n+1}(\sqrt{\lambda} s)\,f_{\theta,n}(s) \dd s\,, \\
\Phi_{n,\lambda}^{(3)}[f_n]
&\,=\,
I_{\mu_n}(\sqrt{\lambda} r) \int_{r}^{\infty} K_{\mu_n}(\sqrt{\lambda} s)\,g^{(2)}_n(s) \dd s\,, \\
\Phi_{n,\lambda}^{(4)}[f_n]
&\,=\,
\sqrt{\lambda} I_{\mu_n}(\sqrt{\lambda} r) \int_{r}^{\infty} sK_{\mu_n-1}(\sqrt{\lambda} s)\,f_{\theta,n}(s) \dd s\,.
\end{align*}
}
Then the assertion \eqref{est2.thm.est.vorticity.RSed.f} follows from the estimates of each term $|x|^{-1} \Phi_{n,\lambda}^{(l)}[f_n]$, $l\in\{1.2.3.4\}$. \\
\noindent (I) Estimates of $|x|^{-1} \Phi_{n,\lambda}^{(1)}[f_n]$ and $|x|^{-1} \Phi_{n,\lambda}^{(2)}[f_n]$: We give a proof only for $|x|^{-1} \Phi_{n,\lambda}^{(2)}[f_n]$ since the proof for $|x|^{-1} \Phi_{n,\lambda}^{(1)}[f_n]$ is similar. The Minkowski inequality leads to
\begin{align*}
\big\|\frac{\Phi_{n,\lambda}^{(2)}[f_n]}{|x|}\big\|_{L^{\tilde{q}}(D)}
& \,=\,
|\lambda|^\frac12
\bigg(
\int_{1}^{\infty} \bigg|\int_{1}^{r} r^{-1} K_{\mu_n}(\sqrt{\lambda} r) 
sI_{\mu_n+1}(\sqrt{\lambda} s)\,f_{\theta,n}(s) \dd s\bigg|^{\tilde{q}} r \dd r 
\bigg)^{\frac{1}{\tilde{q}}} \\
& \le
|\lambda|^\frac12
\int_1^\infty
|sI_{\mu_n+1}(\sqrt{\lambda} s)\,f_{\theta,n}(s)|
\bigg(
\int_{s}^{\infty} |r^{-1} K_{\mu_n}(\sqrt{\lambda} r)|^{\tilde{q}} r \dd r 
\bigg)^{\frac{1}{\tilde{q}}} \dd s\,.
\end{align*}
By \eqref{est5.lem.est1.bessel} and \eqref{est7.lem.est1.bessel} for $k=1$ in Lemma \ref{lem.est1.bessel} and \eqref{est2.lem.est4.bessel} and \eqref{est3.lem.est4.bessel} in Lemma \ref{lem.est4.bessel}, we have
\begin{align*}
\big\|\frac{\Phi_{n,\lambda}^{(2)}[f_n]}{|x|}\big\|_{L^{\tilde{q}}(D)} 
& \le
C |\lambda| \int_{1}^{\frac{1}{{\rm Re}(\sqrt{\lambda})}}
s^{\frac{2}{\tilde{q}}} |f_n(s)| s\dd s 
+ C\,|\lambda|^{-\frac{1}{2\tilde{q}}}
\int_{\frac{1}{{\rm Re}(\sqrt{\lambda})}}^{\infty}
s^{-2+\frac{1}{\tilde{q}}} |f_n(s)| s\dd s\,,
\end{align*}
which implies \eqref{est2.thm.est.vorticity.RSed.f} since $\frac{q}{q-1}(-2+\frac{1}{\tilde{q}})+2<0$ holds if $\tilde{q}\in(\max\{1,\frac{q}{2}\}, q]$. \\
\noindent (II) Estimates of $|x|^{-1} \Phi_{n,\lambda}^{(3)}[f_n]$ and $|x|^{-1} \Phi_{n,\lambda}^{(4)}[f_n]$: We give a proof only for $|x|^{-1} \Phi_{n,\lambda}^{(4)}[f_n]$. After using the Minkowski inequality in the same way as above, from \eqref{est2.lem.est1.bessel}, \eqref{est4.lem.est1.bessel}, and \eqref{est6.lem.est1.bessel} with $k=1$ in Lemma \ref{lem.est1.bessel} and \eqref{est4.lem.est4.bessel} and \eqref{est5.lem.est4.bessel} in Lemma \ref{lem.est4.bessel}, we have
\begin{align*}
&\big\|\frac{\Phi_{n,\lambda}^{(4)}[f_n]}{|x|}\big\|_{L^{\tilde{q}}(D)}
\le
C |\lambda|^\frac12
\int_{1}^{\infty} 
\big|sK_{\mu_n-1}(\sqrt{\lambda} s)\,f_{\theta,n}(s)\big|
\bigg(
\int_{1}^{s} |r^{-1} I_{\mu_n}(\sqrt{\lambda} r)|^{\tilde{q}} r \dd r 
\bigg)^{\frac{1}{\tilde{q}}}  
\dd s \\
& \le
C \beta^{-1} |\lambda|
\int_{1}^{\frac{1}{{\rm Re}(\sqrt{\lambda})}}
s^{\frac{2}{\tilde{q}}}\,|f_n(s)| s \dd s 
+ C |\lambda|^{-\frac{1}{2\tilde{q}}} 
\int_{\frac{1}{{\rm Re}(\sqrt{\lambda})}}^\infty
s^{-2+\frac{1}{\tilde{q}}} |f_n(s)| s \dd s\,,
\end{align*}
which leads to \eqref{est2.thm.est.vorticity.RSed.f}. Hence we obtain the assertion \eqref{est2.thm.est.vorticity.RSed.f}.\\
\noindent (iii) Estimate of $\big|\big\langle \omega^{{\rm ed}\,(1)}_{f,n}, |x|^{-1}(w^{\rm ed}_{f,r})_n \big\rangle_{L^2(D)} \big|$: From \eqref{rep.velocity.nFourier.RSed.f} and \eqref{def.thm.est.vorticity.RSed.f} we see that
\begin{equation}\label{est1.proof.thm.est.vorticity.RSed.f}
\begin{aligned}
\big|\big\langle \omega^{{\rm ed}\,(1)}_{f,n}, \frac{(w^{\rm ed}_{f,r})_n}{|x|} \big\rangle_{L^2(D)} \big| 
&\le 
\Big|\frac{c_{n,\lambda}[f_n]}{F_n(\sqrt{\lambda};\beta)}\Big|^2
\big|\big\langle K_{\mu_n}(\sqrt{\lambda}\,\cdot\,),
\frac{V_{r,n} [K_{\mu_n}(\sqrt{\lambda}\,\cdot\,)]}{|x|} \big\rangle_{L^2(D)} \big| \\
& \quad
+ \Big|\frac{c_{n,\lambda}[f_n]}{F_n(\sqrt{\lambda};\beta)}\Big|\,
\big|\big\langle K_{\mu_n}(\sqrt{\lambda}\,\cdot\,),
\frac{V_{r,n}[\Phi_{n,\lambda}[f_n]]}{|x|} \big\rangle_{L^2(D)} \big|\,.
\end{aligned}
\end{equation}
Then, by Proposition \ref{prop.est.F} and \eqref{est1.cor1.est.velocity.RSed.f} in Corollary \ref{cor1.est.velocity.RSed.f} combined with the results in Lemma \ref{lem.est1.bessel} for $k=0$ and \eqref{est1.proof.prop1.est.velocity.RSed.f} 
and \eqref{est2.proof.prop1.est.velocity.RSed.f} in the proof Proposition \ref{prop1.est.velocity.RSed.f}, we have
\begin{align*}
&\Big|\frac{c_{n,\lambda}[f_n]}{F_n(\sqrt{\lambda};\beta)}\Big|^2
\big|\big\langle K_{\mu_n}(\sqrt{\lambda}\,\cdot\,),
\frac{V_{r,n} [K_{\mu_n}(\sqrt{\lambda}\,\cdot\,)]}{|x|} \big\rangle_{L^2(D)} \big| \nonumber \\
& \le
C \beta^{-3} |\lambda|^{-2+\frac2q} \|f \|_{L^q(D)}^2
\bigg(
\int_1^{\frac{1}{{\rm Re}(\sqrt{\lambda})}} 
r^{-{\rm Re}(\mu_n)} \dd r
+ |\lambda|^{{\rm Re}(\mu_n)-\frac12}
\int_{\frac{1}{{\rm Re}(\sqrt{\lambda})}}^\infty
e^{-{\rm Re}(\sqrt{\lambda}) r} \dd r 
\bigg)\,.
\end{align*}
By \eqref{est1.proof.prop2.est.velocity.RSed.f} in the proof of Proposition \ref{prop2.est.velocity.RSed.f} combined with Lemmas \ref{lem1.est.velocity.RSed.f} and \ref{lem2.est.velocity.RSed.f}, we have
\begin{align*}
&\Big|\frac{c_{n,\lambda}[f_n]}{F_n(\sqrt{\lambda};\beta)}\Big|\,
\big|\big\langle K_{\mu_n}(\sqrt{\lambda}\,\cdot\,),
\frac{V_{r,n}[\Phi_{n,\lambda}[f_n]]}{|x|} \big\rangle_{L^2(D)} \big| \\
&\le
C \beta^{-2} |\lambda|^{-2+\frac2q} \|f \|_{L^q(D)}^2
\bigg(
\int_1^{\frac{1}{{\rm Re}(\sqrt{\lambda})}} 
r^{-{\rm Re}(\mu_n)} \dd r 
+ |\lambda|^{\frac{{\rm Re}(\mu_n)}{2}-\frac14} 
\int_{\frac{1}{{\rm Re}(\sqrt{\lambda})}}^\infty
r^{\frac12} e^{-{\rm Re}(\sqrt{\lambda}) r} \dd r
\bigg)\,.
\end{align*}
Hence, by inserting the above two estimates into \eqref{est1.proof.thm.est.vorticity.RSed.f}, one can check that the assertion \eqref{est3.thm.est.vorticity.RSed.f} holds. This completes the proof of Theorem \ref{thm.est.vorticity.RSed.f}.
\end{proof}
%
\subsection{Problem II: External force ${\rm div}\,F$ and Dirichlet condition}
\label{subsec.RSed.divF}
In this subsection we consider the following resolvent problem for $(w,r)=(w^{\rm ed}_{{\rm div}F}, r^{\rm ed}_{{\rm div}F})$:
\begin{equation}\tag{RS$^{\rm ed}_{{\rm div}F}$}\label{RSed.divF}
\left\{
\begin{aligned}
\lambda w - \Delta w + \beta U^{\bot} {\rm rot}\,w + \nabla r
& \,=\, 
{\rm div}\,F\,,~~~~x \in D\,, \\
{\rm div}\,w &\,=\, 0\,,~~~~x \in D\,, \\
w|_{\partial D} & \,=\, 0\,.
\end{aligned}\right.
\end{equation}
In particular, the estimates for the $\pm 1$-Fourier mode of $w^{\rm ed}_{{\rm div}F}$ are our interest. Here $F=(F_{ij}(x))_{1\le i,j\le 2}$ is a $2\times2$ matrix. We recall that the operator ${\rm div}$ on matrices $G=(G_{ij}(x))_{1\le i,j\le 2}$ is defined as ${\rm div}\,G=(\partial_1 G_{11} + \partial_2 G_{12}, \partial_1 G_{21} + \partial_2 G_{22})^\top$. The assumption on $F$ is as follows: let us take the constant $\gamma\in(\frac12,1)$ of Assumption \ref{assumption} in the introduction. Fix $\gamma'\in(\frac12,\gamma)$. Then we assume that $F$ belongs to the function space $X_{\gamma'}(D)$ defined as
\begin{align}\label{Xgamma}
X_{\gamma'}(D) \,=\,
\{F\in L^2(D)^{2\times 2}~|~|x|^{\gamma'} F\in L^2(D)^{2\times 2} \}\,.
\end{align}
This definition is motivated from the property of the matrix $R\otimes v + R\otimes v$ appearing in \eqref{RSed}, where $R$ is the function in Assumption \ref{assumption} and $v\in D(\mathbb{A}_V)$ is a solution to \eqref{RS}. In view of the regularity of $F$, we define the class of solutions to \eqref{RSed.divF} in each Fourier mode by the weak form. Let $n\in \Z\setminus\{0\}$ and $L^q_\sigma(D)$, $q\in(1,\infty)$, denote the $L^q$-closure of $C^\infty_{0,\sigma}(D)$, and let $p\in(\frac{2}{\gamma'},\infty)$. Then a velocity $w_n \in \mathcal{P}_n (L^{p}_{\sigma}(D) \cap W^{1,p}_0 (D)^2)$ is said to be a weak solution to \eqref{RSed.divF} replacing ${\rm div}\,F$ by $({\rm div}\,F)_n=\mathcal{P}_n{\rm div}\,F$ if
\begin{equation}\tag{RS$^{\rm ed}_{{\rm div}F,n}$}\label{RSed.divF.n}
\begin{aligned}
&\lambda \langle w_n, \varphi \rangle_{L^2(D)}
+ \langle \nabla w_n, \nabla\varphi \rangle_{L^2(D)}
+ \beta \langle U^{\bot} {\rm rot}\,w_n, \varphi \rangle_{L^2(D)} \\
&
\,=\, 
- \langle F, \nabla\mathcal{P}_n\varphi \rangle_{L^2(D)}
\end{aligned}
\end{equation}
holds for all $\varphi\in C^\infty_{0,\sigma}(D)^2$. Then the pressure $r\in W^{1,p}_{{\rm loc}}(\overline{D})$ is recovered by a standard functional analytic argument; see \cite[page 73, Lemma 2.21]{So} for instance. The uniqueness of weak solutions is trivial thanks to the representation formula \eqref{rep.velocity.nFourier.RSed.divF} below. In the following we consider the solutions to \eqref{RSed.divF.n} for given $F\in X_{\gamma'}(D)$.

Let $n\in\Z\setminus\{0\}$. By the solution formula \eqref{rep.velocity.nFourier.RSed.f} in Subsection \ref{subsec.RSed.f}, at least when $F\in C^\infty_0(D)^{2\times2}$, we can represent the $n$-Fourier mode of the solution $w^{\rm ed}_{{\rm div}F}$ to \eqref{RSed.divF} as 
\begin{align}\label{rep.velocity.nFourier.RSed.divF}
& w^{\rm ed}_{{\rm div}\,F,n}
\,=\,
-\frac{c_{n,\lambda}[({\rm div}\,F)_n]}{F_n(\sqrt{\lambda};\beta)} 
V_n [K_{\mu_n}(\sqrt{\lambda}\,\cdot\,)]
+ V_n[\Phi_{n,\lambda}[({\rm div}F)_n]]\,,
\end{align}
if $\lambda\in\C\setminus\overline{\R_{-}}$ satisfies $F_n(\sqrt{\lambda};\beta)\neq0$. Here $c_{n,\lambda}[\,\cdot\,]$, $F_n(\sqrt{\lambda};\beta)$, $V_n[\,\cdot\,]$, and $\Phi_{n,\lambda}[\,\cdot\,]$ are respectively defined in \eqref{const.nFourier.RSed.f}, \eqref{def.F}, \eqref{def.V_n}, and \eqref{Phi.nFourier.RSed.f}. Then the vorticity of $w^{\rm ed}_{{\rm div}F,n}$ is given by
\begin{align}\label{rep.vorticity.nFourier.RSed.divF}
{\rm rot}\,w^{\rm ed}_{{\rm div}F,n}
&\,=\, 
-\frac{c_{n,\lambda}[({\rm div}\,F)_n]}{F_n(\sqrt{\lambda};\beta)} 
K_{\mu_n}(\sqrt{\lambda} r) e^{in\theta}
+ \Phi_{n,\lambda}[({\rm div}\,F)_n](r) e^{in\theta}\,.
\end{align}
We prove the estimates of \eqref{rep.velocity.nFourier.RSed.divF} and \eqref{rep.vorticity.nFourier.RSed.divF} in the next two subsections. Before concluding this subsection, we prepare a useful lemma for the calculation concerning $\Phi_{n,\lambda}[({\rm div}\,F)_n]$.
%
\begin{lemma}\label{lem.rep.Phi.RSed.divF}
Let $n \in \Z \setminus \{0\}$ and $F \in C^\infty_0(D)^{2\times 2}$. Then there are functions $\widetilde{F}_n^{(k)}=\widetilde{F}_n^{(k)}(r)$, $k\in\{1,\ldots,7\}$, each of which is a linear combination containing the $n$-Fourier mode of the components of $F=(F_{ij})_{1\le i,j\le 2}$, such that $\Phi_{n,\lambda}[({\rm div}\,F)_n]$ is represented as
\begin{equation}\label{eq.lem.rep.Phi.RSed.divF}
\begin{aligned}
&~~~\Phi_{n,\lambda}[({\rm div}\,F)_n](r) \\
&\,=\, 
-K_{\mu_n}(\sqrt{\lambda} r)
\bigg( \int_{1}^{r} s^{-1} I_{\mu_n}(\sqrt{\lambda} s)\,\widetilde{F}_n^{(1)}(s)\dd s \\
&~~~~~~~~~~~~~~~~
+ \sqrt{\lambda} 
\int_{1}^{r} I_{\mu_n+1}(\sqrt{\lambda} s)\,\widetilde{F}_n^{(2)}(s) \dd s 
- \lambda 
\int_{1}^{r} sI_{\mu_n}(\sqrt{\lambda} s)\,\widetilde{F}_n^{(3)}(s) \dd s \bigg) \\
&\quad
+ I_{\mu_n}(\sqrt{\lambda} r)
\bigg( \int_r^\infty s^{-1} K_{\mu_n}(\sqrt{\lambda} s)\,\widetilde{F}_n^{(4)}(s) \dd s \\
&~~~~~~~~~~~~~~~~
+ \sqrt{\lambda} \int_r^\infty K_{\mu_n-1}(\sqrt{\lambda} s)\,\widetilde{F}_n^{(5)}(s) \dd s + \lambda \int_r^\infty sK_{\mu_n}(\sqrt{\lambda} s)\,\widetilde{F}_n^{(6)}(s) \dd s \bigg) \\
& \quad
- \sqrt{\lambda} r \big(K_{\mu_n}(\sqrt{\lambda} r) I_{\mu_n+1}(\sqrt{\lambda} r) + K_{\mu_n-1}(\sqrt{\lambda} r) I_{\mu_n}(\sqrt{\lambda} r)\big)\,\widetilde{F}_n^{(7)}(r)\,.
\end{aligned}
\end{equation}
\end{lemma}
%
\begin{proof}
Let $n\in\Z\setminus\{0\}$. By the definition of ${\rm div}\,F$, 
there are functions $G_n^{(l)}\in C^\infty_0((1,\infty))$, $l\in\{1,\ldots,4\}$, such that the $n$-Fourier mode $({\rm div}\,F)_n$ has a representation 
\begin{equation}\label{eq1.lem.rep.Phi.RSed.divF}
\begin{aligned}
&({\rm div}\,F)_n \,=\, 
({\rm div}\,F)_{r,n} e^{i n \theta} {\bf e}_r
+  ({\rm div}\,F)_{\theta,n} e^{i n \theta} {\bf e}_\theta \\
&\,=\,
\big(\partial_r G_n^{(1)}(r) + \frac{1}{r} G_n^{(2)}(r)\big) e^{i n \theta} {\bf e}_r
+ \big(\partial_r G_n^{(3)}(r) + \frac{1}{r} G_n^{(4)}(r)\big) e^{i n \theta} {\bf e}_\theta\,.
\end{aligned}
\end{equation}
Then there are functions $H_n^{(m)}\in C^\infty_0((1,\infty))$, $m\in\{1,\ldots,4\}$, each of which is a linear combination containing the $n$-mode of the components of $F=(F_{ij})_{1\le i,j\le 2}$, such that
\begin{align}
\mu_n ({\rm div}\,F)_{\theta,n}(r) + in ({\rm div}\,F)_{r,n}(r)
& \,=\, \partial_{r} H_n^{(1)}(r) + \frac1r H_n^{(2)}(r)\,, 
\label{eq2.lem.rep.Phi.RSed.divF} \\
\mu_n ({\rm div}\,F)_{\theta,n}(r) - in ({\rm div}\,F)_{r,n}(r)
& \,=\, \partial_{r} H_n^{(3)}(r) + \frac1r H_n^{(4)}(r)
\label{eq3.lem.rep.Phi.RSed.divF}\,.
\end{align}
By inserting \eqref{eq1.lem.rep.Phi.RSed.divF}--\eqref{eq3.lem.rep.Phi.RSed.divF} into the representation of $\Phi[f_n]$ in \eqref{Phi.nFourier.RSed.f} replacing $f_n$ by $({\rm div}\,F)_n$, and using the next relations of Bessel functions $ I_\mu(z)$ and $K_\mu(z)$ (see \cite{Abramowitz} page 376):
\begin{align*}
\frac{\dd I_{\mu}}{\dd z}(z) 
\,=\, \frac{\mu}{z} I_{\mu}(z) + I_{\mu+1}(z)\,,~~~~~~~~
\frac{\dd K_{\mu}}{\dd z}(z) 
\,=\, -\frac{\mu}{z} K_{\mu}(z) - K_{\mu-1}(z)\,,
\end{align*}
we can obtain the assertion \eqref{eq.lem.rep.Phi.RSed.divF}. We omit the details since the calculations are straightforward using integration by parts. The proof is complete.
\end{proof}
%
\subsubsection{Estimates of the velocity solving \eqref{RSed.divF.n} with $|n|=1$}
\label{subsec.RSed.divF.velocity}
The main result of this subsection is the estimates of $w^{\rm ed}_{{\rm div}\,F,n}$ represented as in \eqref{rep.velocity.nFourier.RSed.divF}. Let us recall that $\beta_0$ is the constant in Proposition \ref{prop.est.F}.
%
\begin{theorem}\label{thm.est.velocity.RSed.divF}
Let $|n|=1$, $\gamma'\in(\frac12,\gamma)$, and $p\in(\frac{2}{\gamma'},\infty)$. Fix $\epsilon\in(0,\frac{\pi}{2})$. Then there is a positive constant $C=C(\gamma',p,\epsilon)$ independent of $\beta$ such that the following statement holds. Let $F\in C^\infty_0(D)^{2\times2}$ and $\beta\in(0,\beta_0)$. Then for $\lambda\in\Sigma_{\pi-\epsilon} \cap \mathcal{B}_{e^{-\frac{1}{6\beta}}}(0)$ we have 
\begin{align}
\|w^{\rm ed}_{{\rm div}F,n}\|_{L^p(D)}
& \le 
\frac{C}{\beta^2}
|\lambda|^{-\frac1p} 
\||x|^{\gamma'}F\|_{L^2(D)}\,,
\label{est1.thm.est.velocity.RSed.divF} \\
\big\| \frac{w^{\rm ed}_{{\rm div}F,n}}{|x|} \big\|_{L^2(D)}
& \le 
\frac{C}{\beta^3}
\||x|^{\gamma'}F\|_{L^2(D)}\,.
\label{est2.thm.est.velocity.RSed.divF}
\end{align}
Moreover, \eqref{est1.thm.est.velocity.RSed.divF} and \eqref{est2.thm.est.velocity.RSed.divF} hold all for $F\in X_{\gamma'}(D)$ defined in \eqref{Xgamma}.
\end{theorem}
%

By following a similar procedure as in Subsection \ref{subsec.RSed.f.velocity}, we give the proof of Theorem \ref{thm.est.velocity.RSed.divF} at the end of this subsection. We firstly focus on the term $V_n[\Phi_{n,\lambda}[({\rm div}\,F)_n]]$ in \eqref{rep.velocity.nFourier.RSed.divF}. By using Lemma \ref{lem.rep.Phi.RSed.divF}, one can see that the next decomposition holds. Let $\widetilde{F}_n^{(k)}(r)$, $k\in\{1,\ldots,7\}$, be the functions in Lemma \ref{lem.rep.Phi.RSed.divF}.
%
\begin{lemma}\label{lem1.velocitydecom.RSed.divF}
Let $n \in \Z \setminus \{0\}$ and $F \in C^\infty_0(D)^{2\times 2}$. Then we have 
\begin{align}\label{eq1.lem1.velocitydecom.RSed.divF} 
\frac{1}{r^{|n|}} \int_1^r s^{1+|n|} \Phi_{n,\lambda}[({\rm div}\,F)_n](s) \dd s
\,=\, \sum_{l=1}^{10} J^{(2)}_l[({\rm div}\,F)_n](r)\,,
\end{align}
where 
{\allowdisplaybreaks
\begin{align*}
J^{(2)}_{1}[({\rm div}\,F)_n](r)
&\,=\, -\frac{1}{r^{|n|}} 
\int_1^r \tau^{-1} I_{\mu_n}(\sqrt{\lambda} \tau)\,\widetilde{F}_n^{(1)}(\tau)
\int_\tau^r s^{1+|n|} K_{\mu_n}(\sqrt{\lambda} s) \dd s \dd \tau\,, \\
J^{(2)}_{2}[({\rm div}\,F)_n](r)
&\,=\, -\frac{\sqrt{\lambda}}{r^{|n|}} 
\int_1^r I_{\mu_n+1}(\sqrt{\lambda} \tau)\,\widetilde{F}_n^{(2)}(\tau)
\int_\tau^r s^{1+|n|} K_{\mu_n}(\sqrt{\lambda} s) \dd s \dd \tau\,, \\
J^{(2)}_{3}[({\rm div}\,F)_n](r)
&\,=\, -\frac{\lambda}{r^{|n|}} 
\int_1^r \tau I_{\mu_n}(\sqrt{\lambda} \tau)\,\widetilde{F}_n^{(3)}(\tau)
\int_\tau^r s^{1+|n|} K_{\mu_n}(\sqrt{\lambda} s) \dd s \dd \tau\,, \\
J^{(2)}_{4}[({\rm div}\,F)_n](r)
&\,=\, \frac{1}{r^{|n|}} 
\int_1^r \tau^{-1} K_{\mu_n}(\sqrt{\lambda} \tau)\,\widetilde{F}_n^{(4)}(\tau)
\int_1^\tau s^{1+|n|} I_{\mu_n}(\sqrt{\lambda} s) \dd s \dd \tau\,, \\
J^{(2)}_{5}[({\rm div}\,F)_n](r)
&\,=\, \frac{1}{r^{|n|}} 
\bigg( \int_r^\infty \tau^{-1} K_{\mu_n}(\sqrt{\lambda} \tau)\,\widetilde{F}_n^{(4)}(\tau) \dd \tau \bigg)
\bigg( \int_1^r s^{1+|n|} I_{\mu_n}(\sqrt{\lambda} s) \dd s \bigg)\,, \\
J^{(2)}_{6}[({\rm div}\,F)_n](r)
&\,=\, \frac{\sqrt{\lambda}}{r^{|n|}} 
\int_1^r K_{\mu_n-1}(\sqrt{\lambda} \tau)\,\widetilde{F}^{(5)}_n(\tau)
\int_1^\tau s^{1+|n|} I_{\mu_n}(\sqrt{\lambda} s) \dd s \dd \tau\,, \\
J^{(2)}_{7}[({\rm div}\,F)_n](r)
&\,=\, \frac{\sqrt{\lambda}}{r^{|n|}} 
\bigg( \int_r^\infty K_{\mu_n-1}(\sqrt{\lambda} \tau)\,\widetilde{F}_n^{(5)}(\tau) \dd \tau \bigg)
\bigg( \int_1^r s^{1+|n|} I_{\mu_n}(\sqrt{\lambda} s) \dd s \bigg)\,, \\
J^{(2)}_{8}[({\rm div}\,F)_n](r)
&\,=\, \frac{\lambda}{r^{|n|}} 
\int_1^r \tau K_{\mu_n}(\sqrt{\lambda} \tau)\,\widetilde{F}_n^{(6)}(\tau)
\int_1^\tau s^{1+|n|} I_{\mu_n}(\sqrt{\lambda} s) \dd s \dd \tau\,, \\
J^{(2)}_{9}[({\rm div}\,F)_n](r)
&\,=\, \frac{\lambda}{r^{|n|}} 
\bigg( \int_r^\infty \tau K_{\mu_n}(\sqrt{\lambda} \tau)\,\widetilde{F}_n^{(6)}(\tau) \dd \tau \bigg)
\bigg( \int_1^r s^{1+|n|} I_{\mu_n}(\sqrt{\lambda} s) \dd s \bigg)\,, \\
J^{(2)}_{10}[({\rm div}\,F)_n](r)
&\,=\, -\frac{\sqrt{\lambda}}{r^{|n|}} 
\int_1^r s \big(K_{\mu_n}(\sqrt{\lambda} s) I_{\mu_n+1}(\sqrt{\lambda} s) + K_{\mu_n-1}(\sqrt{\lambda} s) I_{\mu_n}(\sqrt{\lambda} s)\big)\,\widetilde{F}^{(7)}_n(s) \dd s\,.
\end{align*}
}
Here $\widetilde{F}_n^{(k)}(r)$, $k\in\{1,\ldots,7\}$, are the functions in Lemma \ref{lem.rep.Phi.RSed.divF}.
\end{lemma}
%
\begin{proof}
The assertion follows by inserting \eqref{eq.lem.rep.Phi.RSed.divF} in Lemma \ref{lem.rep.Phi.RSed.divF} into the left-hand side of \eqref{eq1.lem1.velocitydecom.RSed.divF}, and changing order of integration as $\int_{1}^{r} \int_{1}^{s} \dd \tau \dd s= \int_{1}^{r} \int_\tau^r \dd s \dd \tau$ and $\int_1^r \int_s^\infty \dd \tau \dd s=\int_1^r  \int_1^\tau \dd s \dd \tau + \int_r^\infty \dd \tau \int_1^r \dd s$. This completes the proof.
\end{proof}
%
%
The next lemma gives the estimates to $J^{(2)}_{l}[({\rm div}\,F)_n]$, $l\in\{1,\ldots,10\}$, in Lemma \ref{lem1.velocitydecom.RSed.divF}.
%
\begin{lemma}\label{lem1.est.velocity.RSed.divF}
Let $|n|=1$ and $\gamma' \in(\frac12,\gamma)$, and let $\lambda\in\Sigma_{\pi-\epsilon} \cap \mathcal{B}_1(0)$ for some $\epsilon\in(0,\frac{\pi}{2})$. Then there is a positive constant $C=C(\gamma',\epsilon)$ independent of $\beta$ such that the following statement holds. Let $F \in C^\infty_0(D)^{2\times 2}$. Then for $l \in \{1,\cdots,10\}$ we have
\begin{align}
& \big|J^{(2)}_{l}[({\rm div}\,F)_n](r) \big|
\le 
\frac{C}{\beta} (|\lambda|^\frac12 r^2
+ r^{2-{\rm Re}(\mu_n)} + r^{1-\gamma'})\,
\||x|^{\gamma'} F \|_{L^2(D)}\,,\nonumber \\
&
~~~~~~~~~~~~~~~~~~~~~~~~~~~~~~~~~~~~~~~~~~~~~~~~~~~~~~~~~~~~
~~~~~~~~~~~~~~~~~~~~~~~~~~~~~~~~
1 \le r < {\rm Re}(\sqrt{\lambda})^{-1}\,,
\label{est1.lem1.est.velocity.RSed.divF} \\
& \big|J^{(2)}_{l}[({\rm div}\,F)_n](r) \big| 
\le 
\frac{C}{\beta}
(|\lambda|^{-\frac12} + r^{1-\gamma'})\,
\||x|^{\gamma'} F \|_{L^2(D)}\,,~~~~~~
r \ge {\rm Re}(\sqrt{\lambda})^{-1}
\label{est2.lem1.est.velocity.RSed.divF}\,.
\end{align}
\end{lemma}
%
{\color{black}
\begin{remark}\label{remark.est1.velocity.RSed.divF}
The statement of Lemma \ref{lem1.est.velocity.RSed.divF} is in fact valid for $\gamma' \in(0, \gamma)$ if we choose the constant $\beta\in(0,1)$ small enough depending on $\gamma'$. However, in order to avoid the technical difficulty, a simplification has been made here by assuming the lower bound $\gamma' \in(\frac12,\gamma)$. We note that this assumption is essentially required later in the proof of Theorem \ref{thm.est.vorticity.RSed.divF}.
\end{remark}
}
%
\begin{proofx}{Lemma \ref{lem1.est.velocity.RSed.divF}}
(i) Estimate of $J^{(2)}_{1}[({\rm div}\,F)_n]$: For $1\le r<{\rm Re}(\sqrt{\lambda})^{-1}$, by \eqref{est5.lem.est1.bessel} for $k=0$ in Lemma \ref{lem.est1.bessel} and \eqref{est1.lem.est2.bessel} for $k=0$ in Lemma \ref{lem.est2.bessel} in Appendix \ref{app.est.bessel}, we find
\begin{align*}
|J^{(2)}_{1}[({\rm div}\,F)_n](r)|
& \le 
Cr^{-1}
\int_1^r |\tau^{-1} I_{\mu_n}(\sqrt{\lambda} \tau)\,\widetilde{F}_n^{(1)}(\tau)|\,
\bigg|\int_\tau^r s^2 K_{\mu_n}(\sqrt{\lambda} s) \dd s \bigg| \dd \tau \\
& \le 
Cr^{2-{\rm Re}(\mu_n)}
\int_1^r \tau^{{\rm Re}(\mu_n)-2} |F_n(\tau)| \tau \dd \tau\,,
\end{align*}
which implies $|J^{(2)}_{1}[({\rm div}\,F)_n](r)| \le Cr^{2-{\rm Re}(\mu_n)} \||x|^{\gamma'} F \|_{L^2}$. 
{\color{black} We note that ${\rm Re}(\mu_n)-2 >-1$ for any $\beta\in(0,1)$ and the condition $\gamma' \in(\frac12, \gamma)$ has been applied.}
For $r \ge {\rm Re}(\sqrt{\lambda})^{-1}$, by \eqref{est5.lem.est1.bessel} and \eqref{est7.lem.est1.bessel} for $k=0$ in Lemma \ref{lem.est1.bessel} and \eqref{est2.lem.est2.bessel} and \eqref{est3.lem.est2.bessel} for $k=0$ in Lemma \ref{lem.est2.bessel}, we have
\begin{align*}
&|J^{(2)}_{1}[({\rm div}\,F)_n](r)| \\
& \le 
Cr^{-1}
\bigg(
\int_1^{\frac{1}{{\rm Re}(\sqrt{\lambda})}} 
+ \int_{\frac{1}{{\rm Re}(\sqrt{\lambda})}}^r \bigg)
|\tau^{-1} I_{\mu_n}(\sqrt{\lambda} \tau)\,\widetilde{F}_n^{(1)}(\tau)|\,
\bigg| \int_\tau^{r} s^2 K_{\mu_n}(\sqrt{\lambda} s) \dd s \bigg| \dd \tau \\
& \le 
C\,|\lambda|^{-\frac12} 
\int_1^{\frac{1}{{\rm Re}(\sqrt{\lambda})}} 
\tau^{-1} |F_n(\tau)| \tau \dd \tau 
+ C\,|\lambda|^{-\frac12}  
\int_{\frac{1}{{\rm Re}(\sqrt{\lambda})}}^r 
\tau^{-1} |F_n(\tau)| \tau \dd \tau\,,
\end{align*}
which leads to $|J^{(2)}_{1}[({\rm div}\,F)_n](r)| \le C|\lambda|^{-\frac12} \||x|^{\gamma'} F \|_{L^2}$. \\
\noindent (ii) Estimate of $J^{(2)}_{2}[({\rm div}\,F)_n]$: In the similar manner as the proof of $J^{(2)}_1[({\rm div}\,F)_n]$, for $1\le r<{\rm Re}(\sqrt{\lambda})^{-1}$ we have $|J^{(2)}_{2}[({\rm div}\,F)_n](r)|\le C|\lambda|^\frac12 r^2 \|F \|_{L^2}$, and for $r \ge {\rm Re}(\sqrt{\lambda})^{-1}$ we have $|J^{(2)}_{2}[({\rm div}\,F)_n](r)| \le C |\lambda|^{-\frac12} \|F \|_{L^2}$. We omit since the proof is straightforward. \\
\noindent (iii) Estimate of $J^{(2)}_{3}[({\rm div}\,F)_n]$: For $1\le r<{\rm Re}(\sqrt{\lambda})^{-1}$, we have $|J^{(2)}_{3}[({\rm div}\,F)_n](r)|\le C|\lambda|^\frac12 r^2 \|F \|_{L^2}$ by same way as the proof of $J^{(2)}_1[f_n]$. For $r \ge {\rm Re}(\sqrt{\lambda})^{-1}$, we observe that
\begin{align*}
&|J^{(2)}_{3}[({\rm div}\,F)_n](r)| \\
& \le 
C\,|\lambda|\,r^{-1}
\bigg( \int_1^{\frac{1}{{\rm Re}(\sqrt{\lambda})}} 
+ \int_{\frac{1}{{\rm Re}(\sqrt{\lambda})}}^r \bigg)
|\tau I_{\mu_n}(\sqrt{\lambda} \tau)\,\widetilde{F}_n^{(3)}(\tau)|\,
\bigg| \int_\tau^r s^2 K_{\mu_n}(\sqrt{\lambda} s) \dd s \bigg| \dd \tau \\
& \le 
C\,|\lambda|^{-\frac12} r^{-1}
\int_1^{\frac{1}{{\rm Re}(\sqrt{\lambda})}}
|F_n(\tau)| \tau \dd \tau 
+ C\,r^{-1}
\int_{\frac{1}{{\rm Re}(\sqrt{\lambda})}}^r
\tau |F_n(\tau)| \tau \dd \tau \,.
\end{align*}
Thus we have $|J^{(2)}_{3}[({\rm div}\,F)_n](r)| \le C( |\lambda|^{-\frac12} \|F \|_{L^2} + r^{1-\gamma'} \||x|^{\gamma'} F \|_{L^2})$.\\
\noindent (iv) Estimate of $J^{(2)}_{4}[({\rm div}\,F)_n]$: For $1\le r < {\rm Re}(\sqrt{\lambda})^{-1}$, by \eqref{est1.lem.est1.bessel} and \eqref{est3.lem.est1.bessel} in Lemma \ref{lem.est1.bessel} and \eqref{est1.lem.est3.bessel} for $k=0$ in Lemma \ref{lem.est3.bessel}, we find
\begin{align*}
|J^{(2)}_{4}[({\rm div}\,F)_n](r)|
& \le 
r^{-1}
\int_1^r |\tau^{-1} K_{\mu_n}(\sqrt{\lambda} \tau)\,\widetilde{F}^{(4)}_n(\tau)|
\int_1^\tau |s^2 I_{\mu_n}(\sqrt{\lambda} s)| \dd s \dd \tau \\
& \le 
{\color{black} C r^{-1}}
\int_1^r \tau |F_n(\tau)| \tau \dd \tau\,,
\end{align*}
which implies $|J^{(2)}_{4}[({\rm div}\,F)_n](r)| \le C r^{1-\gamma'}\,\||x|^{\gamma'} F \|_{L^2}$. For $r \ge {\rm Re}(\sqrt{\lambda})^{-1}$, \eqref{est1.lem.est1.bessel}, \eqref{est3.lem.est1.bessel}, and \eqref{est6.lem.est1.bessel} for $k=0$ in Lemma \ref{lem.est1.bessel} and \eqref{est1.lem.est3.bessel} and \eqref{est2.lem.est3.bessel} for $k=0$ in Lemma \ref{lem.est3.bessel} yield
\begin{align*}
&~~~|J^{(2)}_{4}[({\rm div}\,F)_n](r)| \\
& \le 
C r^{-1}
\bigg( \int_1^{\frac{1}{{\rm Re}(\sqrt{\lambda})}} 
+ \int_{\frac{1}{{\rm Re}(\sqrt{\lambda})}}^r \bigg)
|\tau^{-1} K_{\mu_n}(\sqrt{\lambda} \tau)\,\widetilde{F}_n^{(4)}(\tau)|\,
\int_1^\tau |s^2 I_{\mu_n}(\sqrt{\lambda} s)| \dd s \dd \tau \\
& \le 
C |\lambda|^{-\frac12} r^{-1}
\int_1^{\frac{1}{{\rm Re}(\sqrt{\lambda})}}
|F_n(\tau)| \tau \dd \tau 
+ C\,|\lambda|^{-\frac12} r^{-1} 
\int_{\frac{1}{{\rm Re}(\sqrt{\lambda})}}^r
|F_n(\tau)| \tau \dd \tau\,,
\end{align*}
which leads to $|J^{(2)}_{4}[({\rm div}\,F)_n](r)| \le C|\lambda|^{-\frac12} \|F\|_{L^2}$. \\
\noindent (v) Estimate of $J^{(2)}_{5}[({\rm div}\,F)_n]$: For $1\le r < {\rm Re}(\sqrt{\lambda})^{-1}$, 
{\color{black}
by the same estimates in Lemmas \ref{lem.est1.bessel} and \ref{lem.est3.bessel} which have been used in (iv) 
}
we find
\begin{align*}
&~~~|J^{(2)}_{5}[({\rm div}\,F)_n](r)| \\
& \le 
r^{-1}
\int_1^r |s^2 I_{\mu_n}(\sqrt{\lambda} s)| \dd s  
\bigg(
\int_r^{\frac{1}{{\rm Re}(\sqrt{\lambda})}}  
+ \int_{\frac{1}{{\rm Re}(\sqrt{\lambda})}}^\infty 
\bigg)
|\tau^{-1} K_{\mu_n}(\sqrt{\lambda} \tau)\,\widetilde{F}^{(4)}_n(\tau)| \dd \tau \\
& \le 
C\,r^{{\rm Re}(\mu_n)+2}
\int_r^{\frac{1}{{\rm Re}(\sqrt{\lambda})}}  
\tau^{-{\rm Re}(\mu_n)-2} |F_n(\tau)| \tau \dd \tau \\
& \quad
+ C\,|\lambda|^{\frac{{\rm Re}(\mu_n)}{2}-\frac14} r^{{\rm Re}(\mu_n)+2}
\int_{\frac{1}{{\rm Re}(\sqrt{\lambda})}}^\infty
\tau^{-\frac52} e^{-{\rm Re}(\sqrt{\lambda}) \tau} |F_n(\tau)| \tau \dd \tau \,,
\end{align*}
and thus we see that $|J^{(2)}_{5}[({\rm div}\,F)_n](r)| \le C(r^{1-\gamma'} \||x|^{\gamma'}\,F \|_{L^2} + |\lambda|^{\frac12} r^{2} \| F \|_{L^2})$ holds. For $r \ge {\rm Re}(\sqrt{\lambda})^{-1}$, we have 
\begin{align*}
|J^{(2)}_{5}[({\rm div}\,F)_n](r)|
& \le 
r^{-1}
\int_1^r |s^2 I_{\mu_n}(\sqrt{\lambda} s)| \dd s
\int_r^\infty |\tau^{-1} K_{\mu_n}(\sqrt{\lambda} \tau)\,\widetilde{F}^{(4)}(\tau)| \dd \tau  \\
& \le 
C\,|\lambda|^{-1} r^{\frac12} e^{{\rm Re}(\sqrt{\lambda}) r}
\int_r^\infty
\tau^{-\frac52} e^{-{\rm Re}(\sqrt{\lambda}) \tau} |F_n(\tau)| \tau \dd \tau\,,
\end{align*}
which implies $|J^{(2)}_{5}[({\rm div}\,F)_n](r)| \le C |\lambda|^{-\frac12} \| F \|_{L^2}$. \\
\noindent (vi) Estimates of $J^{(2)}_l[({\rm div}\,F)_n]$, $l\in\{6,7,8,9\}$: In the similar manner as the proofs for $J^{(2)}_4[f_n]$ and $J^{(2)}_5[({\rm div}\,F)_n]$, we see that
\begin{align*}
|J^{(2)}_l[({\rm div}\,F)_n](r)| 
& \le
C\beta^{-1}|\lambda|^\frac12 r^2 \|F \|_{L^2}\,,~~~~
1\le r < {\rm Re}(\sqrt{\lambda})^{-1}\,, \\
|J^{(2)}_l[({\rm div}\,F)_n](r)| 
& \le 
C(\beta^{-1} |\lambda|^{-\frac12} \|F \|_{L^2}
+ r^{1-\gamma'} \||x|^{\gamma'} F\|_{L^2})\,,~~~~
r \ge {\rm Re}(\sqrt{\lambda})^{-1}\,,
\end{align*}
for $l\in\{6,7,8,9\}$. We omit the details since the calculations are straightforward. \\
\noindent (vii) Estimate of $J^{(2)}_{10}[({\rm div}\,F)_n]$: For $1\le r < {\rm Re}(\sqrt{\lambda})^{-1}$, 
{\color{black} 
from \eqref{est1.lem.est1.bessel}--\eqref{est4.lem.est1.bessel}, 
and \eqref{est5.lem.est1.bessel} for $k=0,1$ in Lemma \ref{lem.est1.bessel}
}
we have
\begin{align*}
|J^{(2)}_{10}[({\rm div}\,F)_n](r)| 
\le
C \beta^{-1} |\lambda| 
\int_1^r |F_n(s)| s \dd s\,,
\end{align*}
which implies $|J^{(2)}_{10}[({\rm div}\,F)_n](r)| \le C \beta^{-1} |\lambda|^\frac12 r^2 \|F \|_{L^2}$. For $r \ge {\rm Re}(\sqrt{\lambda})^{-1}$, 
{\color{black} 
from \eqref{est1.lem.est1.bessel}--\eqref{est4.lem.est1.bessel}, 
and \eqref{est5.lem.est1.bessel}--\eqref{est7.lem.est1.bessel} for $k=0,1$ in Lemma \ref{lem.est1.bessel}
}
we have
\begin{align*}
|J^{(2)}_{10}[({\rm div}\,F)_n](r)| 
& \le 
C \beta^{-1} |\lambda|\,r^{-1}
\int_1^{\frac{1}{{\rm Re}(\sqrt{\lambda})}} s |F_n(s)| s\dd s
+ C\,r^{-1}
\int_{\frac{1}{{\rm Re}(\sqrt{\lambda})}}^r s^{-1} |F_n(s)| s\dd s\,, 
\end{align*}
which leads to $|J^{(2)}_{10}[({\rm div}\,F)_n](r)|\le C(\beta^{-1} |\lambda|^{-\frac12} \|F \|_{L^2} + r^{1-\gamma'} \||x|^{\gamma'} F\|_{L^2})$. This completes the proof of Lemma  \ref{lem1.est.velocity.RSed.divF}.
\end{proofx}
%
We continue the analysis on $V_n[\Phi_{n,\lambda}[({\rm div}\,F)_n]]$ in \eqref{rep.velocity.nFourier.RSed.divF}. The next decomposition is also useful in calculation as is Lemma \ref{lem1.velocitydecom.RSed.divF}.
%
\begin{lemma}\label{lem2.velocitydecom.RSed.divF} 
Let $n \in \Z \setminus \{0\}$ and $F \in C^\infty_0(D)^{2\times 2}$. Then we have
\begin{align}\label{eq1.lem2.velocitydecom.RSed.divF}
r^{|n|} \int_r^\infty s^{1-|n|} \Phi_{n,\lambda}[({\rm div}\,F)_n](s) \dd s
\,=\, \sum_{l=11}^{20} J^{(2)}_l[({\rm div}\,F)_n](r)\,,
\end{align}
where
{\allowdisplaybreaks
\begin{align*}
J^{(2)}_{11}[({\rm div}\,F)_n](r)
&\,=\, -r^{|n|} 
\int_1^r \tau^{-1} I_{\mu_n}(\sqrt{\lambda} \tau)\,\widetilde{F}_n^{(1)}(\tau)
\int_r^\infty s^{1-|n|} K_{\mu_n}(\sqrt{\lambda} s) \dd s \dd \tau\,, \\
J^{(2)}_{12}[({\rm div}\,F)_n](r)
&\,=\, -r^{|n|} 
\int_r^\infty \tau^{-1} I_{\mu_n}(\sqrt{\lambda} \tau)\,\widetilde{F}_n^{(1)}(\tau)
\int_\tau^\infty s^{1-|n|} K_{\mu_n}(\sqrt{\lambda} s) \dd s \dd \tau\,, \\
J^{(2)}_{13}[({\rm div}\,F)_n](r)
&\,=\, -\sqrt{\lambda}\,r^{|n|} 
\int_1^r I_{\mu_n+1}(\sqrt{\lambda} \tau)\,\widetilde{F}_n^{(2)}(\tau)
\int_r^\infty s^{1-|n|} K_{\mu_n}(\sqrt{\lambda} s) \dd s \dd \tau\,, \\
J^{(2)}_{14}[({\rm div}\,F)_n](r)
&\,=\, -\sqrt{\lambda}\,r^{|n|}  
\int_r^\infty I_{\mu_n+1}(\sqrt{\lambda} \tau)\,\widetilde{F}_n^{(2)}(\tau)
\int_\tau^\infty s^{1-|n|} K_{\mu_n}(\sqrt{\lambda} s) \dd s \dd \tau\,, \\
J^{(2)}_{15}[({\rm div}\,F)_n](r)
&\,=\, \lambda\,r^{|n|} 
\int_1^r \tau I_{\mu_n}(\sqrt{\lambda} \tau)\,\widetilde{F}_n^{(3)}(\tau)
\int_r^\infty s^{1-|n|} K_{\mu_n}(\sqrt{\lambda} s) \dd s \dd \tau\,, \\
J^{(2)}_{16}[({\rm div}\,F)_n](r)
&\,=\, \lambda\,r^{|n|}  
\int_r^\infty \tau I_{\mu_n}(\sqrt{\lambda} \tau)\,\widetilde{F}_n^{(3)}(\tau)
\int_\tau^\infty s^{1-|n|} K_{\mu_n}(\sqrt{\lambda} s) \dd s \dd \tau\,, \\
J^{(2)}_{17}[({\rm div}\,F)_n](r)
&\,=\, r^{|n|}  
\int_r^\infty \tau^{-1} K_{\mu_n}(\sqrt{\lambda} \tau)\,\widetilde{F}_n^{(4)}(\tau)
\int_r^\tau s^{1-|n|} I_{\mu_n}(\sqrt{\lambda} s) \dd s \dd \tau\,, \\
J^{(2)}_{18}[({\rm div}\,F)_n](r)
&\,=\, \sqrt{\lambda}\,r^{|n|}
\int_r^\infty K_{\mu_n-1}(\sqrt{\lambda} \tau)\,\widetilde{F}_n^{(5)}(\tau)
\int_r^\tau s^{1-|n|} I_{\mu_n}(\sqrt{\lambda} s) \dd s \dd \tau\,, \\
J^{(2)}_{19}[({\rm div}\,F)_n](r)
&\,=\, \lambda\,r^{|n|}
\int_r^\infty s K_{\mu_n}(\sqrt{\lambda} \tau)\,\widetilde{F}_n^{(6)}(\tau)
\int_r^\tau s^{1-|n|} I_{\mu_n}(\sqrt{\lambda} s) \dd s \dd \tau\,, \\
J^{(2)}_{20}[({\rm div}\,F)_n](r)
&\,=\, 
- \sqrt{\lambda}\,r^{|n|}
\int_r^\infty
s \big(K_{\mu_n}(\sqrt{\lambda} s) I_{\mu_n+1}(\sqrt{\lambda} s) + K_{\mu_n-1}(\sqrt{\lambda} s) I_{\mu_n}(\sqrt{\lambda} s)\big)\,\widetilde{F}_n^{(7)}(s) \dd s\,.
\end{align*}
}
Here $\widetilde{F}_n^{(k)}(r)$, $k\in\{1,\ldots,7\}$, are the functions in Lemma \ref{lem.rep.Phi.RSed.divF}.
\end{lemma}
%
\begin{proof}
The assertion is a consequence of inserting \eqref{eq.lem.rep.Phi.RSed.divF} in Lemma \ref{lem.rep.Phi.RSed.divF} into the left-hand side of \eqref{eq1.lem2.velocitydecom.RSed.divF}, and changing order of integration as $\int_r^\infty \int_1^s \dd \tau \dd s= \int_1^r \dd \tau \int_r^\infty \dd s + \int_r^\infty \int_\tau^\infty \dd s \dd \tau$ and $\int_r^\infty \int_s^\infty \dd \tau \dd s = \int_r^\infty  \int_r^\tau \dd s \dd \tau$. This completes the proof.
\end{proof}
%
The next lemma summarizes the estimates to $J^{(2)}_{l}[f_n]$, $l\in\{11,\ldots,20\}$, in Lemma \ref{lem2.velocitydecom.RSed.divF}.
%
\begin{lemma}\label{lem2.est.velocity.RSed.divF}
Let $|n|=1$ and $\gamma' \in(\frac12,\gamma)$, and let $\lambda\in\Sigma_{\pi-\epsilon} \cap \mathcal{B}_1(0)$ for some $\epsilon\in(0,\frac{\pi}{2})$. Then there is a positive constant $C=C(\gamma',\epsilon)$ independent of $\beta$ such that the following statement holds. Let $F \in C^\infty_0(D)^{2\times 2}$. Then for $l \in \{11,\cdots,20\}$ we have
\begin{align}
& |J^{(2)}_{l}[({\rm div}\,F)_n](r)| 
\le 
\frac{C}{\beta} 
\big( |\lambda|^\frac12 r^2 + r^{2-{\rm Re}(\mu_n)} + r^{1-\gamma'} \big)\,
\||x|^{\gamma'}F \|_{L^2(D)}\,, \nonumber \\
& ~~~~~~~~~~~~~~~~~~~~~~~~~~~~~~~~~~~~~~~~~~~~~~~~~~~~~~~~~~~~
~~~~~~~~~~~~~~~~~~~~~~~~~~~~~
1\le r < {\rm Re}(\sqrt{\lambda})^{-1}\,,
\label{est1.lem2.est.velocity.RSed.divF} \\
&|J^{(2)}_{l}[({\rm div}\,F)_n](r)| 
\le
C \big(|\lambda|^{-\frac12} + r^{1-\gamma'} \big)\,
\||x|^{\gamma'}F \|_{L^2(D)}\,,~~~~~~
r \ge {\rm Re}(\sqrt{\lambda})^{-1}\,.
\label{est2.lem2.est.velocity.RSed.divF}
\end{align}
\end{lemma}
%
{\color{black} 
\begin{remark}\label{remark.est2.velocity.RSed.divF} 
We make a simplification in Lemma \ref{lem2.est.velocity.RSed.divF} by assuming $\gamma' \in(\frac12,\gamma)$ as in Lemma \ref{remark.est1.velocity.RSed.divF}. It still holds for any $\gamma' \in(0, \gamma)$ if the constant $\beta=\beta(\gamma')$ is chosen sufficiently small. On the other hand, the condition $\gamma' \in(\frac12,\gamma)$ is needed in the proof of Theorem \ref{thm.est.vorticity.RSed.divF}.
\end{remark}
}
%
\begin{proofx}{Lemma \ref{lem2.est.velocity.RSed.divF}}
(i) Estimate of $J^{(2)}_{11}[({\rm div}\,F)_n]$: For $1\le r < {\rm Re}(\sqrt{\lambda})^{-1}$ , by \eqref{est5.lem.est1.bessel} for $k=0$ in Lemma \ref{lem.est1.bessel} and \eqref{est4.lem.est2.bessel} for $k=0$ in Lemma \ref{lem.est2.bessel} in Appendix \ref{app.est.bessel}, we find
\begin{align*}
|J^{(2)}_{11}[({\rm div}\,F)_n](r)|
& \le 
r \bigg|\int_r^\infty K_{\mu_n}(\sqrt{\lambda} s) \dd s \bigg|
\int_1^r |\tau^{-1} I_{\mu_n}(\sqrt{\lambda} \tau)\,\widetilde{F}_n^{(1)}(\tau)| \dd \tau \\
& \le 
C \beta^{-1} r^{2-{\rm Re}(\mu_n)}
\int_1^r \tau^{{\rm Re}(\mu_n)-2} |F_n(\tau)| \tau \dd \tau\,,
\end{align*}
which implies $|J^{(2)}_{11}[({\rm div}\,F)_n](r)| \le C \beta^{-1} r^{2-{\rm Re}(\mu_n)} \||x|^{\gamma'} F \|_{L^2}$. For $r \ge {\rm Re}(\sqrt{\lambda})^{-1}$, by \eqref{est5.lem.est1.bessel} and \eqref{est7.lem.est1.bessel} for $k=0$ in Lemma \ref{lem.est1.bessel} and \eqref{est5.lem.est2.bessel} for $k=0$ in Lemma \ref{lem.est2.bessel}, we see that
\begin{align*}
&|J^{(2)}_{11}[({\rm div}\,F)_n](r)| 
\le 
r \int_r^\infty |K_{\mu_n}(\sqrt{\lambda} s)| \dd s
\bigg(
\int_1^{\frac{1}{{\rm Re}(\sqrt{\lambda})}} 
+ \int_{\frac{1}{{\rm Re}(\sqrt{\lambda})}}^r \bigg)
|\tau^{-1} I_{\mu_n}(\sqrt{\lambda} \tau)\,\widetilde{F}_n^{(1)}(\tau)|\dd \tau \\
& \le 
C |\lambda|^{\frac{{\rm Re}(\mu_n)}{2}-\frac34} r^{\frac12} 
e^{-{\rm Re}(\sqrt{\lambda}) r}
\int_1^{\frac{1}{{\rm Re}(\sqrt{\lambda})}}
\tau^{{\rm Re}(\mu_n)-2} |F_n(\tau)| \tau \dd \tau \\
& \quad
+ C |\lambda|^{-1} r^{\frac12} 
e^{-{\rm Re}(\sqrt{\lambda}) r}
\int_{\frac{1}{{\rm Re}(\sqrt{\lambda})}}^r 
\tau^{-\frac52} e^{{\rm Re}(\sqrt{\lambda}) \tau}
 |F_n(\tau)| \tau \dd \tau \,.
\end{align*}
Thus we have $|J^{(2)}_{11}[({\rm div}\,F)_n](r)| \le C |\lambda|^{-\frac12} \||x|^{\gamma'}F \|_{L^2}$. \\
\noindent (ii) Estimate of $J^{(2)}_{12}[({\rm div}\,F)_n]$: For $1\le r < {\rm Re}(\sqrt{\lambda})^{-1}$, by \eqref{est5.lem.est1.bessel} and \eqref{est7.lem.est1.bessel} for $k=0$ in Lemma \ref{lem.est1.bessel} and \eqref{est4.lem.est2.bessel} and \eqref{est5.lem.est2.bessel} for $k=0$ in Lemma \ref{lem.est2.bessel}, we observe that
\begin{align*}
& |J^{(2)}_{12}[({\rm div}\,F)_n](r)| 
\le 
\,r
\bigg( \int_r^{\frac{1}{{\rm Re}(\sqrt{\lambda})}} 
+ \int_{\frac{1}{{\rm Re}(\sqrt{\lambda})}}^\infty  \bigg)
|\tau^{-1} I_{\mu_n}(\sqrt{\lambda} \tau)\,\widetilde{F}_n^{(1)}(\tau)|\,
\bigg|\int_\tau^\infty K_{\mu_n}(\sqrt{\lambda} s) \dd s \bigg| \dd \tau \\
& \le 
C \beta^{-1} r
\int_r^{\frac{1}{{\rm Re}(\sqrt{\lambda})}} 
\tau^{-1} |F_n(\tau)| \tau \dd \tau 
+ C |\lambda|^{-1} r
\int_{\frac{1}{{\rm Re}(\sqrt{\lambda})}}^\infty
\tau^{-3} |F_n(\tau)| \tau \dd \tau\,,
\end{align*}
which implies $|J^{(2)}_{12}[({\rm div}\,F)_n](r)| \le C \beta^{-1} r^{1-\gamma'} \||x|^{\gamma'}F \|_{L^2}$. For $r \ge {\rm Re}(\sqrt{\lambda})^{-1}$, by \eqref{est7.lem.est1.bessel} for $k=0$ in Lemma \ref{lem.est1.bessel} and \eqref{est5.lem.est2.bessel} for $k=0$ in Lemma \ref{lem.est2.bessel} we find
\begin{align*}
|J^{(2)}_{12}[({\rm div}\,F)_n](r)| 
& \le 
C r \int_r^\infty 
|\tau^{-1} I_{\mu_n}(\sqrt{\lambda} \tau)\,\widetilde{F}_n^{(1)}(\tau)|\,
\int_\tau^\infty |K_{\mu_n}(\sqrt{\lambda} s)| \dd s \dd \tau \\
& \le 
C |\lambda|^{-1} r 
\int_r^\infty 
\tau^{-3} |F_n(\tau)| \tau \dd \tau\,,
\end{align*}
which leads to $|J^{(2)}_{12}[({\rm div}\,F)_n](r)| \le C |\lambda|^{-\frac12} \|F \|_{L^2}$. \\
\noindent (iii) Estimates of $J^{(2)}_l[({\rm div}\,F)_n]$, $l\in\{13,14,15,16\}$: In the similar manner as the proofs of $J^{(2)}_{11}[f_n]$ and $J^{(2)}_{12}[({\rm div}\,F)_n]$,  we have
\begin{align*}
|J^{(2)}_l[({\rm div}\,F)_n](r)| 
& \le 
C \beta^{-1}
(|\lambda|^\frac12 r^2 \|F \|_{L^2}
+ r^{1-\gamma'} \||x|^{\gamma'}F \|_{L^2})\,,~~~~
1\le r < {\rm Re}(\sqrt{\lambda})^{-1}\,, \\
|J^{(2)}_l[({\rm div}\,F)_n](r)| 
& \le 
C(|\lambda|^{-\frac12} \|F \|_{L^2}
+ r^{1-\gamma'} \||x|^{\gamma'}F\|_{L^2})\,,~~~~
r \ge {\rm Re}(\sqrt{\lambda})^{-1}\,,
\end{align*}
for $l\in\{13,14,15,16\}$. We omit the details since the calculations are straightforward. \par
\noindent (iv) Estimates of $J^{(2)}_l[({\rm div}\,F)_n]$, $l\in\{17,18,19\}$: We give a proof only for $J^{(2)}_{19}[({\rm div}\,F)_n]$ since the proofs for $J^{(2)}_{17}[({\rm div}\,F)_n]$ and $J^{(2)}_{18}[({\rm div}\,F)_n]$ are similar. For $1\le r < {\rm Re}(\sqrt{\lambda})^{-1}$, from \eqref{est1.lem.est1.bessel}, \eqref{est3.lem.est1.bessel}, and \eqref{est6.lem.est1.bessel} for $k=0$ in Lemma \ref{lem.est1.bessel} and \eqref{est3.lem.est3.bessel} and \eqref{est4.lem.est3.bessel} for $k=0$ in Lemma \ref{lem.est3.bessel}, we observe that
\begin{align*}
&|J^{(2)}_{19}[({\rm div}\,F)_n](r)|
\le 
|\lambda| r
\bigg(
\int_r^{\frac{1}{{\rm Re}(\sqrt{\lambda})}} 
+ \int_{\frac{1}{{\rm Re}(\sqrt{\lambda})}}^\infty
\bigg)
|\tau K_{\mu_n}(\sqrt{\lambda} \tau)\,\widetilde{F}_n^{(6)}(\tau)|\,
\int_r^\tau |I_{\mu_n}(\sqrt{\lambda} s)| \dd s \dd \tau \\
& \le 
C |\lambda| r
\int_r^{\frac{1}{{\rm Re}(\sqrt{\lambda})}} 
\tau |F_n(\tau)| \tau \dd \tau 
+ C r
\int_{\frac{1}{{\rm Re}(\sqrt{\lambda})}}^\infty
\tau^{-1} |F_n(\tau)| \tau \dd \tau\,,
\end{align*}
which implies $|J^{(2)}_{{\color{black}19}}[({\rm div}\,F)_n](r)| \le C r^{1-\gamma'} \||x|^{\gamma'}F \|_{L^2}$. For $r \ge {\rm Re}(\sqrt{\lambda})^{-1}$, by \eqref{est6.lem.est1.bessel} for $k=0$ in Lemma \ref{lem.est1.bessel} and \eqref{est5.lem.est3.bessel} in Lemma \ref{lem.est3.bessel} for $k=0$, we have
\begin{align*}
|J^{(2)}_{19}[({\rm div}\,F)_n](r)| 
& \le 
|\lambda|\,r
\int_r^\infty
|\tau K_{\mu_n}(\sqrt{\lambda} \tau)\,\widetilde{F}_n^{(6)}(\tau)|\,
\int_r^\tau |I_{\mu_n}(\sqrt{\lambda} s)| \dd s \dd \tau \\
& \le 
C\,r
\int_r^\infty 
\tau^{-1} |F_n(\tau)| \tau \dd \tau\,,
\end{align*}
which leads to $|J^{(2)}_{19}[({\rm div}\,F)_n](r)|\le Cr^{1-\gamma'} \||x|^{\gamma'}F \|_{L^2}$. \\
\noindent (v) Estimate of $J^{(2)}_{20}[({\rm div}\,F)_n]$: For $1\le r < {\rm Re}(\sqrt{\lambda})^{-1}$, 
{\color{black} 
from \eqref{est1.lem.est1.bessel}--\eqref{est4.lem.est1.bessel}, 
and \eqref{est5.lem.est1.bessel}--\eqref{est7.lem.est1.bessel} for $k=0,1$ in Lemma \ref{lem.est1.bessel},
}
we have
\begin{align*}
|J^{(2)}_{20}[({\rm div}\,F)_n](r)| 
\le 
C\beta^{-1}|\lambda|\,r
\int_r^{\frac{1}{{\rm Re}(\sqrt{\lambda})}} 
s |F_n(s)| s \dd s 
+ C\,r
\int_{\frac{1}{{\rm Re}(\sqrt{\lambda})}}^\infty
s^{-1} |F_n(s)| s \dd s\,.
\end{align*}
Thus we have $|J^{(2)}_{20}[({\rm div}\,F)_n](r)| \le C\beta^{-1} r^{1-\gamma'}
\||x|^{\gamma'}F \|_{L^2}$. For $r \ge {\rm Re}(\sqrt{\lambda})^{-1}$, 
{\color{black} 
from \eqref{est6.lem.est1.bessel} and \eqref{est7.lem.est1.bessel} for $k=0,1$ in Lemma \ref{lem.est1.bessel},
}
we have
\begin{align*}
|J^{(2)}_{20}[({\rm div}\,F)_n](r)| 
& \le 
C\,r
\int_r^\infty 
\tau^{-1} |F_n(\tau)| \tau \dd \tau\,,
\end{align*}
which implies $|J^{(2)}_{20}[({\rm div}\,F)_n](r)|\le C r^{1-\gamma'}\||x|^{\gamma'}F \|_{L^2}$. This completes the proof of Lemma \ref{lem2.est.velocity.RSed.divF}.
\end{proofx}
%
From Lemmas \ref{lem1.est.velocity.RSed.divF} and \ref{lem2.est.velocity.RSed.divF} we see that the following estimates hold.
%
\begin{corollary}\label{cor1.est.velocity.RSed.divF}
Let $|n|=1$, $\gamma' \in(\frac12,\gamma)$, and $p\in(\frac{2}{\gamma'},\infty)$, 
and let $\lambda\in\Sigma_{\pi-\epsilon} \cap \mathcal{B}_1(0)$ for some $\epsilon\in(0,\frac{\pi}{2})$. Then there is a positive constant $C=C(\gamma',p,\epsilon)$ independent of $\beta$ such that the following statement holds. Let $F\in C^\infty_0(D)^{2\times2}$. Then for $l\in\{1,\ldots,20\}$ we have
\begin{align}
|c_{n,\lambda}[({\rm div}\,F)_n]|
&\le
\frac{C}{\beta}
\||x|^{\gamma'}F\|_{L^2(D)}\,,
\label{est1.cor1.est.velocity.RSed.divF} \\
\|r^{-1} J^{(2)}_l[({\rm div}\,F)_n]\|_{L^p(D)}
&\le
\frac{C}{\beta} |\lambda|^{-\frac1p}
\||x|^{\gamma'}F\|_{L^2(D)}\,,
\label{est2.cor1.est.velocity.RSed.divF} \\
\|r^{-2} J^{(2)}_l[({\rm div}\,F)_n]\|_{L^2(D)}
& \le 
\frac{C}{\beta^2}
\||x|^{\gamma'}F\|_{L^2(D)}\,.
\label{est3.cor1.est.velocity.RSed.divF}
\end{align}
Here $c_{n,\lambda}[({\rm div}\,F)_n]$ is the constant in \eqref{const.nFourier.RSed.f} replacing $f_n$ by $({\rm div}\,F)_n$.
\end{corollary}
%
%
\begin{proof}
(i) Estimate of $c_{n,\lambda}[({\rm div}\,F)_n]$: By the definitions of $J^{(2)}_{l}[f_n]$ for $l \in \{11,\cdots,20\}$ in Lemma \ref{lem2.velocitydecom.RSed.divF}, we see that $|c_{n,\lambda}[({\rm div}\,F)_n]| \le \sum_{l=12,14,16,17,18,19,20} |J^{(2)}_l[f_n](1)|$. Hence we obtain the estimate \eqref{est1.cor1.est.velocity.RSed.divF} by putting $r=1$ to \eqref{est1.lem2.est.velocity.RSed.divF} in Lemma \ref{lem2.est.velocity.RSed.divF}. \\
\noindent (ii) Estimate of $r^{-1} J^{(2)}_l[({\rm div}\,F)_n]$: By Lemmas \ref{lem1.est.velocity.RSed.divF} and \ref{lem2.est.velocity.RSed.divF}, for $p\in[\frac{2}{\gamma'},\infty)$ we have
\begin{align*}
\sup_{r\ge1} r^{\frac2p} |r^{-1} J^{(2)}_l[({\rm div}\,F)_n](r)|
& \le C \beta^{-1} |\lambda|^{-\frac1p} \||x|^{\gamma'}F\|_{L^2(D)}\,.
\end{align*}
Thus by the Marcinkiewicz interpolation theorem we obtain \eqref{est2.cor1.est.velocity.RSed.divF} for $p\in(\frac{2}{\gamma'},\infty)$. \\
\noindent (iii) Estimate of $r^{-2} J^{(2)}_l[({\rm div}\,F)_n]$: The assertion 
\eqref{est3.cor1.est.velocity.RSed.divF} can be checked easily by using Lemmas \ref{lem1.est.velocity.RSed.divF} and \ref{lem2.est.velocity.RSed.divF} and $({\rm Re}(\mu_n)-1)^\frac12\approx O(\beta)$. This completes the proof.
\end{proof}
%
The next proposition gives the estimate for the term $V_n[\Phi_{n,\lambda}[({\rm div}\,F)_n]]$ in \eqref{rep.velocity.nFourier.RSed.divF}.
%
\begin{proposition}\label{prop2.est.velocity.RSed.divF}
Let $|n|=1$, $\gamma' \in(\frac12,\gamma)$, and $p\in(\frac{2}{\gamma'},\infty)$, and let $\lambda\in\Sigma_{\pi-\epsilon} \cap \mathcal{B}_1(0)$ for some $\epsilon\in(0,\frac{\pi}{2})$. Then there is a positive constant $C=C(\gamma',p,\epsilon)$ independent of $\beta$ such that for $F\in C^\infty_0(D)^{2\times2}$ we have
\begin{align}
\|V_n[\Phi_{n,\lambda}[({\rm div}\,F)_n]]\|_{L^p(D)}
&\le
\frac{C}{\beta} |\lambda|^{-\frac1p} 
\||x|^{\gamma'}F\|_{L^2(D)}\,,
\label{est1.prop2.est.velocity.RSed.divF} \\
\big\|\frac{V_n[\Phi_{n,\lambda}[({\rm div}\,F)_n]]}{|x|}\big\|_{L^2(D)}
& \le 
\frac{C}{\beta^2}
\||x|^{\gamma'}F\|_{L^2(D)}\,.
\label{est2.prop2.est.velocity.RSed.divF}
\end{align}
\end{proposition}
%
\begin{proof}
In the similar manner as the proof of Proposition \ref{prop2.est.velocity.RSed.f} we find
\begin{align*}
&|V_n[\Phi_{n,\lambda}[({\rm div}\,F)_n]](r)| \\
&\le C \Big( r^{-2} \sum_{l=1}^{20} |J^{(1)}_l[({\rm div}\,F)_n](1)|
+ \sum_{l=1}^{20} |r^{-1} J^{(1)}[({\rm div}\,F)_n](r)| \Big)\,.
\end{align*}
Thus the assertions \eqref{est1.prop2.est.velocity.RSed.divF} and \eqref{est2.prop2.est.velocity.RSed.divF} follow from Corollary \ref{cor1.est.velocity.RSed.divF}. The proof is complete.
\end{proof}
%

From Corollary \ref{cor1.est.velocity.RSed.divF} and Proposition \ref{prop2.est.velocity.RSed.divF}, Theorem \ref{thm.est.velocity.RSed.divF} follows.
\begin{proofx}{Theorem \ref{thm.est.velocity.RSed.divF}} 
(i) Estimate for the case $F\in C^\infty_0(D)^{2\times2}$: It suffices to prove that the first term in the right-hand side of \eqref{rep.velocity.nFourier.RSed.divF} has the estimates \eqref{est1.thm.est.velocity.RSed.divF} and \eqref{est2.thm.est.velocity.RSed.divF} in view of Proposition \ref{prop2.est.velocity.RSed.divF}. By using Proposition \ref{prop.est.F} and \eqref{est1.cor1.est.velocity.RSed.divF} in Corollary \ref{cor1.est.velocity.RSed.divF}, we see that \eqref{est1.thm.est.velocity.RSed.divF} and \eqref{est2.thm.est.velocity.RSed.divF} respectively follow from \eqref{est1.prop1.est.velocity.RSed.f} and \eqref{est2.prop1.est.velocity.RSed.f} in Proposition \ref{prop1.est.velocity.RSed.f}. \\
\noindent (ii) Estimate for the case $F\in X_{\gamma'}(D)$: Let us take sequences $\{G^{(m)}\}_{m=1}^\infty\subset C^\infty_0(D)^{2\times2}$ and $\{w_n^{(m)} \}_{n=1}^\infty\subset \mathcal{P}_n(L^{p}_{\sigma}(D) \cap W^{1,p}_0(D)^2)$ such that $\displaystyle{\lim_{m\to\infty}\||x|^{\gamma'}(F-G^{(m)})\|_{L^2(D)}=0}$ and $w_n^{(m)}$ is a (unique) solution to \eqref{RSed.divF.n} replacing $F$ by $G^{(m)}$. 
{\color{black} 
Then we see that $w_n^{(m)}$ satisfies  \eqref{est1.thm.est.velocity.RSed.divF}, \eqref{est2.thm.est.velocity.RSed.divF}, and the estimates in Theorem \ref{thm.est.vorticity.RSed.divF} below replacing $F$ by $G^{(m)}$. By extending $w_n^{(m)}$ by zero to $\R^2$ and denoting it again by $w_n^{(m)}$, we see that $w_n^{(m)}\in W^{1,p}(\R^2)\cap L^p_{\sigma}(\R^2)$ from $w_n^{(m)}|_{\partial D}=0$ and that $-\Delta w_n^{(m)} = \nabla^{\bot} {\rm rot}\,w_n^{(m)}$, $\nabla^{\bot}=(\partial_2, -\partial_1)^\top$, in $\R^2$ from ${\rm div}\,w_n^{(m)}=0$. Hence, thanks to $\|\nabla \nabla^{\bot}\,(-\Delta_{\R^2})^{-1} h\|_{L^p(\R^2)} \le C \|h\|_{L^p(\R^2)}$, we have the inequality $\|\nabla w_n^{(m)}\|_{L^p(D)} \le C \|{\rm rot}\,w_n^{(m)}\|_{L^p(D)}$. Then we observe that the limit $\displaystyle{w_n=\lim_{m\to\infty} w_n^{(m)} \in \mathcal{P}_n(L^p_\sigma(D)\cap W^{1,p}_0(D)^2})$ exists and satisfies \eqref{est1.thm.est.velocity.RSed.divF}, \eqref{est2.thm.est.velocity.RSed.divF},
and the estimates in Theorem \ref{thm.est.vorticity.RSed.divF}.
}
Moreover, by taking the limit $m\to\infty$ in \eqref{RSed.divF.n} replacing $F$ by $G^{(m)}$, we see that $w_n$ gives a weak solution to \eqref{RSed.divF.n}. The proof is complete. 
\end{proofx}
%
%
\subsubsection{Estimates of the vorticity for \eqref{RSed.divF.n} with $|n|=1$}
\label{subsec.RSed.divF.vorticity}
In this subsection we estimate the vorticity $\omega^{\rm ed}_{{\rm div}F,n}(r)=({\rm rot}\,w^{\rm ed}_{{\rm div}F,n})e^{-in\theta}$ with $|n|=1$, where ${\rm rot}\,w^{\rm ed}_{{\rm div}F,n}$ is represented as \eqref{rep.vorticity.nFourier.RSed.divF}. We take the constant $\beta_0$ in Proposition \ref{prop.est.F}.
%
\begin{theorem}\label{thm.est.vorticity.RSed.divF}
Let $|n|=1$, $\gamma'\in(\frac12,\gamma)$, $p\in[2,\infty)$, and $q\in(1,\infty)$. Fix $\epsilon\in(0,\frac{\pi}{2})$. Then there is a positive constant $C=C(\gamma',p,q,\epsilon)$ independent of $\beta$ such that the following statement holds. Let $F\in C^\infty_0(D)^{2\times2}$, $f\in L^q(D)^2$, and $\beta\in(0,\beta_0)$. Set 
\begin{align}
\omega^{\rm ed\,(1)}_{{\rm div}F,n}(r)
\,=\, -\frac{c_{n,\lambda}[({\rm div}\,F)_n]}{F_n(\sqrt{\lambda};\beta)} 
K_{\mu_n}(\sqrt{\lambda} r)\,, ~~~~~~
\omega^{\rm ed\,(2)}_{{\rm div}F,n}(r)
\,=\, \Phi_{n,\lambda}[({\rm div}\,F)_n](r)\,.
\label{def.thm.est.vorticity.RSed.divF}
\end{align}
Then for $\lambda\in\Sigma_{\pi-\epsilon} \cap \mathcal{B}_{e^{-\frac{1}{6\beta}}}(0)$ we have
\begin{align}
\|\omega^{\rm ed\,(1)}_{{\rm div}F,n}\|_{L^p(D)}
&\le 
\frac{C}{\beta (p {\rm Re}(\mu_n)-2)^\frac1p}
\| |x|^{\gamma'} F \|_{L^2(D)}\,, 
\label{est1.thm.est.vorticity.RSed.divF} \\
\|\omega^{\rm ed\,(2)}_{{\rm div}F,n}\|_{L^p(D)}
+ \beta \big\|\frac{\omega^{\rm ed\,(2)}_{{\rm div}F,n}}{|x|} \big\|_{L^1(D)}
&\le 
\frac{C}{\beta} \| |x|^{\gamma'} F \|_{L^2(D)}\,,
\label{est2.thm.est.vorticity.RSed.divF} \\
\big|\big\langle \omega^{\rm ed\,(1)}_{{\rm div}F,n}, \frac{(w^{\rm ed}_{f,r})_n}{|x|} \big\rangle_{L^2(D)} \big|
&\le 
\frac{C}{\beta^5} |\lambda|^{-1+\frac1q} \|f \|_{L^q(D)}
\| |x|^{\gamma'} F \|_{L^2(D)}\,.
\label{est3.thm.est.vorticity.RSed.divF}
\end{align}
Moreover, \eqref{est1.thm.est.vorticity.RSed.divF},  \eqref{est2.thm.est.vorticity.RSed.divF}, and \eqref{est3.thm.est.vorticity.RSed.divF} hold all for $F\in X_{\gamma'}(D)$ defined in \eqref{Xgamma} by a density argument as in the proof of Theorem \ref{thm.est.velocity.RSed.divF} above.
\end{theorem}
%
\begin{proof}
(i) Estimate of $\omega^{\rm ed\,(1)}_{{\rm div}F,n}$: The estimate \eqref{est1.thm.est.vorticity.RSed.divF} is a direct consequence of Proposition \ref{prop.est.F}, \eqref{est1.cor1.est.velocity.RSed.divF} in Corollary \ref{cor1.est.velocity.RSed.divF}, and \eqref{est1.lem.est4.bessel} in Lemma \ref{lem.est4.bessel} in Appendix \ref{app.est.bessel}. \\
\noindent (ii) Estimates of $\omega^{\rm ed\,(2)}_{{\rm div}F,n}$ and $|x|^{-1} \omega^{\rm ed\,(2)}_{{\rm div}F,n}$: Firstly we decompose $\omega^{\rm ed\,(2)}_{{\rm div}F,n}$ into $\omega^{\rm ed\,(2)}_{{\rm div}F,n} = \sum_{l=1}^{7} \Phi_{n,\lambda}^{(l)}[({\rm div}\,F)_n]$ as in Lemma \ref{lem.rep.Phi.RSed.divF}. Then the assertion \eqref{est2.thm.est.vorticity.RSed.divF} follows from the estimates of each term 
$\Phi_{n,\lambda}^{(l)}[({\rm div}\,F)_n]$, $l\in\{1,\ldots,7\}$. \\
\noindent (I) Estimates of $\Phi_{n,\lambda}^{(l)}[({\rm div}\,F)_n]$, $l\in\{1,2,3\}$: We give a proof only for $\Phi_{n,\lambda}^{(3)}[({\rm div}\,F)_n]$ since the proofs for $\Phi_{n,\lambda}^{(1)}[({\rm div}\,F)_n]$ and $\Phi_{n,\lambda}^{(2)}[({\rm div}\,F)_n]$ are similar. The Minkowski inequality and the Fubini theorem lead to 
\begin{align*}
&\|\Phi_{n,\lambda}^{(3)}[({\rm div}\,F)_n]\|_{L^p(D)}
+ \beta \big\|\frac{\Phi_{n,\lambda}^{(3)}[({\rm div}\,F)_n]}{|x|}\big\|_{L^1(D)} \\
& \le
|\lambda|
\int_{1}^{\infty}
|s I_{\mu_n}(\sqrt{\lambda} s) \widetilde{F}_n^{(3)}(s)|
\bigg(
\bigg( \int_{s}^{\infty} | K_{\mu_n}(\sqrt{\lambda} r)|^p r \dd r \bigg)^\frac1p
+ \beta \int_{s}^{\infty} |K_{\mu_n}(\sqrt{\lambda} r)| \dd r
\bigg)
\dd s \,.
\end{align*}
By \eqref{est5.lem.est1.bessel} and \eqref{est7.lem.est1.bessel} for $k=0$ in Lemma \ref{lem.est1.bessel} and \eqref{est6.lem.est4.bessel} and \eqref{est7.lem.est4.bessel} in Lemma \ref{lem.est4.bessel}, we have
\begin{align*}
&\|\Phi_{n,\lambda}^{(3)}[({\rm div}\,F)_n]\|_{L^p(D)}
+ \beta \big\|\frac{\Phi_{n,\lambda}^{(3)}[({\rm div}\,F)_n]}{|x|}\big\|_{L^1(D)} \\
& \le
C \beta^{-1} |\lambda|^\frac12
\int_{1}^{\frac{1}{{\rm Re}(\sqrt{\lambda})}} 
|F_n(s)| s \dd s
+ C |\lambda|^\frac14
\int_{\frac{1}{{\rm Re}(\sqrt{\lambda})}}^{\infty}
s^{-\frac12-\gamma'} |s^{\gamma'} F_n(s)| s\dd s\,,
\end{align*}
which implies \eqref{est2.thm.est.vorticity.RSed.divF} since the condition $\gamma'\in(\frac12,1)$ is assumed. \\
\noindent (II) Estimates of $\Phi_{n,\lambda}^{(l)}[({\rm div}\,F)_n]$, $l\in\{4,5,6\}$: We give a proof only for $\Phi_{n,\lambda}^{(6)}[({\rm div}\,F)_n]$ since the proofs for $\Phi_{n,\lambda}^{(4)}[({\rm div}\,F)_n]$ and $\Phi_{n,\lambda}^{(5)}[({\rm div}\,F)_n]$ are similar. After using the Minkowski inequality and the Fubini theorem, by \eqref{est1.lem.est1.bessel}, \eqref{est3.lem.est1.bessel},  and \eqref{est6.lem.est1.bessel} for $k=0$ in Lemma \ref{lem.est1.bessel} and \eqref{est8.lem.est4.bessel} and \eqref{est9.lem.est4.bessel} in Lemma \ref{lem.est4.bessel}, we observe that
\begin{align*}
&
\|\Phi_{n,\lambda}^{(6)}[({\rm div}\,F)_n]\|_{L^p(D)} 
+ \big\|\frac{\Phi_{n,\lambda}^{(6)}[({\rm div}\,F)_n]}{|x|}\big\|_{L^1(D)} \\
&\le
|\lambda|
\int_{1}^{\infty}
\big|sK_{\mu_n}(\sqrt{\lambda} s)\,\widetilde{F}_n^{(6)}(s)\big|
\bigg( 
\bigg( \int_1^s |I_{\mu_n}(\sqrt{\lambda} r)|^p r\dd r \bigg)^\frac1p
+ \int_1^s |I_{\mu_n}(\sqrt{\lambda} r)| \dd r 
\bigg)
\dd s \\
&\le
C |\lambda|^\frac12
\int_{1}^{\frac{1}{{\rm Re}(\sqrt{\lambda})}}
|F_n(s)|s\dd s
+ C |\lambda|^\frac14
\int_{\frac{1}{{\rm Re}(\sqrt{\lambda})}}^\infty
s^{-\frac12-\gamma'} |s^{\gamma'} F_n(s)| s\dd s\,,
\end{align*}
which leads to \eqref{est2.thm.est.vorticity.RSed.divF} by the condition $\gamma'\in(\frac12,1)$. \\
(III) Estimate of $\Phi_{n,\lambda}^{(7)}[({\rm div}\,F)_n]$: The proof is straightforward using the results in Lemma \ref{lem.est1.bessel} and thus we omit the details. \\
\noindent (iii) Estimate of $\big|\big\langle \omega^{\rm ed\,(1)}_{{\rm div}F,n}, |x|^{-1}(w^{\rm ed}_{f,r})_n \big\rangle_{L^2(D)} \big|$: We omit since the proof is parallel to that for \eqref{est3.thm.est.vorticity.RSed.f} in Theorem \ref{thm.est.vorticity.RSed.f} using \eqref{est1.cor1.est.velocity.RSed.divF} in Corollary \ref{cor1.est.velocity.RSed.divF}. The proof is complete.
\end{proof}
%
\subsection{Problem III: No external force and boundary data $b$}
\label{subsec.RSed.b}
%
In this subsection we give the estimates for $(w,r)=(w^{\rm ed}_{b}, r^{\rm ed}_{b})$ solving the next problem:
\begin{equation}\tag{RS$^{\rm ed}_{b}$}\label{RSed.b}
\left\{
\begin{aligned}
\lambda w - \Delta w + \beta U^{\bot} {\rm rot}\,w + \nabla r
& \,=\, 
0\,,~~~~x \in D\,, \\
{\rm div}\,w &\,=\, 0\,,~~~~x \in D\,, \\
w|_{\partial D} & \,=\, b\,.
\end{aligned}\right.
\end{equation}
Firstly we prove the representation formula to the problem \eqref{RSed.b}.
%
\begin{lemma}\label{lem.rep.RSed.b}
Let $|n|=1$ and $b \in L^\infty(\partial D)^2$, and let $\lambda\in\C\setminus (\overline{\R_{-}} \cup \mathcal{Z}(F_n))$. Suppose that $w^{\rm ed}_{b}$ is a solution 
to \eqref{RSed.b}. Then the $n$-Fourier modes $w^{\rm ed}_{b,n}$ and $\omega^{\rm ed}_{b,n}=({\rm rot}\,w^{\rm ed}_{b,n})e^{-in\theta}$ satisfy the following representations{\rm :}
\begin{align}
w^{\rm ed}_{b,n}
&\,=\, 
\frac{T_n (b)}{F_n(\sqrt{\lambda}; \beta)} V_{n}[K_{\mu_{n}}(\sqrt{\lambda}\,\cdot\,)]
+\frac{\mathcal{V}_{n}[b](\theta)}{r^2}\,,
\label{rep1.lem.rep.RSed.b} \\
\omega^{\rm ed}_{b,n}(r)
&\,=\, \frac{T_n (b)}{F_n(\sqrt{\lambda}; \beta)}\,K_{\mu_{n}}(\sqrt{\lambda} r)\,,
\label{rep2.lem.rep.RSed.b}
\end{align}
where the operator $T_n(b)$ and the vector field $\mathcal{V}_{n}[b](\theta)$ are 
{\color{black} defined} 
as 
\begin{align}
T_n(b) \,=\, \frac{b_{r,n}}{in} - b_{\theta,n}\,,~~~~~~
\mathcal{V}_{n}[b](\theta) 
\,=\, b_{r,n} e^{in\theta} {\bf e}_{r} 
+ \frac{b_{r,n}}{in} e^{in\theta} {\bf e}_{\theta}\,.
\label{def.lem.rep.RSed.boundary}
\end{align}
Here $\mathcal{Z}(F_n)$ is the set in \eqref{def.zeros.F} and $V_{n}[\,\cdot\,]$ is the Biot-Savart law in \eqref{def.V_n}.
\end{lemma}
%
\begin{proof} It is easy to see that $u=\frac{T_n (b)}{F_n(\sqrt{\lambda}; \beta)} V_{n}[K_{\mu_{n}}(\sqrt{\lambda}\,\cdot\,)]$ solves
\begin{equation}\label{equation.proof.lem.rep.RSed.boundary}
\left\{
\begin{aligned}
\lambda u - \Delta u + \beta U^{\bot} {\rm rot}\,u + \nabla p
&\,=\, 0\,,~~~~x \in D\,, \\
{\rm div}\,u &\,=\, 0\,,~~~~x \in D\,, \\
u_r|_{\partial D} \,=\, 0\,,~~~~~~
&u_\theta|_{\partial D} \,=\,-T_n (b) {\color{black} e^{in\theta}} \,,
\end{aligned}\right.
\end{equation}
with some pressure $p\in W^{1,1}_{{\rm loc}}(\overline{\Omega})$. The vector field $\frac{\mathcal{V}_{n}[b](\theta)}{r^2}$ corrects the boundary condition in \eqref{equation.proof.lem.rep.RSed.boundary} so that $u+\frac{\mathcal{V}_{n}[b](\theta)}{r^2}$ solves \eqref{RSed.b} replacing $b$ by $b_n$. The proof is complete.
\end{proof}
%
%
The estimates for $w^{\rm ed}_{b,n}$ and $\omega^{\rm ed}_{b,n}$ in Lemma \ref{lem.rep.RSed.b} are the main results of this subsection. We recall that $\beta_0$ is the constant in Proposition \ref{prop.est.F}.
%
\begin{theorem}\label{thm.est.RSed.b}
Let $|n|=1$, $p\in(1,\infty]$, and $q\in(1,\infty)$. Fix $\epsilon\in(0,\frac{\pi}{2})$. Then there is a positive constant $C=C(p,q,\epsilon)$ independent of $\beta$ such that the following statement holds. Let $b \in L^\infty(\partial D)^2$, $f\in L^q(D)^2$, and $\beta\in(0,\beta_0)$. Then for $\lambda\in\Sigma_{\pi-\epsilon} \cap \mathcal{B}_{e^{-\frac{1}{6\beta}}}(0)$ we have
\begin{align}
\|w^{\rm ed}_{b,n}\|_{L^p(D)} 
&\le 
\frac{C}{\beta} |\lambda|^{-\frac1p}
\|b\|_{L^\infty(\partial D)}\,,
\label{est1.thm.est.RSed.b} \\
\big\|\frac{w^{\rm ed}_{b,n}}{|x|} \big\|_{L^2(D)} 
&\le 
\frac{C}{\beta^2} 
\|b\|_{L^\infty(\partial D)}\,,
\label{est2.thm.est.RSed.b} \\
\|\omega^{\rm ed}_{b,n}\|_{L^2(D)} 
&\le 
\frac{C}{\beta}
\|b\|_{L^\infty(\partial D)}\,,
\label{est3.thm.est.RSed.b} \\
\big|\big\langle \omega^{\rm ed}_{b,n}, \frac{(w^{\rm ed}_{f,r})_n}{|x|} \big\rangle_{L^2(D)} \big|
&\le 
\frac{C}{\beta^4} |\lambda|^{-1+\frac1q} \|f \|_{L^q(D)}
\|b\|_{L^\infty(\partial D)}\,.
\label{est4.thm.est.RSed.b}
\end{align}
\end{theorem}
%
\begin{proof}
The estimates \eqref{est1.thm.est.RSed.b} and \eqref{est2.thm.est.RSed.b} follow by Propositions \ref{prop.est.F} and \ref{prop1.est.velocity.RSed.f}, while \eqref{est3.thm.est.RSed.b} follows by Proposition \ref{prop.est.F} and \eqref{est1.lem.est4.bessel} with $p=2$ in Lemma \ref{lem.est4.bessel} in Appendix \ref{app.est.bessel}. The proof for \eqref{est4.thm.est.RSed.b} is parallel to that for \eqref{est3.thm.est.vorticity.RSed.f} in Theorem \ref{thm.est.vorticity.RSed.f}. The proof is complete.
\end{proof}
%
\subsection{Resolvent estimate in region exponentially close to the origin}
\label{apriori2}
%
In this subsection we treat the problem \eqref{RS} when the resolvent parameter $\lambda$ is exponentially close to the origin. We start with the a priori estimate of the term $\big\langle ({\rm rot}\,v)_{n}, \frac{v_{r,n}}{|x|} \big\rangle_{L^2(D)}$, $|n|=1$, when $0<|\lambda|<e^{-\frac{1}{6\beta}}$, which is needed in closing the energy computation. We recall that $D$ denotes the exterior disk $\{x\in\R^2~|~|x|>1\}$, and that  
{\color{black} 
$R$ and $\gamma$
}
are defined in Assumption \ref{assumption}. Let $\beta_0$ be the constant in Proposition \ref{prop.est.F}.
%
\begin{proposition}\label{prop1.small.lambda.energy.est.resol.}
Let $|n|=1$, $q\in(1,2]$, and $f \in L^q(\Omega)^2$, and let $\lambda\in\Sigma_{\pi-\epsilon} \cap \mathcal{B}_{e^{-\frac{1}{6\beta}}}(0)$
for some $\epsilon\in(0,\frac{\pi}{2})$. Suppose that $v \in D(\mathbb{A}_V)$ is a solution to \eqref{RS}. Then we have
\begin{equation}\label{est1.prop.small.lambda.energy.est.resol.}
\begin{aligned}
\big|\big\langle ({\rm rot}\,v)_{n}, \frac{v_{r,n}}{|x|} \big\rangle_{L^2(D)} \big| 
\le 
\frac{C}{\beta^5}
|\lambda|^{-2+\frac2q}
\|f\|_{L^q(\Omega)}^2 
+ \frac{K}{\beta^5}
{\color{black} (d + \beta d^\frac12)^2}
\|\nabla v\|_{L^2(\Omega)}^2\,,
\end{aligned}
\end{equation}
as long as $\beta\in(0,\beta_0)$. The constant $C$ is independent of 
{\color{black} $d$ and $\beta$,} 
and depends on $\gamma$, $q$, and $\epsilon$, while $K$ is greater than $1$ and independent of 
{\color{black} $d$, $\beta$, and $q$}, 
and depends on $\gamma$ and $\epsilon$.
\end{proposition}
%
\begin{proof} 
In this proof we denote the function space $L^q(D)$ by $L^q$ to simplify notation. Firstly we fix a positive number $\gamma'\in(\frac12, \gamma)$, and set $F=-(R\otimes v + v\otimes R)|_{D}$ and 
{\color{black} $b= \mathcal{P}_n v|_{\partial D}$.}
It is easy to see that $F$ belongs to the function space $X_{\gamma'}(D)$ defined in \eqref{Xgamma}, and that $b\in L^\infty(\partial D)^2$. Moreover, a direct calculation and Assumption \ref{assumption} imply that
\begin{align}
\||x|^{\gamma'}F\|_{L^2} \le 
K_0 d \|\nabla v\|_{L^2(\Omega)}\,,~~~~~~~~
\|b\|_{L^\infty(\partial D)} \le 
{\color{black} K_0 d^\frac12} \|\nabla v\|_{L^2(\Omega)}\,.
\label{est1.proof.prop.small.lambda.energy.est.resol.}
\end{align}
Here $K_0$ denotes the constant which depends on $\gamma$ and is independent of {\color{black} $d$, $\beta$, and $q\in(1,2]$.} 
{\color{black} The latter estimate in \eqref{est1.proof.prop.small.lambda.energy.est.resol.} is proved as follows: the zero extension of $v\in D(\mathbb{A}_V)$ to the whole plane $\R^2$, which is denoted again by $v$, is an element of $W^{1,2}(\R^2)$. Hence we have
\begin{align*}
\|b\|_{L^\infty(\partial D)} 
& \le 
\int_0^{2\pi} \int_{1-2d}^{1} |\nabla v(r\cos\theta, r\sin\theta)| \dd r \dd \theta \\
& \le 
K_0 d^\frac12 \|\nabla v\|_{L^2(\Omega)}\,,
\end{align*}
where the H\"{o}lder inequality is applied in the last line.}
In the following we use the notations in Subsections \ref{subsec.RSed.f}--\ref{subsec.RSed.b}. Since $v|_{D}$ is a solution to the problem \eqref{RSed}, by the solution formula we have the decompositions for $v_{n}$, $|n|=1$:
\begin{align}
v_{n} &\,=\,w^{\rm ed}_{f,n} 
+ w^{\rm ed}_{{\rm div}F,n} + w^{\rm ed}_{b,n}~~~~~~{\rm in}~~D\,,
\label{decom1.proof.prop.small.lambda.energy.est.resol.} \\
({\rm rot}\,v)_{n} &\,=\,\omega^{\rm ed}_{f,n}
+ \omega^{\rm ed}_{{\rm div}F,n}
+ \omega^{\rm ed}_{b,n}~~~~~~{\rm in}~~D\,.
\label{decom2.proof.prop.small.lambda.energy.est.resol.}
\end{align}
Then, in view of \eqref{decom2.proof.prop.small.lambda.energy.est.resol.}, the assertion \eqref{est1.prop.small.lambda.energy.est.resol.} follows from estimating the next three terms: 
\begin{align*}
\big|\big\langle \omega^{\rm ed}_{f,n}, \frac{v_{r,n}}{|x|} \big\rangle_{L^2} \big|\,,~~~~~~
\big|\big\langle \omega^{\rm ed}_{{\rm div}F,n}, \frac{v_{r,n}}{|x|} \big\rangle_{L^2} \big|\,,~~~~~~
\big|\big\langle \omega^{\rm ed}_{b,n}, \frac{v_{r,n}}{|x|} \big\rangle_{L^2} \big|\,.
\end{align*}
(i) Estimate of $|\langle \omega^{\rm ed}_{f,n}, \frac{v_{r,n}}{|x|}\rangle_{L^2}|$: We fix a number $p\in(\frac{2}{\gamma'},\infty)$. Note that $p\in(q,\infty)$ holds since $\frac{2}{\gamma'}>2$. Then setting $p'= \frac{p}{p-1}\in(1,q)$ and using the H\"{o}lder inequality, we see that
\begin{align}
\big|\big\langle \omega^{\rm ed}_{f,n}, \frac{v_{r,n}}{|x|} \big\rangle_{L^2} \big| 
& \le 
\big|\big\langle \omega^{\rm ed\,(1)}_{f,n}, \frac{v_{r,n}}{|x|} \big\rangle_{L^2} \big| 
+ \big\|\frac{\omega^{\rm ed\,(2)}_{f,n}}{|x|}\big\|_{L^{p'}} \|v_{n}\|_{L^p}\,.
\label{est3.proof.prop.small.lambda.energy.est.resol.}
\end{align}
From \eqref{est1.thm.est.vorticity.RSed.f} and \eqref{est3.thm.est.vorticity.RSed.f} in Theorem \ref{thm.est.vorticity.RSed.f}, \eqref{est2.thm.est.velocity.RSed.divF} in Theorem \ref{thm.est.velocity.RSed.divF}, and \eqref{est2.thm.est.RSed.b} in Theorem \ref{thm.est.RSed.b} we observe that
\begin{align*}
\big|\big\langle \omega^{\rm ed\,(1)}_{f,n}, \frac{v_{r,n}}{|x|} \big\rangle_{L^2} \big| 
& \le
\big|\big\langle \omega^{\rm ed\,(1)}_{f,n}, 
\frac{(w^{\rm ed}_{f,r})_n}{|x|} \big\rangle_{L^2} \big| 
+ \|\omega^{\rm ed\,(1)}_{f,n}\|_{L^2}
\Big(
\big\|\frac{w^{\rm ed}_{{\rm div}F,n}}{|x|} \big\|_{L^2}
+ \big\|\frac{w^{\rm ed}_{b,n}}{|x|}\big\|_{L^2}
\Big) \\
& \le
\frac{C}{\beta^5}
|\lambda|^{-1+\frac1q}
\|f\|_{L^q}
\Big(
|\lambda|^{-1+\frac1q}
\|f\|_{L^q}
+ \big(\||x|^{\gamma'}F\|_{L^2} + \beta \|b\|_{L^\infty(\partial D)} \big)
\Big)\,.
\end{align*}
{\color{black} Then by \eqref{est1.proof.prop.small.lambda.energy.est.resol.} we find}
\begin{align}
&\big|\big\langle \omega^{\rm ed\,(1)}_{f,n}, \frac{v_{r,n}}{|x|} \big\rangle_{L^2} \big| \nonumber \\
&\le
\frac{C}{\beta^5}
|\lambda|^{-1+\frac1q}
\|f\|_{L^q}
\big(
|\lambda|^{-1+\frac1q}
\|f\|_{L^q}
+ {\color{black} (d + \beta d^\frac12)}
\|\nabla v\|_{L^2(\Omega)}
\big)\,.
\label{est5.proof.prop.small.lambda.energy.est.resol.} 
\end{align}
On the other hand, since $\frac1p+\frac1{p'}=1$ holds, by using \eqref{est2.thm.est.vorticity.RSed.f} replacing $\tilde{q}$ by $p'$ in Theorem \ref{thm.est.vorticity.RSed.f}, \eqref{est1.thm.est.velocity.RSed.f} in Theorem \ref{thm.est.velocity.RSed.f}, 
\eqref{est1.thm.est.velocity.RSed.divF} in Theorem \ref{thm.est.velocity.RSed.divF}, 
and \eqref{est1.thm.est.RSed.b} in Theorem \ref{thm.est.RSed.b}, we have
\begin{align}
& \big\|\frac{\omega^{\rm ed\,(2)}_{f,n}}{|x|}\big\|_{L^{p'}} \|v_{n}\|_{L^p} \nonumber \\
& \le
\frac{C}{\beta^3}
|\lambda|^{-1+\frac1q}
\|f\|_{L^q}
\Big(
|\lambda|^{-1+\frac1q}
\|f\|_{L^q} 
+ \big(\||x|^{\gamma'}F\|_{L^2} + \beta \|b\|_{L^\infty(\partial D)} \big)
\Big) \nonumber \\
& \le
\frac{C}{\beta^3}
|\lambda|^{-1+\frac1q}
\|f\|_{L^q}
\big(
|\lambda|^{-1+\frac1q}
\|f\|_{L^q}
+ {\color{black} (d + \beta d^\frac12)}
\|\nabla v\|_{L^2(\Omega)}
\big)\,.
\label{est6.proof.prop.small.lambda.energy.est.resol.} 
\end{align}
Then inserting \eqref{est5.proof.prop.small.lambda.energy.est.resol.} and \eqref{est6.proof.prop.small.lambda.energy.est.resol.} into \eqref{est3.proof.prop.small.lambda.energy.est.resol.} we obtain
\begin{align}
& \big|\big\langle \omega^{\rm ed}_{f,n}, \frac{v_{r,n}}{|x|} \big\rangle_{L^2} \big| \nonumber \\
& \le
\frac{C}{\beta^5}
|\lambda|^{-1+\frac1q}
\|f\|_{L^q}
\big(
|\lambda|^{-1+\frac1q}
\|f\|_{L^q}
+ {\color{black} (d + \beta d^\frac12)}
\|\nabla v\|_{L^2(\Omega)}
\big)\,.
\label{mainest1.proof.prop.small.lambda.energy.est.resol.}
\end{align}
\noindent (ii) Estimate of $|\langle \omega^{\rm ed}_{{\rm div}F,n}, \frac{v_{r,n}}{|x|} \rangle_{L^2}|$: By using the H\"{o}lder inequality we find
\begin{equation}\label{est7.proof.prop.small.lambda.energy.est.resol.}
\begin{aligned}
\big|\big\langle \omega^{\rm ed}_{{\rm div}F,n}, 
\frac{v_{r,n}}{|x|} \big\rangle_{L^2} \big| 
& \le 
\|\omega^{\rm ed}_{{\rm div}F,n}\|_{L^2}
\Big( 
\big\|\frac{w^{\rm ed}_{{\rm div}F,n}}{|x|} \big\|_{L^2}
+ \big\|\frac{w^{\rm ed}_{b,n}}{|x|}\big\|_{L^2}
\Big) \\
& \quad
+ \big|\big\langle \omega^{\rm ed\,(1)}_{{\rm div}F,n}, 
\frac{(w^{\rm ed}_{f,r})_n}{|x|} \big\rangle_{L^2} \big| 
+ \|\frac{\omega^{\rm ed\,(2)}_{{\rm div}F,n}}{|x|}\|_{L^1}
\big\|w^{\rm ed}_{f,n} \big\|_{L^\infty}\,.
\end{aligned}
\end{equation}
By Theorem \ref{thm.est.vorticity.RSed.divF}, \eqref{est2.thm.est.velocity.RSed.divF} in Theorem \ref{thm.est.velocity.RSed.divF}, 
and \eqref{est2.thm.est.RSed.b} in Theorem \ref{thm.est.RSed.b} we see that
\begin{align}
\|\omega^{\rm ed}_{{\rm div}F,n}\|_{L^2}
\Big( 
\big\|\frac{w^{\rm ed}_{{\rm div}F,n}}{|x|} \big\|_{L^2}
+ \big\|\frac{w^{\rm ed}_{b,n}}{|x|}\big\|_{L^2}
\Big)
& \le
\frac{K}{\beta^5} 
\||x|^{\gamma'}F\|_{L^2}
\big(\||x|^{\gamma'}F\|_{L^2} + \beta \|b\|_{L^\infty(\partial D)} \big)\nonumber \\
& \le
\frac{K}{\beta^{5}}
{\color{black} (d + \beta d^\frac12)^2}
\|\nabla v\|_{L^2(\Omega)}^2\,,
\label{est8.proof.prop.small.lambda.energy.est.resol.}
\end{align}
where we note that the constant $K$ depends only on $\epsilon$ and $\gamma$, and is independent of 
{\color{black} $d$ and $\beta$,} 
and, in particular, of $q\in(1,2]$. Theorem \ref{thm.est.vorticity.RSed.divF} and \eqref{est1.thm.est.velocity.RSed.f} with $p=\infty$ in Theorem \ref{thm.est.velocity.RSed.f} lead to
\begin{align}
& \big|\big\langle \omega^{\rm ed\,(1)}_{{\rm div}F,n}, 
\frac{(w^{\rm ed}_{f,r})_n}{|x|} \big\rangle_{L^2} \big| 
+ \|\frac{\omega^{\rm ed\,(2)}_{{\rm div}F,n}}{|x|}\|_{L^1}
\big\|w^{\rm ed}_{f,n} \big\|_{L^\infty} \nonumber \\
& \le
\frac{C}{\beta^5}
|\lambda|^{-1+\frac1q}
\|f\|_{L^q} \||x|^{\gamma'}F\|_{L^2} 
\le
\frac{C}{\beta^{5}}
{\color{black} d}
|\lambda|^{-1+\frac1q}
\|f\|_{L^q} 
\|\nabla v\|_{L^2(\Omega)}\,.
\label{est9.proof.prop.small.lambda.energy.est.resol.}
\end{align}
Inserting 
\eqref{est8.proof.prop.small.lambda.energy.est.resol.} and 
\eqref{est9.proof.prop.small.lambda.energy.est.resol.} into \eqref{est7.proof.prop.small.lambda.energy.est.resol.} we have
\begin{align}
& \big|\big\langle \omega^{\rm ed}_{{\rm div}F,n}, 
\frac{v_{r,n}}{|x|} \big\rangle_{L^2} \big| \nonumber \\
& \le 
\frac{C}{\beta^{5}}
{\color{black} d}
|\lambda|^{-1+\frac1q}
\|f\|_{L^q(\Omega)} 
\|\nabla v\|_{L^2(\Omega)}
+ \frac{K}{\beta^{5}}
{\color{black} (d + \beta d^\frac12)^2}
\|\nabla v\|_{L^2(\Omega)}^2\,.
\label{mainest2.proof.prop.small.lambda.energy.est.resol.} 
\end{align}
\noindent (iii) Estimate of $|\langle \omega^{\rm ed}_{b,n}, \frac{v_{r,n}}{|x|}\rangle_{L^2}|$: Using the Schwartz inequality and Theorem \ref{thm.est.RSed.b} we find
\begin{align}
&\big|\big\langle \omega^{\rm ed}_{b,n}, \frac{v_{r,n}}{|x|} \big\rangle_{L^2(D)} \big| 
\le 
\big|\big\langle \omega^{\rm ed}_{b,n}, 
\frac{(w^{\rm ed}_{f,r})_n}{|x|} \big\rangle_{L^2} \big| 
+ \|\omega^{\rm ed}_{b,n}\|_{L^2}
\Big( 
\big\|\frac{w^{\rm ed}_{{\rm div}F,n}}{|x|} \big\|_{L^2}
+ \big\|\frac{w^{\rm ed}_{b,n}}{|x|}\big\|_{L^2}
\Big) \nonumber \\
& \le 
\frac{1}{\beta^4}
\Big(
C |\lambda|^{-1+\frac1q}
\|f\|_{L^q}
\|b\|_{L^\infty(\partial D)}
+ K \|b\|_{L^\infty(\partial D)}
\big(\||x|^{\gamma'}F\|_{L^2}
+ \beta \|b\|_{L^\infty(\partial D)} \big)
\Big) \nonumber \\
& \le 
{\color{black} \frac{C}{\beta^{5}}
\beta d^\frac12
}
|\lambda|^{-1+\frac1q}
\|f\|_{L^q(\Omega)} 
\|\nabla v\|_{L^2(\Omega)}
+ \frac{K}{\beta^5}
{\color{black} (d + \beta d^\frac12)^2}
\|\nabla v\|_{L^2(\Omega)}^2\,.
\label{mainest3.proof.prop.small.lambda.energy.est.resol.} 
\end{align}
Finally we obtain the assertion 
\eqref{est1.prop.small.lambda.energy.est.resol.} by collecting 
\eqref{mainest1.proof.prop.small.lambda.energy.est.resol.}, 
\eqref{mainest2.proof.prop.small.lambda.energy.est.resol.}, and 
\eqref{mainest3.proof.prop.small.lambda.energy.est.resol.}, 
and using the Young inequality in the form
\begin{align*}
& 
\frac{C}{\beta^5}
{\color{black}
(d + \beta d^\frac12)
}
|\lambda|^{-1+\frac1q}
\|f\|_{L^q(\Omega)} 
\|\nabla v\|_{L^2(\Omega)}
\\
& \le 
\frac{C}{\beta^5}
|\lambda|^{-2+\frac2q}
\|f\|_{L^q(\Omega)}^2
+ {\frac{{\color{black} (d + \beta d^\frac12)^2}}{\beta^5} } 
\|\nabla v\|_{L^2(\Omega)}^2\,.
\end{align*}
The proof is complete.
\end{proof}
%

Now we shall establish the resolvent estimate to \eqref{RS} when $0<|\lambda|<e^{-\frac{1}{6\beta}}$, by closing the energy computation starting from Proposition \ref{prop.general.energy.est.resol.} in Subsection \ref{apriori1}. 
%
\begin{proposition}\label{prop2.small.lambda.energy.est.resol.} 
Let $\epsilon \in (0,\frac{\pi}{4})$, and let $\beta_1$ 
{\color{black} and $d_1$}, 
$\beta_0$, and $K$ be the constants respectively in Propositions \ref{prop.general.energy.est.resol.}, \ref{prop.est.F}, and \ref{prop1.small.lambda.energy.est.resol.}. 
Then the following statements hold.

\noindent {\rm (1)} Fix positive numbers $\beta_3\in(0,
{\color{black} \min\{\beta_1, d_1^\frac12, \beta_0\}}
)$ and 
{\color{black} $\mu_\ast\in(0, (64K)^{-1})$.} 
Then the set 
\begin{align}
\Sigma_{\frac{3}{4}\pi-\epsilon} \cap \mathcal{B}_{e^{-\frac{1}{6\beta}}}(0)
\label{set.thm.small.lambda.energy.est.resol.} 
\end{align}
is included in the resolvent $\rho(-\mathbb{A}_V)$ for any $\beta\in(0,\beta_3)$ 
and 
{\color{black} $d\in(0,\mu_\ast \beta^2)$. }

\noindent {\rm (2)} Let $q\in(1.2]$ and $f\in L^2_\sigma(\Omega) \cap L^q(\Omega)^2$. Then we have
\begin{equation}\label{est1.thm.small.lambda.energy.est.resol.} 
\begin{split}
\|(\lambda+{\mathbb A}_V)^{-1} f\|_{L^2(\Omega)} 
& \le 
\frac{C}{\beta^2} |\lambda|^{-\frac32+\frac1q}
\|f\|_{L^q(\Omega)}\,,~~~~
\lambda\in \Sigma_{\frac{3}{4}\pi-\epsilon} \cap \mathcal{B}_{e^{-\frac{1}{6\beta}}}(0)\,,\\
\|\nabla (\lambda+{\mathbb A}_V)^{-1} f\|_{L^2(\Omega)} 
& \le 
\frac{C}{\beta^2} |\lambda|^{-1+\frac1q}
\|f\|_{L^q(\Omega)}\,,~~~~
\lambda\in \Sigma_{\frac{3}{4}\pi-\epsilon} \cap \mathcal{B}_{e^{-\frac{1}{6\beta}}}(0)\,,
\end{split}
\end{equation}
as long as $\beta\in(0,\beta_3)$ and 
{\color{black} $d\in(0,\mu_\ast \beta^2)$.}
The constant $C$ is independent of $\beta$.
\end{proposition}
%
\begin{proof} 
(1) Let $\lambda\in\Sigma_{\frac{3}{4}\pi-\epsilon} \cap \mathcal{B}_{e^{-\frac{1}{6\beta}}}(0)$ and suppose that $v \in D(\mathbb{A}_V)$ is a solution to \eqref{RS}. 
Since 
{\color{black} $d\in(0,\mu_\ast \beta^2)$} 
ensures
{\color{black} $d\in(0,d_1)$ and $K (d + \beta d^\frac12)^2 \beta^{-4} \le \frac{1}{16}$,} 
by inserting \eqref{est1.prop.small.lambda.energy.est.resol.} in Proposition \ref{prop1.small.lambda.energy.est.resol.} into \eqref{est1.prop.general.energy.est.resol.} and \eqref{est2.prop.general.energy.est.resol.} in Proposition \ref{prop.general.energy.est.resol.}, and by combining them we find
\begin{equation}\label{est1.proof.thm.small.lambda.energy.est.resol.} 
\begin{aligned}
&\big(|{\rm Im}(\lambda)| + {\rm Re}(\lambda)\big) \| v\|_{L^2(\Omega)}^2
+ \frac14 \|\nabla v\|_{L^2(\Omega)}^2 \\
&~~~~~~
\le 
\frac{C}{\beta^4} |\lambda|^{-2+\frac2q} \|f\|_{L^q(\Omega)}^2
+ C\|f\|_{L^q(\Omega)}^{\frac{2q}{3q-2}} \|v\|_{L^2(\Omega)}^{\frac{4(q-1)}{3q-2}}\,.
\end{aligned}
\end{equation}
Then, since $\lambda\in\Sigma_{\frac{3}{4}\pi-\epsilon}$ implies that $|{\rm Im}(\lambda)| + {\rm Re}(\lambda)>c|\lambda|$ holds with some positive constant $c=c(\epsilon)$, the assertion $\Sigma_{\frac{3}{4}\pi-\epsilon} \cap \mathcal{B}_{e^{-\frac{1}{6\beta}}}(0)\subset\rho(-\mathbb{A}_V)$ follows. \\
\noindent (2) The estimate \eqref{est1.thm.small.lambda.energy.est.resol.} can be easily checked by using \eqref{est1.proof.thm.small.lambda.energy.est.resol.}. The proof is complete.
\end{proof}
%
\section{Proof of Theorem \ref{maintheorem}}
\label{sec.maintheorem}
%
In this section we prove Theorem \ref{maintheorem}. The proof is an easy consequence of Propositions \ref{prop.laege.lambda.energy.est.resol.} and \ref{prop2.small.lambda.energy.est.resol.} respectively in Subsections \ref{apriori1} and  \ref{apriori2}.
\begin{proofx}{Theorem \ref{maintheorem}} 
{\color{black} 
Firstly we note that it suffices to prove the following two estimates:
\begin{align}
\|e^{-t {\mathbb A}_V} f\|_{L^2(\Omega)} 
& \le 
\frac{C}{\beta^2} 
t^{-\frac1q+\frac12}
\|f\|_{L^q(\Omega)}\,,
~~~~~~ t>0\,,
\label{est0.proof.maintheorem} \\
\|\nabla e^{-t {\mathbb A}_V} f\|_{L^2(\Omega)} 
& \le 
\frac{C}{\beta^2} 
t^{-\frac1q}
\|f\|_{L^q(\Omega)}\,,
~~~~~~ t>0\,,
\label{est0'.proof.maintheorem}
\end{align}
for $f\in L^2_\sigma(\Omega) \cap L^q(\Omega)^2$. Then the assertions \eqref{est1.maintheorem} and \eqref{est2.maintheorem} are easy consequences from the Gagliardo-Nirenberg inequality.}
Let $\beta_2$ be the constant in Proposition \ref{prop.laege.lambda.energy.est.resol.}. {\color{black} We note that $\mathcal{S}_{\beta_2} \cap \mathcal{B}_{e^{-\frac{1}{6\beta_2}}}(0) \neq \emptyset$ holds since $12 e^{\frac{1}{e}} \beta_2^2 < 1$ follows from the condition $\beta_2\in(0,\frac{1}{12})$.} Then there is a constant $\epsilon_0 \in(\frac{\pi}{4},\frac{\pi}{2})$ such that the sector $\Sigma_{\pi-\epsilon_0}$ is included in the set $\mathcal{S}_{\beta} \cup \mathcal{B}_{e^{-\frac{1}{6\beta}}}(0)$ for any {\color{black} $\beta\in(0,\beta_2)$.} \\
Let $\beta_3$ be the constant in Proposition \ref{prop2.small.lambda.energy.est.resol.}. Fix a number $\beta_\ast\in
(0,{\color{black} \min\{\beta_2, \beta_3\}})$. Then by Propositions \ref{prop.laege.lambda.energy.est.resol.} and \ref{prop2.small.lambda.energy.est.resol.}, there is a positive constant $\mu_\ast$ such that the sector $\Sigma_{\pi-\epsilon_0}$ is included in the resolvent $\rho(-\mathbb{A}_V)$ as long as $\beta\in(0,\beta_\ast)$ and 
{\color{black} $d\in(0,\mu_\ast \beta^2)$.} 
Moreover, from the same propositions, for $q\in(1,2]$ and $f\in L^2_\sigma(\Omega) \cap L^q(\Omega)^2$ we have
\begin{equation}\label{est1.proof.maintheorem}
\begin{aligned}
\|(\lambda+{\mathbb A}_V)^{-1} f\|_{L^2(\Omega)} 
& \le 
\frac{C}{\beta^2} |\lambda|^{-\frac32+\frac1q}
\|f\|_{L^q(\Omega)}\,,~~~~
\lambda\in{\color{black} \Sigma_{\pi-\epsilon_0}}\,,\\
\|\nabla (\lambda+{\mathbb A}_V)^{-1} f\|_{L^2(\Omega)} 
& \le 
\frac{C}{\beta^2} |\lambda|^{-1+\frac1q}
\|f\|_{L^q(\Omega)}\,,~~~~
\lambda\in{\color{black} \Sigma_{\pi-\epsilon_0}}\,.
\end{aligned}
\end{equation}
In particular, the first line in \eqref{est1.proof.maintheorem} implies the estimate 
{\color{black} \eqref{est0.proof.maintheorem}} for $q=2$. Next we consider the case $q\in(1,2)$. Fix a number $\phi\in(\frac{\pi}{2},\pi-\epsilon_0)$ and take a curve $\gamma(b)=\{z\in\C~|~|{\rm arg}\,z|=\phi\,,\,|z|\ge b \} \cup \{z\in\C~|~|{\rm arg}\,z|\le\phi\,,\,|z|=b \}$, $b\in(0,1)$, oriented counterclockwise. Then the semigroup $e^{-t {\mathbb A}_V}$ admits a Dunford integral representation 
\begin{align*}
e^{-t {\mathbb A}_V} f
\,=\,
\frac{1}{2\pi i}
\int_{\gamma(b)}
e^{t \lambda }(\lambda+{\mathbb A}_V)^{-1} f  \dd \lambda\,,~~~~~~
t>0\,,
\end{align*}
for $f\in L^2_\sigma(\Omega) \cap L^q(\Omega)^2$. Then by taking the limit $b\to0$ we observe from \eqref{est1.proof.maintheorem} that
\begin{align*}
\|e^{-t {\mathbb A}_V} f\|_{L^2(\Omega)}
& \le
\frac{C}{\beta^2} \|f\|_{L^q(\Omega)} \int_{0}^{\infty}
s^{-\frac32+\frac1q} e^{\color{black}{(\cos \phi)ts}} \dd s \\
& \le
\frac{C}{\beta^2} t^{-\frac1q+\frac12}
\|f\|_{L^q(\Omega)}\,,~~~~~~
t>0\,,
\end{align*}
which shows that 
{\color{black} \eqref{est0.proof.maintheorem}}
holds {\color{black} for $q\in(1,2)$}. The estimate 
{\color{black} \eqref{est0'.proof.maintheorem}}
can be obtained in a similar manner using the Dunford integral. This completes the proof of Theorem \ref{maintheorem}.
\end{proofx}
%
\appendix
%
\section{Asymptotics of the order $\mu_n(\beta)$ for small $\beta$}
\label{app.est.mu}
This appendix is devoted to the statement of the asymptotic behavior for $\mu_n(\beta)=(n^2+in\beta)^\frac12$, ${\rm Re}(\mu_n)>0$, with $|n|=1$ when the constant $\beta\in(0,1)$ in Assumption \ref{assumption} reaches to zero. The following result is essentially proved in \cite{Ma1}.
%
\begin{lemma}[ {\rm\cite[Lemma B.1]{Ma1}}]\label{lem.est.mu} 
Let $|n|=1$. Then $\mu_n(\beta)$ satisfies the expansion
\begin{align}
{\rm Re}(\mu_n(\beta)) &\,=\, 1+\frac{\beta^2}{8} + O(\beta^4)\,, ~~~~~~
0<\beta\ll 1\,,
\label{est1.lem.est.mu} \\
{\rm Im}(\mu_n(\beta)) &\,=\, \frac{\beta}{2} + O(\beta^3)\,, ~~~~~~
0<\beta\ll 1\,.
\label{est2.lem.est.mu} 
\end{align}
\end{lemma}
%
\section{Estimates of the Modified Bessel Function}
\label{app.est.bessel}
In this appendix we collect the basic estimates for the modified Bessel functions $ K_{\mu_n}(z)$ and $ I_{\mu_n}(z)$ of the order $\mu_n=(n^2+in\beta)^\frac12$, ${\rm Re}(\mu_n)>0$, with $|n|=1$ and $\beta\in(0,1)$. We are especially interested in the $\beta$-dependence in each estimate, since our analysis in Section \ref{sec.RSed}, where the results in this appendix are applied, essentially requires the smallness of $\beta$. We denote by $\mathcal{B}_\rho(0)$ the disk in the complex plane $\C$ centered at the origin with radius $\rho>0$. 
%
\begin{lemma}\label{lem.est1.bessel} 
Let $|n|=1$, $k=0,1$, and $R\in[1,\infty)$. Fix $\epsilon\in(0,\frac{\pi}{2})$. Then there is a positive constant $C=C(R,\epsilon)$ independent of $\beta$ such that the following statements hold.

\noindent {\rm (1)} Let $z \in\Sigma_{\epsilon} \cap \mathcal{B}_R(0)$. Then $K_{\mu_n}(z)$ and $K_{\mu_n-1}(z)$ satisfy the expansions 
\begin{align}
K_{\mu_n}(z) 
&\,=\, 
\frac{\Gamma(\mu_n)}{2} \big(\frac{z}{2}\big)^{-\mu_n}
+ R^{(1)}_n(z)\,, 
\label{est1.lem.est1.bessel} \\
K_{\mu_n-1}(z) 
&\,=\,
\frac{\pi}{2\sin((\mu_n-1)\pi)}
\big( \frac{1}{\Gamma(2-\mu_n)} \big(\frac{z}{2})^{-\mu_n+1}
- \frac{1}{\Gamma(\mu_n)} \big(\frac{z}{2})^{\mu_n-1} \big) + R^{(2)}_n(z)\,.
\label{est2.lem.est1.bessel}
\end{align}
Here $\Gamma(z)$ denotes the Gamma function and the remainders $R^{(1)}_n(z)$ and $R^{(2)}_n(z)$ satisfy 
\begin{align}
|R^{(1)}_n(z)| 
&\le C|z|^{2-{\rm Re}(\mu_n)}(1+|\log|z||)\,,~~~~~~
z \in\Sigma_{\epsilon} \cap \mathcal{B}_R(0)\,,
\label{est3.lem.est1.bessel} \\
|R^{(2)}_n(z)| 
&\le C|z|^{3-{\rm Re}(\mu_n)}(1+|\log|z||)\,,~~~~~~
z \in\Sigma_{\epsilon} \cap \mathcal{B}_R(0)\,.
\label{est4.lem.est1.bessel}
\end{align}

\noindent {\rm (2)} The following estimates hold.
\begin{align}
|I_{\mu_n+k}(z)| 
&\le C |z|^{{\rm Re}(\mu_n)+k}\,,~~~~~~
z\in\Sigma_{\epsilon} \cap \mathcal{B}_R(0)\,,
\label{est5.lem.est1.bessel} \\
|K_{\mu_n-k}(z)|
&\le C |z|^{-\frac12} e^{-{\rm Re}(z)}\,,~~~~~~
z \in\Sigma_\epsilon \cap \mathcal{B}_R(0)^{{\rm c}}\,, 
\label{est6.lem.est1.bessel} \\
|I_{\mu_n+k}(z)|
&\le C |z|^{-\frac12} e^{{\rm Re}(z)}\,,~~~~~~
z \in\Sigma_\epsilon \cap \mathcal{B}_R(0)^{{\rm c}}\,.
\label{est7.lem.est1.bessel}
\end{align}
\end{lemma}
%
\begin{proof}
(1) The expansions \eqref{est1.lem.est1.bessel} and \eqref{est2.lem.est1.bessel} follow from the definition of $K_\mu(z)$ in Subsection \ref{subsec.RSed.f} combined with the well-known Euler reflection formula for the Gamma function. The estimates of the remainder terms \eqref{est3.lem.est1.bessel} and \eqref{est4.lem.est1.bessel} are also consequences of the same definition, and we omit the calculations which are easily checked. \par
\noindent (2) The estimate \eqref{est5.lem.est1.bessel} directly follows from the definition of $I_\mu(z)$ in Subsection \ref{subsec.RSed.f}. In order to prove \eqref{est6.lem.est1.bessel} and \eqref{est7.lem.est1.bessel}, let us recall the integral formulas for $K_{\mu}(z)$ and $I_{\mu}(z)$:
\begin{align*}
K_{\mu}(z) 
&\,=\, 
\frac{\pi^\frac12}{\Gamma(\mu+\frac12)} \big(\frac{z}{2}\big)^\mu
\int_0^\infty e^{-z\cosh t} (\sinh t)^{2\mu} \dd t\,,\\
I_{\mu}(z) 
&\,=\, 
\frac{1}{\pi^{\frac12}\,\Gamma(\mu+\frac12)}
\big( \frac{z}{2} \big)^{\mu}
\int_{0}^{\pi} e^{z \cos \theta} (\sin \theta)^{2\mu} \dd \theta \,,
\end{align*}
which are valid if ${\rm Re}(\mu)>-\frac12$ and $z\in\Sigma_{\frac{\pi}{2}}$ (see \cite{Abramowitz} page 376) . Then \eqref{est6.lem.est1.bessel} and \eqref{est7.lem.est1.bessel}, especially the absence of the $\beta$-singularity in the right-hand sides, can be proved  by using the identities $\cosh^2 t-\sinh^2 t=1$ and $\cos^2\theta + \sin^2\theta=1$. The proof is complete.
\end{proof} 
%

In the following we present three lemmas without proofs, since they are straightforward adaptations of Lemma \ref{lem.est1.bessel} and Lemma \ref{lem.est.mu} in Appendix \ref{app.est.mu}.
%
\begin{lemma}\label{lem.est2.bessel} 
Let $|n|=1$ and $k=0,1$, and let $\lambda\in\Sigma_{\pi-\epsilon} \cap \mathcal{B}_1(0)$ for some $\epsilon\in(0,\frac{\pi}{2})$. Then there is a constant $C>0$ independent of $\beta$ such that the following statements hold.

\noindent {\rm (1)} If $1\le \tau \le r \le {\rm Re}(\sqrt{\lambda})^{-1}$, then
\begin{align}
\bigg|\int_\tau^r s^{2-k} K_{\mu_n-k}(\sqrt{\lambda} s) \dd s\bigg|
\le 
\frac{C}{\beta^k}
|\lambda|^{-\frac{{\rm Re}(\mu_n)}{2}+\frac{k}{2}} r^{-{\rm Re}(\mu_n)+3}\,.
\label{est1.lem.est2.bessel}
\end{align}
\noindent {\rm (2)} If $1 \le \tau\le {\rm Re}(\sqrt{\lambda})^{-1} \le r$, then
\begin{align}
\bigg|\int_\tau^r s^{2-k} K_{\mu_n-k}(\sqrt{\lambda} s) \dd s\bigg|
\le 
\frac{C}{\beta^k}
\,|\lambda|^{-\frac32+\frac{k}{2}}\,.
\label{est2.lem.est2.bessel}
\end{align}
\noindent {\rm (3)} If ${\rm Re}(\sqrt{\lambda})^{-1} \le \tau\le r$, then
\begin{align}
\int_\tau^r |s^{2-k} K_{\mu_n-k}(\sqrt{\lambda} s)| \dd s
\le 
C|\lambda|^{-\frac34} \tau^{\frac32-k} e^{-{\rm Re}(\sqrt{\lambda}) \tau}\,.
\label{est3.lem.est2.bessel}
\end{align}
\noindent {\rm (4)} If $1\le \tau \le {\rm Re}(\sqrt{\lambda})^{-1}$, then
\begin{align}
\bigg|\int_\tau^\infty s^{-k} K_{\mu_n-k}(\sqrt{\lambda} s) \dd s\bigg|
\le \frac{C}{\beta^{1+k}} 
|\lambda|^{-\frac{{\rm Re}\,(\mu_n)}{2}+\frac{k}{2}} \tau^{-{\rm Re}\,(\mu_n)+1}\,.
\label{est4.lem.est2.bessel}
\end{align}
\noindent {\rm (5)} If $\tau \ge {\rm Re}(\sqrt{\lambda})^{-1}$, then
\begin{align}
\int_\tau^\infty |s^{-k} K_{\mu_n-k}(\sqrt{\lambda} s)| \dd s
\le 
C |\lambda|^{-\frac34} \tau^{-\frac12-k} e^{-{\rm Re}(\sqrt{\lambda}) \tau}\,.
\label{est5.lem.est2.bessel}
\end{align}
\end{lemma}
%
\begin{lemma}\label{lem.est3.bessel}
Let $|n|=1$ and $k=0,1$, and let $\lambda\in\Sigma_{\pi-\epsilon} \cap \mathcal{B}_1(0)$ for some $\epsilon\in(0,\frac{\pi}{2})$. Then there is a constant $C>0$ independent of $\beta$ such that the following statements hold.

\noindent {\rm (1)} If $1\le \tau \le {\rm Re}(\sqrt{\lambda})^{-1}$, then
\begin{align}
\int_1^\tau |s^{2-k} I_{\mu_n+k}(\sqrt{\lambda} s)| \dd s
\le 
C |\lambda|^{\frac{{\rm Re}(\mu_n)}{2 }+ \frac{k}{2}} 
\tau^{{\rm Re}(\mu_n)+3}\,.
\label{est1.lem.est3.bessel}
\end{align}
\noindent {\rm (2)} If $\tau \ge {\rm Re}(\sqrt{\lambda})^{-1}$, then
\begin{align}
\int_1^\tau |s^{2-k} I_{\mu_n+k}(\sqrt{\lambda} s)| \dd s
\le
C |\lambda|^{-\frac34} \tau^{\frac32-k} e^{{\rm Re}(\sqrt{\lambda}) \tau}\,.
\label{est2.lem.est3.bessel}
\end{align}
\noindent {\rm (3)} If $1\le r \le \tau \le {\rm Re}(\sqrt{\lambda})^{-1}$, then
\begin{align}
\int_r^\tau |s^{-k} I_{\mu_n+k}(\sqrt{\lambda} s)| \dd s
\le
C |\lambda|^{\frac{{\rm Re}(\mu_n)}{2}+\frac{k}{2}} \tau^{{\rm Re}(\mu_n)+1}\,.
\label{est3.lem.est3.bessel}
\end{align}
\noindent {\rm (4)} If $1\le r \le {\rm Re}(\sqrt{\lambda})^{-1} \le \tau$, then
\begin{align}
\int_r^\tau |s^{-k} I_{\mu_n+k}(\sqrt{\lambda} s)| \dd s
\le
C |\lambda|^{-\frac34} \tau^{-\frac12-k} e^{{\rm Re}(\sqrt{\lambda}) \tau}\,.
\label{est4.lem.est3.bessel}
\end{align}
\noindent {\rm (5)} If ${\rm Re}(\sqrt{\lambda})^{-1} \le r \le \tau$, then
\begin{align}
\int_r^\tau |s^{-k} I_{\mu_n+k}(\sqrt{\lambda} s)| \dd s
\le
C |\lambda|^{-\frac34} \tau^{-\frac12-k} e^{{\rm Re}(\sqrt{\lambda}) \tau}\,.
\label{est5.lem.est3.bessel}
\end{align}
\end{lemma}
%
\begin{lemma}\label{lem.est4.bessel} 
Let $|n|=1$ and $p\in(1,\infty)$, and let $\lambda\in\Sigma_{\pi-\epsilon} \cap \mathcal{B}_1(0)$ for some $\epsilon\in(0,\frac{\pi}{2})$. Then there is a constant $C>0$ independent of $\beta$ such that the following statements hold.

\noindent {\rm (1)} If additionally $p\in[2,\infty)$, then
\begin{align}
\|K_{\mu_n}(\sqrt{\lambda}\,\cdot\,)\|_{L^p((1,\infty); r\dd r)}
\le \frac{C}{(p {\rm Re}(\mu_n)-2)^\frac1p} |\lambda|^{-\frac{{\rm Re}\,(\mu_n)}{2}}\,.
\label{est1.lem.est4.bessel}
\end{align}
{\rm (2)} If $1 \le r \le {\rm Re}(\sqrt{\lambda})^{-1}$, then 
\begin{align}
\bigg( \int_r^\infty |s^{-1} K_{\mu_n}(\sqrt{\lambda} s)|^p s \dd s \bigg)^\frac1p
\le
C |\lambda|^{-\frac{{\rm Re}(\mu_n)}{2}} r^{-{\rm Re}(\mu_n)-1+\frac2p}\,.
\label{est2.lem.est4.bessel}
\end{align}
{\rm (3)} If $r \ge {\rm Re}(\sqrt{\lambda})^{-1}$, then
\begin{align}
\bigg( \int_r^\infty |s^{-1} K_{\mu_n}(\sqrt{\lambda} s)|^p s \dd s \bigg)^\frac1p
\le
C |\lambda|^{-\frac14-\frac{1}{2p}} r^{-\frac32+\frac1p} 
e^{-{\rm Re}(\sqrt{\lambda}) r}\,.
\label{est3.lem.est4.bessel}
\end{align}
{\rm (4)} If $1 \le r \le {\rm Re}(\sqrt{\lambda})^{-1}$, then 
\begin{align}
\bigg( \int_1^r |s^{-1} I_{\mu_n}(\sqrt{\lambda} s)|^p s \dd s \bigg)^\frac1p
\le
C |\lambda|^{\frac{{\rm Re}\,(\mu_n)}{2}} r^{{\rm Re}(\mu_n)-1+\frac2p}\,.
\label{est4.lem.est4.bessel}
\end{align}
{\rm (5)} If $r \ge {\rm Re}(\sqrt{\lambda})^{-1}$, then
\begin{align}
\bigg(\int_1^r |s^{-1} I_{\mu_n}(\sqrt{\lambda} s)|^p s \dd s \bigg)^\frac1p
& \le
C |\lambda|^{-\frac14-\frac1{2p}} r^{-\frac32+\frac1p} e^{{\rm Re}(\sqrt{\lambda}) r}\,.
\label{est5.lem.est4.bessel}
\end{align}
{\rm (6)} If additionally $p\in[2,\infty)$ and if $1 \le r \le {\rm Re}(\sqrt{\lambda})^{-1}$, then 
\begin{equation}\label{est6.lem.est4.bessel}
\begin{aligned}
& \bigg( \int_r^\infty |K_{\mu_n}(\sqrt{\lambda} s)|^p s \dd s \bigg)^\frac1p
+ \beta \int_r^\infty |K_{\mu_n}(\sqrt{\lambda} s)| \dd s \\
& \le
\frac{C}{\beta} |\lambda|^{-\frac{{\rm Re}(\mu_n)}{2}} r^{-{\rm Re}(\mu_n)+1}\,.
\end{aligned}
\end{equation}
{\rm (7)} If additionally $p\in[2,\infty)$ and if  $r \ge {\rm Re}(\sqrt{\lambda})^{-1}$, then
\begin{align}
\bigg( \int_r^\infty |K_{\mu_n}(\sqrt{\lambda} s)|^p s \dd s \bigg)^\frac1p
+ \int_r^\infty |K_{\mu_{n}}(\sqrt{\lambda} s)| \dd s 
& \le
C |\lambda|^{-\frac12} e^{-{\rm Re}(\sqrt{\lambda}) r}\,.
\label{est7.lem.est4.bessel}
\end{align}
{\rm (8)} If additionally $p\in[2,\infty)$ and if $1 \le r \le {\rm Re}(\sqrt{\lambda})^{-1}$, then 
\begin{align}
\bigg( \int_1^r |I_{\mu_n}(\sqrt{\lambda} s)|^p s \dd s \bigg)^\frac1p
+ \int_1^r |I_{\mu_n}(\sqrt{\lambda} s)| \dd s 
\le
C |\lambda|^{\frac{{\rm Re} (\mu_n)}{2}} r^{{\rm Re}(\mu_n)+1}\,.
\label{est8.lem.est4.bessel}
\end{align}
{\rm (9)} If additionally $p\in[2,\infty)$ and if $r \ge {\rm Re}(\sqrt{\lambda})^{-1}$, then
\begin{align}
\bigg( \int_1^r |I_{\mu_n}(\sqrt{\lambda} s)|^p s \dd s \bigg)^\frac1p
+ \int_1^r |I_{\mu_{n}}(\sqrt{\lambda} s)| \dd s 
& \le
C |\lambda|^{-\frac12} e^{{\rm Re}(\sqrt{\lambda}) r}\,.
\label{est9.lem.est4.bessel}
\end{align}
\end{lemma}
%
%

\section{Proof of Proposition \ref{prop.est.F}}
\label{app.proof.prop.est.F}

Proposition \ref{prop.est.F} is a direct consequence of the next lemma. Let us recall that $\mathcal{B}_\rho(0)$ denotes the disk in the complex plane $\C$ centered at the origin with radius $\rho>0$. 
%
\begin{lemma}\label{lem.est.F} 
Let $|n|=1$. Then for any $\epsilon \in (0,\frac{\pi}{2})$ there is a positive constant $\beta_0=\beta_0(\epsilon)$ depending only on $\epsilon$ such that as long as $\beta\in (0,\beta_0)$ and $\lambda \in \Sigma_{\pi-\epsilon} \cap \mathcal{B}_{\beta^4}(0)$ we have
\begin{align}
|F_n(\sqrt{\lambda};\beta)| 
\ge 
\frac{C}{\beta} |\lambda|^{-\frac{{\rm Re}(\mu_n)}{2}} 
\min\{1\,,-\beta^2 \log {\color{black} |\lambda|} \}\,,
\label{est1.lem.est.F}
\end{align}
where $F_n(\sqrt{\lambda};\beta)$ is the function in \eqref{def.F} and the constant $C$ depends only on $\epsilon$.
\end{lemma}
%
\begin{proof}
The proof is carried out with the similar spirit as in \cite[Proposition 3.34]{Ma1}, where the nonexistence of zeros of $F_n(\sqrt{\lambda};\beta)$ in $\lambda\in
\mathcal{B}_{\beta^4}(0)$ is proved. However, its proof is based on a contradiction argument, and quantitative estimates are not explicitly stated. Hence here we give the lower bound estimate of $|F_n(\sqrt{\lambda};\beta)|$ for completeness.

Let $\lambda \in \Sigma_{\pi-\epsilon} \cap \mathcal{B}_{\frac12}(0)$ and set $\zeta_n =\zeta_n(\beta) =\mu_n(\beta)-1$. Then, by combining Lemmas 3.31--3.33 and Corollary A.8 in \cite{Ma1}, we observe that the next expansion holds:
\begin{align}
\zeta_n F_n(\sqrt{\lambda};\beta) 
\,=\, 
\frac{\Gamma(1+\zeta_n)}{\sqrt{\lambda}} \big(\frac{\sqrt{\lambda}}{2} \big)^{-\zeta_n}\,\bigg(1 - \big(e^{\gamma(\zeta_n)} \frac{\sqrt{\lambda}}{2} \big)^{\zeta_n} 
+ R_n(\lambda) \bigg)\,,
\label{exp.proof.lem.est.F}
\end{align}
for sufficiently small $\beta$ depending on $\epsilon\in(0,\frac{\pi}{2})$. Here the function $\gamma(\zeta_n)$ have the expansion
\begin{align}
\gamma(\zeta_n) \,=\, \gamma + O(|\zeta_n|)~~~~{\rm as}~~~~|\zeta_n| \rightarrow 0\,,
\label{est1.proof.lem.est.F}
\end{align}
where $\gamma$ denotes the Euler constant $\gamma = 0.5772\cdots$. The remainder $R_n$ in \eqref{exp.proof.lem.est.F} satisfies
\begin{align}
|R_n(\lambda)| \le C_1 |\lambda|^{\frac{{\rm Re}(\mu_n)}{2}}\,,~~~~~~
\lambda \in \Sigma_{\pi-\epsilon} \cap \mathcal{B}_\frac12(0)\,,
\label{est2.proof.lem.est.F}
\end{align}
with a constant $C_1=C_1(\epsilon)$ independent of small $\beta$. To simplify notation we set
\begin{align}
z\,=\, \sqrt{\lambda}\,, ~~~~~~~~
\tilde{z}\,=\, e^{\gamma(\zeta_n)} \frac{\sqrt{\lambda}}{2}\,,~~~~~~~~
\theta(\tilde{z}) \,=\, {\rm arg}\,\tilde{z}\,.
\label{def1.proof.lem.est.F}
\end{align}
If $\beta$ is sufficiently small, then we see from \eqref{est1.proof.lem.est.F} and \eqref{def1.proof.lem.est.F} that
\begin{align}
\frac12 \le \big| \frac{\tilde{z}}{z} \big| \le 1\,,~~~~~~
|\theta(\tilde{z})| \le \frac{\pi}{2} - \frac{\epsilon}{4}\,,~~~~~~~~
\lambda \in \Sigma_{\pi-\epsilon} \cap \mathcal{B}_{\frac12}(0)\,.
\label{est3.proof.lem.est.F}
\end{align}
Now we set
\begin{align}
h(\tilde{z}, \zeta_n) 
&\,=\, {\rm Re}(\zeta_n) \log|\tilde{z}| - {\rm Im}(\zeta_n) \theta(\tilde{z})\,,
\label{def2.proof.lem.est.F} \\
\Omega(\tilde{z}, \zeta_n) 
&\,=\, 
{\rm Re}(\zeta_n) \theta(\tilde{z}) + {\rm Im}(\zeta_n) \log|\tilde{z}|
\nonumber \\
&\,=\,
\big( {\rm Re}(\zeta_n) + \frac{{\rm Im}(\zeta_n)^2}{{\rm Re}(\zeta_n)} \big)
\theta(\tilde{z})
+ \frac{{\rm Im}(\zeta_n)}{{\rm Re}(\zeta_n)} h(\tilde{z}, \zeta_n)\,.
\label{def3.proof.lem.est.F}
\end{align}
Then it is easy to see that 
\begin{align}
1-\tilde{z}^{\zeta_n} 
\,=\, 1 - e^{h(\tilde{z}, \zeta_n)} e^{i \Omega(\tilde{z}, \zeta_n)}\,.
\label{est4.proof.lem.est.F}
\end{align}

In the following we show the lower bound estimate of $|1-\tilde{z}^{\zeta_n}|$. Firstly let us take a small positive constant $\kappa=\kappa(\epsilon)\ll 1$ so that
\begin{align}
\big( {\rm Re}(\zeta_n) + (1+\kappa) \frac{{\rm Im}(\zeta_n)^2}{{\rm Re}(\zeta_n)} \big)
\big( \frac{\pi}{2} - \frac{\epsilon}{4} \big)
< \pi
\label{est5.proof.lem.est.F}
\end{align}
holds. The existence of such $\kappa$ is verified by using Lemma \ref{lem.est.mu} in Appendix \ref{app.est.mu} if $\beta$ is sufficiently small depending on $\epsilon$. Note that the smallness of $\kappa$ depends only on $\epsilon$. \\
\noindent (i) Case $|h(\tilde{z}, \zeta_n)| \le \kappa |{\rm Im}(\zeta_n)| |\theta(\tilde{z})|$: In this case, \eqref{est3.proof.lem.est.F}, \eqref{def3.proof.lem.est.F}, and \eqref{est5.proof.lem.est.F} ensure that 
\begin{align}
|\Omega(\tilde{z}, \zeta_n)| < \pi\,,
\label{est5'.proof.lem.est.F}
\end{align}
and thus that $e^{i \Omega(\tilde{z}, \zeta_n)}$ is close to $1$ if and only if $\Omega(\tilde{z}, \zeta_n)$ is close to $0$. From \eqref{def2.proof.lem.est.F} we have
\begin{align*}
-{\rm Re}(\zeta_n) \log|\tilde{z}|
\le (1+\kappa) |{\rm Im}(\zeta_n)| |\theta(\tilde{z})| \,,
\end{align*}
which leads to, for sufficiently small $\beta$, 
\begin{align*}
|\theta(\tilde{z})| 
\ge - \frac{1}{1+\kappa} \frac{{\rm Re}(\zeta_n)}{|{\rm Im}(\zeta_n)|} \log|\tilde{z}| 
\ge - \frac{\beta}{2} \log|\tilde{z}|\,,
\end{align*}
where $\frac{{\rm Re}(\zeta_n)}{|{\rm Im}(\zeta_n)|} = \frac{\beta}{4} + O(\beta^3)$ is applied in Lemma \ref{lem.est.mu}. Then from \eqref{def3.proof.lem.est.F} we have
\begin{align*}
|\Omega(\tilde{z}, \zeta_n)|
& \ge
\big( {\rm Re}(\zeta_n) 
+ (1-\kappa) \frac{{\rm Im}(\zeta_n)^2}{{\rm Re}(\zeta_n)} \big)
|\theta(\tilde{z})| 
\ge
-\beta \log|\tilde{z}|\,,
\end{align*}
if $\beta$ is small enough. On the other hand, it is straightforward to see that
\begin{align*}
|1-\tilde{z}^{\zeta_n} |
\ge
\max\{|1 - e^{h(\tilde{z}, \zeta_n)} \cos\Omega(\tilde{z}, \zeta_n)|\,,~
e^{h(\tilde{z}, \zeta_n)} |\sin\Omega(\tilde{z}, \zeta_n)|
\}\,.
\end{align*}
Since $e^{h(\tilde{z}, \zeta_n)}\in[\frac12,\frac32]$, $|\sin x| \ge \frac{2|x|}{\pi}$ on $|x|\in[0,\frac{\pi}{2}]$, and $1>\frac{|\Omega(\tilde{z}, \zeta_n)|}{\pi}$ by \eqref{est5'.proof.lem.est.F}, we have
\begin{align}
|1-\tilde{z}^{\zeta_n} |
\ge \min\{1\,,~\frac{|\Omega(\tilde{z}, \zeta_n)|}{\pi}\} 
\ge
-\frac{\beta}{\pi} \log|\tilde{z}|\,.
\label{est6.proof.lem.est.F}
\end{align}
(ii) Case $|h(\tilde{z}, \zeta_n)| > \kappa |{\rm Im}(\zeta_n)| |\theta(\tilde{z})|$: When $|\theta(\tilde{z})| > -\frac12 \frac{{\rm Re}(\zeta_n)}{|{\rm Im}(\zeta_n)|} \log |\tilde{z}|$, we have
\begin{align*}
|h(\tilde{z}, \zeta_n)| 
\ge 
- \frac{\kappa \beta^2}{2} \log|\tilde{z}|\,.
\end{align*}
On the other hand, when $|\theta(\tilde{z})| \le -\frac12 \frac{{\rm Re}(\zeta_n)}{|{\rm Im}(\zeta_n)|} \log |\tilde{z}|$, \eqref{def2.proof.lem.est.F} implies that
\begin{align*}
|h(\tilde{z}, \zeta_n)| 
\ge -\frac12 {\rm Re}(\zeta_n) \log|\tilde{z}|
\ge -\frac{\beta^2}{2} \log|\tilde{z}|\,.
\end{align*}
Thus in the case (ii), since $|1 - \tilde{z}^{\zeta_n} | \ge \big|1 - |\tilde{z}^{\zeta_n}|\big| = |1 - e^{h(\tilde{z}, \zeta_n)}|$, we observe that
\begin{align}
|1 - \tilde{z}^{\zeta_n} |
\ge
\min\{1\,,~|h(\tilde{z}, \zeta_n)|\}
\ge
\min\{1\,,-\frac{\kappa \beta^2}{2} \log|\tilde{z}|\}\,.
\label{est7.proof.lem.est.F}
\end{align}
Hence, by collecting \eqref{est3.proof.lem.est.F}, \eqref{est6.proof.lem.est.F}, and \eqref{est7.proof.lem.est.F}, we have the next lower estimate of $|1-\tilde{z}^{\zeta_n}|$:
\begin{align}
|1 - \tilde{z}^{\zeta_n}|
\ge \frac{\kappa}4 \min\{1\,,-\beta^2 \log|z|\}\,.
\label{est8.proof.lem.est.F}
\end{align}
Finally by inserting \eqref{est2.proof.lem.est.F} and \eqref{est8.proof.lem.est.F} into \eqref{exp.proof.lem.est.F} we obtain
\begin{align*}
|\zeta_n F_n(\sqrt{\lambda};\beta)| 
& \ge
C |\lambda|^{-\frac{{\rm Re}(\mu_n)}{2}} 
\big( \kappa \min\{1\,,-\beta^2 \log|z|\}
- C_1 |\lambda|^{\frac{{\rm Re}(\mu_n)}{2}} \big)\,,
\end{align*}
which implies the assertion \eqref{est1.lem.est.F} if $\lambda \in \Sigma_{\pi-\epsilon} \cap \mathcal{B}_{\beta^4}(0)$ and $\beta$ is sufficiently small depending on $\epsilon$. The proof is complete.
\end{proof}

\noindent 
{\bf Acknowledgements}\, The author would like to express sincere thanks to Professor Yasunori Maekawa for valuable discussions and Professor Isabelle Gallagher for helpful comments. This work is partially supported by the Grant-in-Aid for JSPS Fellows 17J00636.

\end{document}